\documentclass{article}[11pt]
\usepackage{amsmath,amssymb,stmaryrd,enumerate,bbm,latexsym,xcolor,theorem}
\usepackage[a4paper]{geometry}
\usepackage[T1]{fontenc}
\usepackage[english]{babel}
\usepackage{mathrsfs}
\usepackage{amstext}
\usepackage{amssymb}
\usepackage{color}
\usepackage{hyperref}
\usepackage{dsfont}
\usepackage{verbatim} 
\usepackage{stmaryrd}

\newcommand{\assign}{:=}
\newcommand{\backassign}{=:}
\newcommand{\comma}{{,}}
\newcommand{\mathd}{\mathrm{d}}
\newcommand{\nobracket}{}
\newcommand{\tmaffiliation}[1]{\\ #1}
\newcommand{\tmcolor}[2]{{\color{#1}{#2}}}
\newcommand{\tmemail}[1]{\\ \textit{Email:} \texttt{#1}}
\newcommand{\tmop}[1]{\ensuremath{\operatorname{#1}}}
\newcommand{\tmscript}[1]{\text{\scriptsize{$#1$}}}
\newcommand{\tmtextbf}[1]{{\bfseries{#1}}}
\newcommand{\tmtextit}[1]{{\itshape{#1}}}
\newenvironment{enumeratenumeric}{\begin{enumerate}[1.] }{\end{enumerate}}
\newenvironment{enumerateroman}{\begin{enumerate}[i.] }{\end{enumerate}}
\numberwithin{equation}{section}
\numberwithin{figure}{section}
\newenvironment{proof}{\noindent\textbf{Proof.}}{\hspace*{\fill}$\Box$\medskip}
\theoremstyle{plain}
\newtheorem{theorem}{Theorem}[section]
\newtheorem{corollary}[theorem]{Corollary}
\newtheorem{definition}[theorem]{Definition}
{\theorembodyfont{\rmfamily}\newtheorem{example}[theorem]{Example}}
\newtheorem{lemma}[theorem]{Lemma}
\newtheorem{proposition}[theorem]{Proposition}
{\theorembodyfont{\rmfamily}\newtheorem{remark}[theorem]{Remark}}
\newcommand{\tmkeywords}{\textbf{Keywords:} }
\begin{document}

\title{Nonlinear Young differential equations: a review}

\author{
  L.~Galeati
  \tmaffiliation{Institute of Applied Mathematics \&\\
  Hausdorff Center for Mathematics\\
  University of Bonn\\
  Germany}
  \tmemail{lucio.galeati@iam.uni-bonn.de}
}

\maketitle

\begin{abstract}
  Nonlinear Young integrals have been first introduced
  in~{\cite{catelliergubinelli}} and provide a natural generalisation of
  classical Young ones, but also a versatile tool in the pathwise study of
  regularisation by noise phenomena. We present here a self-contained account
  of the theory, focusing on wellposedness results for abstract nonlinear
  Young differential equations, together with some new extensions; convergence
  of numerical schemes and nonlinear Young PDEs are also treated. Most results
  are presented for general (possibly infinite dimensional) Banach spaces and
  without using compactness assumptions, unless explicitly stated.
  
  \
  
  {\noindent}\tmtextbf{MSC(2020):} Primary: 60L20. Secondary: 60L50, 34A08.
\end{abstract}

\tmkeywords{Nonlinear Young integral, Young differential equations, numerical
schemes, flow property, transport equations, parabolic Young equations.}

{\tableofcontents}

\section{Introduction}\label{sec1}

The main goal of this article is to solve and study differential equations of
the form
\begin{equation}
  x_t = x_0 + \int_0^t A (\mathd s, x_s) \label{intro eq1}
\end{equation}
where $x$ is an $\alpha$-H\"{o}lder continuous path taking values in a Banach
space $V$ and $A : [0, T] \times V \rightarrow V$ is a vector field with
suitable space-time H\"{o}lder regularity. If $A$ is sufficiently smooth in
time, then $A (\mathd s, x_s)$ can be interpreted as $\partial_t A (s, x_s)
\mathd s$, so that~\eqref{intro eq1} can be regarded as an ODE in integral
form; here however we are interested in the case $\partial_t A$ does not
exist, so that~{\eqref{intro eq1}} does not admit a classical interpretation.

\

In the case $A (t, z) = f (z) y_t$, where $y$ is an $U$-valued
$\alpha$-H\"{o}lder continuous path and $f$ maps $V$ into the space of linear
maps from $U$ to $V$, equation~{\eqref{intro eq1}} can be rewritten as
\begin{equation}
  x_t = x_0 + \int_0^t f (x_s) \mathd y_s \label{intro eq2}
\end{equation}
which can be regarded as a rough differential equation driven by a signal $y$.

\

In the regime $\alpha \in (1 / 2, 1]$, for sufficiently regular $f$,
equation~{\eqref{intro eq2}} can be rigorously interpreted by means of Young
integrals, introduced in~{\cite{young}}; wellposedness of Young differential
equations (YDEs) was first studied in~{\cite{lyons1}}. After that, several
alternative approaches to~{\eqref{intro eq2}} have been developed, either by
means of fractional calculus~{\cite{zahle}} or numerical
schemes~{\cite{davie}}; see also the review~{\cite{lejay}} for a
self-contained exposition of the main results for YDEs and the
paper~{\cite{cong}} for some recent developments. YDEs have found several
applications in the study of SDEs driven by fractional Brownian motion (fBm)
of parameter $H > 1 / 2$, see for instance~{\cite{rascanu}}.

\

Although equation~{\eqref{intro eq1}} may be seen as a natural generalization
of~{\eqref{intro eq2}}, its development is much more recent. Nonlinear Young
integrals of the form
\[ \int_0^t A (\mathd s, x_s) \]
were first defined in~{\cite{catelliergubinelli}} in applications to
additively perturbed ODEs and subsequently rediscovered in~{\cite{hu}}, where
they were employed to give a pathwise interpretation to Feynman-Kac formulas
and SPDEs with random coefficients.

\

In this paper we will consider exclusively the time regularity regime $\alpha
> 1 / 2$, also known as the Young (or or level-1 rough path) regime. However it is now
well known, since the pioneering work of Lyons~{\cite{lyons2}}, that it is
possible to give meaning to equation~{\eqref{intro eq2}} even in the case
$\alpha \leq 1 / 2$ by means of the theory of rough paths, see the
monographs~{\cite{frizvictoir}},~{\cite{frizhairer}} for a detailed account on
the topic. An analogue extesion of~{\eqref{intro eq1}} to the case of
\tmtextit{nonlinear rough paths} has been recently achieved
in~{\cite{coghinilssen}},~{\cite{nualart}}; so far however it hasn't found the
same variety of applications, discussed below, as the nonlinear Young case.
Let us finally mention that all of the above can also be seen as subcases of
the theory of rough flows developed
in~{\cite{bailleul2015}},~{\cite{bailleulriedel}}.

\

Nonlinear YDEs of the form~{\eqref{intro eq1}} mostly present direct analogue
results to their classical counterpart~{\eqref{intro eq2}}, but their
importance and the main motivation for this work lies in their
\tmtextit{versatility}. Indeed, many differential systems which a priori do
not present such structure, may be \tmtextit{recast} as nonlinear YDEs; this
allows to give them meaning in situations where classical theory breaks down.

\

This methodology seems seems particularly effective in applications to
\tmtextit{regularization by noise} phenomena; to clarify what we mean, let us
illustrate the following example, taken
from~{\cite{choukgubinelli1}},~{\cite{choukgubinelli2}}. In these works the
authors study abstract modulated PDEs of the form
\begin{equation}
  \mathd \varphi_t = A \varphi \dot{w}_t + \mathcal{N} (\varphi_t) \mathd t
  \label{intro modulated pde}
\end{equation}
where $w : [0, T] \rightarrow \mathbb{R}$ is a continuous (possibly very
rough) path, $A$ is the generator of a group $\{ e^{t A} \}_{t \in
\mathbb{R}}$ and $\mathcal{N}$ is a nonlinear functional, possibly ill-posed
in low regularity spaces. Formally, setting $\psi_t \assign e^{- w_t A}
\varphi_t$, $\psi$ would solve
\[ \psi_t = \psi_0 + \int_0^t e^{- w_s A} \mathcal{N} (e^{w_s A} \psi_s)
   \mathd s, \]
which can be regarded as an instance of~{\eqref{intro eq1}} for the choice
\begin{equation}
  A (t, z) = \int_0^t e^{- w_s A} \mathcal{N} (e^{w_s A} z) \mathd s.
  \label{intro field}
\end{equation}
Under suitable assumption, even if $w$ is not smooth (actually
\tmtextit{exactly} because it is rough, as measured by its
$\rho$-irregularity), it is possibile to rigorously define the field $A$, even
if the integral appearing on the r.h.s. of~{\eqref{intro field}} is not
meaningful in the Lebesgue sense. As a consequence, the transformation of the
state space given by $\varphi \mapsto \psi$ allows to interpret the original
PDE~{\eqref{intro modulated pde}} as a suitable nonlinear YDE; the general
abstract theory presented here can then be applied, immediately yielding
wellposedness results.

\

A similar reasoning holds for additively perturbed ODEs of the form
\[ x_t = x_0 + \int_0^t b (x_s) \mathd s + w_t \]
which were first considered in~{\cite{catelliergubinelli}}, in which case the
transformation amounts to $x \mapsto \theta \assign x - w$. This case has
recently received a lot of attention and developed into a general theory of
pathwise regularisation by noise for ODEs and SDEs,
see~{\cite{galeatigubinelli1}}, {\cite{galeatigubinelli2}},
{\cite{galeatiharang}}, {\cite{perkowski}}, {\cite{harang}} and on a related
note~{\cite{harangmayorcas}}.

\

Motivated by the above discussion, we collect here several results for
abstract nonlinear YDEs which have appeared in the above references, together
with some new extensions; they provide general criteria for existence,
uniqueness and stability of solutions to~{\eqref{intro eq1}}, as well as
convergence of numerical schemes and differentiability of the flow. This work
is deeply inspired by the review~{\cite{lejay}}, of which it can be partially
regarded as an extension; all the theory is developed in (possibly infinte
dimensional) Banach spaces and relies systematically on the use of the sewing
lemma, a by now standard feature of the rough path framework. We hope however
that the also reader already acquainted with RDEs can find the paper of
interest due to later Sections~\ref{sec5}-\ref{sec7}, containing less standard
results and applications to Young PDEs.

\

\tmtextbf{Structure of the paper.} In Section~\ref{sec2}, the nonlinear Young
integral is constructed and its main properties are established.
Section~\ref{sec3} is devoted to criteria for existence, uniqueness, stability
and convergence of numerical schemes for nonlinear YDEs, Sections~\ref{sec3.4}
and~\ref{sec3.5} focusing on several variants of the main case.
Section~\ref{sec4} deals continuity of the solutions with respect to the data
of the problem, giving conditions for the existence of a flow and
differentiability of the It\^{o} map. The results from Section~\ref{sec3.3}
are revisited in Section~\ref{sec5}, where more refined criteria for
uniqueness of solutions are given; we label them as ``conditional uniqueness''
results, as they require additional assumptions which are often met in
probabilistic applications, but are difficult to check by purely analytic
arguments. Sections~\ref{sec6} and~\ref{sec7} deal respectively with Young
transport and parabolic type of PDEs. We chose to collect in the Appendix some
useful tools and further topics.

\

\tmtextbf{Notation.} Here is a list of the most relevant and frequently used
notations and conventions:
\begin{itemize}
  \item We write $a \lesssim b$ if $a \leqslant C b$ for a suitable constant,
  $a \lesssim_x b_{}$ to stress the dependence $C = C (x)$.
  
  \item We will always work on a finite time interval $[0, T]$; the Banach
  spaces $V$, $W$ appearing might be infinite dimensional but will be always
  assumed separable for simplicity.
  
  \item Given a Banach space $(E, \| \cdot \|_E)$, we set $C^0_t E = C ([0, T]
  ; E)$ endowed with supremum norm
  \[ \| f \|_{\infty} = \sup_{t \in [0, T]} \| f_t \|_E \quad \forall \, f \in
     C^0_t E \]
  where $f_t : = f (t)$ and we adopt the incremental notation $f_{s, t}
  \assign f_t - f_s$. Similarly, for any $\alpha \in (0, 1)$ we set
  $C^{\alpha}_t E = C^{\alpha} ([0, T] ; E)$ be the space of
  $\alpha$-H\"{o}lder continuous functions with norm
  \[ \llbracket f \rrbracket_{\alpha} \assign \sup_{\tmscript{\begin{array}{c}
       0 \leqslant s < t \leqslant T
     \end{array}}} \frac{\| f_{s, t} \|_E}{| t - s |^{\alpha}}, \qquad \| f
     \|_{\alpha} \assign \| f \|_{\infty} + \llbracket f \rrbracket_{\alpha} .
  \]
  \item The above notation will be applied to several choice of $E$ such as
  $C^{\alpha}_t V$, $C^{\alpha}_t \mathbb{R}^d$ but also $C^{\alpha}_t
  C^{\beta, \lambda}_{V, W}$ or $C^{\alpha}_t C^{\beta}_{V, W, \tmop{loc}}$,
  for which we refer to Definitions~\ref{sec2 defn field1} and~\ref{sec2 defn
  field2}.
  
  \item We denote by $\mathcal{L} (V ; W)$ the set of all linear bounded
  operators from $V$ to $W$, $L (V) = L (V ; V)$.
  
  \item Whenever we will refer to differentiability this must be understood in
  the sense of Frech{\'e}t, unless specified otherwise; given a map $F : V
  \rightarrow W$ we regard its Frech{\'e}t differential $D^k F$ of order $k$
  as a map from $V$ to $\mathcal{L}^k (V ; W)$, the set of bounded $k$-linear
  forms from $V^k$ to $W$. We will use indifferently $D F (x, y) = D F (x)
  (y)$ for the differential at point $x$ evaluated along the direction $y$.
  
  \item Given a linear unbounded operator $A$, $\tmop{Dom} (A)$ denotes its
  domain, $\tmop{rg} (A)$ its range.
  
  \item As a rule of thumb, whenever $J (\Gamma)$ appears, it denotes the
  sewing of $\Gamma : \Delta_2 \rightarrow E$; we refer to
  Section~\ref{sec2.1} for more details on the sewing map. Similarly, in
  proofs based on a Banach fixed point argument, $I$ will denote the map whose
  constractivity must be established.
  
  \item As a rule of thumb, we will use $C_i$, $i \in \mathbb{N}$ for the
  constants appearing in the main statements and $\kappa_i$ for those only
  appearing inside the proofs; the numbering restarts at each statement and is
  only meant to distinguish the dependence of the constants from relevant
  parameters.
\end{itemize}
\section{The nonlinear Young integral}\label{sec2}

This section is devoted to the construction of nonlinear Young integrals and
nonlinear Young calculus more in general, as a preliminary step to the study
of nonlinear Young differential equations which will be developed in the next
section. We follow the modern rough path approach to abstract integration,
based on the sewing lemma as developed in~{\cite{gubinelli}}
and~{\cite{feyel}}, which is recalled first.

\subsection{Preliminaries}\label{sec2.1}

This subsections contains an exposition of the sewing lemma and the definition
of the joint space-time H\"{o}lder continous drifts $A$ we will work with; the
reader already acquainted with this concepts may skip it.

\

Given a finite interval $[0, T]$, consider the $n$\mbox{-}simplex $\Delta_n :
= \{ (t_1, \ldots, t_n) : 0 \leqslant t_1 \leqslant \ldots \leqslant t_n
\leqslant T \}$. Let $V$ be a Banach space, for any $\Gamma : \Delta_2
\rightarrow V$ we define $\delta \Gamma : \Delta_3 \rightarrow V$ by
\[ \delta \Gamma_{s, u, t} : = \, \Gamma_{s, t} - \Gamma_{s, u} - \Gamma_{u,
   t} . \]
We say that $\Gamma \in C^{\alpha, \beta}_2 ([0, T] ; V) = C^{\alpha, \beta}_2
V$ if $\Gamma_{t, t} = 0$ for all $t \in [0, T]$ and $\| \Gamma \|_{\alpha,
\beta} < \infty$, where
\[ \| \Gamma \|_{\alpha} \assign \sup_{s < t} \frac{\| \Gamma_{s, t} \|_V}{| t
   - s |^{\alpha}}, \quad \left\| \delta \, \Gamma \right\|_{\beta} \assign
   \sup_{s < u < t} \frac{\left\| \delta \, \Gamma_{s, u, t} \right\|_V}{| t -
   s |^{\beta}}, \quad \| \Gamma \|_{\alpha, \beta} \assign \| \Gamma
   \|_{\alpha} + \left\| \delta \, \Gamma \right\|_{\beta} . \]
For a map $f : [0, T] \rightarrow V$, we still denote by $f_{s, t}$ the
increment $f_t - f_s$.

\begin{lemma}[Sewing lemma]
  \label{sec2 sewing lemma}Let $\alpha$, $\beta$ be such that $0 < \alpha < 1
  < \beta$. For any $\Gamma \in C^{\alpha, \beta}_2 V$ there exists a unique
  map $\mathcal{J} (\Gamma) \in C^{\alpha}_t V$ such that $\mathcal{J}
  (\Gamma)_0 = 0$ and
  \begin{equation}
    \| \mathcal{J} (\Gamma)_{s, t} - \Gamma_{s, t} \|_V \leqslant C_1 \, \|
    \delta \Gamma \|_{\beta}  | t - s |^{\beta} \label{sec2 sewing property 1}
  \end{equation}
  where the constant $C_1$ can be taken as $C_1 = (1 - 2^{\beta - 1})^{- 1}$.
  Thus the sewing map $\mathcal{J} : C^{\alpha, \beta}_2 V \rightarrow
  C^{\alpha}_t V$ is linear and bounded and there exists $C_2 = C_2 (\alpha,
  \beta, T)$ such that
  \begin{equation}
    \left\| \mathcal{J} \, (\Gamma) \right\|_{\alpha} \leqslant C_2  \| \Gamma
    \|_{\alpha, \beta} . \label{sec2 sewing property 2}
  \end{equation}
  For a given $\Gamma$, $\mathcal{J} \, (\Gamma)$ is characterized as the
  unique limit of Riemann\mbox{-}Stjeltes sums: for any $t > 0$
  \[ \mathcal{J} \, (\Gamma)_t = \lim_{| \Pi | \rightarrow 0} \sum_i
     \Gamma_{t_i, t_{i + 1}} . \]
  The notation above means that for any sequence of partitions $\Pi_n = \{ 0 =
  t_0 < t_1 < \ldots < t_{k_n} = t \}$ with mesh $| \Pi_n | = \sup_{i = 1,
  \ldots, k_n} | t_i - t_{i - 1} | \rightarrow 0$ as $n \rightarrow \infty$,
  it holds
  \[ \mathcal{J} \, (\Gamma)_t = \lim_{n \rightarrow \infty} \sum_{i = 0}^{k_n
     - 1} \Gamma_{t_i, t_{i + 1}} . \]
\end{lemma}

For a proof, see Lemma~4.2 from~{\cite{frizhairer}}.

\begin{remark}
  \label{sec2 remark sewing}Let us stress two important aspects of the above
  result. The first one is that all the estimates do not depend on the Banach
  space $V$ considered; the second one is that, even when the map $\mathcal{J}
  (\Gamma)$ is already known to exist, property~{\eqref{sec2 sewing property
  1}} still gives non trivial estimates on its behaviour. In particular, if $f
  \in C^{\alpha}_t V$ is a function such that $\| \Gamma_{s, t} - f_{s, t}
  \|_V \leqslant \kappa | t - s |^{\alpha}$ for an unknown constant $\kappa$,
  then by the sewing lemma we can deduce that $f = \mathcal{J} (\Gamma)$ and
  that $\kappa$ can be taken as $C_1 \, \| \delta \Gamma \|_{\beta}$.
\end{remark}

Next we need to introduce suitable classes of H\"{o}lder continuous maps on
Banach spaces.

\begin{definition}
  \label{sec2 defn field1}Let $V, W$ Banach spaces, $f \in C (V ; W)$, $\beta
  \in (0, 1)$. We say that $f$ is locally $\beta$-H\"{o}lder continuous and
  write $f \in C^{\beta}_{V, W, \tmop{loc}}$ if for any $R > 0$ the following
  quantities are finite:
  \[ \llbracket f \rrbracket_{\beta, R} \assign
     \sup_{\tmscript{\begin{array}{c}
       x \neq y \in V\\
       \| x \|_V, \| y \|_V \leqslant R
     \end{array}}} \frac{\| f (x) - f (y) \|_W}{\| x - y \|_V^{\beta}}, \quad
     \| f \|_{\beta, R} \assign \llbracket f \rrbracket_{\beta, R} +
     \sup_{\tmscript{\begin{array}{c}
       x \in V\\
       \| x \|_V \leqslant R
     \end{array}}} \| f (x) \|_V . \]
  For $\lambda \in (0, 1]$, we define the space $C^{\beta, \lambda}_{V, W}$ as
  the collection of all $f \in C (V ; W)$ such that
  \[ \llbracket f \rrbracket_{\beta, \lambda} \assign \sup_{R \geqslant 1}
     R^{- \lambda}  \llbracket f \rrbracket_{\beta, R}, \quad \| f \|_{\beta,
     \lambda} \assign \llbracket f \rrbracket_{\beta, \lambda} + \| f (0) \|_V
     < \infty . \]
  Finally, the classical H\"{o}lder space $C^{\beta}_{V, W}$ is defined as the
  collection of all $f \in C (V ; W)$ such that
  \[ \llbracket f \rrbracket_{\beta} \assign \sup_{\tmscript{\begin{array}{c}
       x \neq y \in V
     \end{array}}} \frac{\| f (x) - f (y) \|_W}{\| x - y \|_V^{\beta}}, \quad
     \| f \|_{\beta} = \llbracket f \rrbracket_{\beta} + \sup_{x \in V} \| f
     (x) \|_V < \infty . \]
  
\end{definition}

\begin{remark}
  We ask the reader to keep in mind that although linked, $\llbracket f
  \rrbracket_{\beta, R}$ and $\llbracket f \rrbracket_{\beta, \lambda}$ denote
  two different quantities. Throughout the paper $R$ will always denote the
  radius of an open ball in $V$ and consequently all related seminorms are
  localised on such ball; instead the parameter $\lambda$ measures the
  polynomial growth of $\llbracket \cdot \rrbracket_{\beta, R}$ as a function
  of $R$.
  
  $C^{\beta}_{V, W, \tmop{loc}}$ is a Fr{\'e}chet space with the topology
  induced by the seminorms $\{ \| f \|_{\beta, R} \}_{R \geqslant 0}$, while
  $C^{\beta, \lambda}_{V, W}$ and $C^{\beta}_{V, W}$ are Banach spaces.
  Observe that if $f \in C^{\beta, \lambda}_{V, W}$, we have an upper bound on
  its growth at infinity, since for any $x \in V$ with $\| x \|_V \geqslant 1$
  it holds
  \[ \| f (x) \|_V \leqslant \| f (x) - f (0) \|_V + \| f (0) \|_V \leqslant
     \| x \|_V^{\beta}  \llbracket f \rrbracket_{\beta, \| x \|_V} + \| f (0)
     \|_V \leqslant \| f \|_{\beta, \lambda} (1 + \| x \|_V^{\beta + \lambda})
     . \]
  In particular, if $\beta + \lambda \leqslant 1$, then $f$ has at most linear
  growth.
\end{remark}

We can now introduce fields $A : [0, T] \times V \rightarrow W$ satisfying a
joint space-time H\"{o}lder continuity. We adopt the incremental notation
$A_{s, t} (x) \assign A (t, x) - A (s, x)$, as well as $A_t (x) = A (t, x)$;
from now on, whenever $A$ appears, it is implicitly assumed that $A (0, x) =
0$ for all $x \in V$.

\begin{definition}
  \label{sec2 defn field2}Given $A$ as above, $\alpha, \beta \in (0, 1)$, we
  say that $A \in C^{\alpha}_t C^{\beta}_{V, W, \tmop{loc}}$ if for any $R
  \geqslant 0$ it holds
  \[ \llbracket A \rrbracket_{\alpha, \beta} \assign \sup_{0 \leqslant s < t
     \leqslant T} \frac{\llbracket A_{s, t} \rrbracket_{\beta, R}}{| t - s
     |^{\alpha} }, \quad \| A \|_{\alpha, \beta} \assign \sup_{0 \leqslant s <
     t \leqslant T} \frac{\| A_{s, t} \|_{\beta, R}}{| t - s |^{\alpha} } <
     \infty . \]
  We say that $A \in C^{\alpha}_t C^{\beta, \lambda}_{V, W}$ if
  \[ \llbracket A \rrbracket_{\alpha, \beta, \lambda} \assign \sup_{0
     \leqslant s < t \leqslant T} \frac{\llbracket A_{s, t} \rrbracket_{\beta,
     \lambda}}{| t - s |^{\alpha} }, \quad \| A \|_{\alpha, \beta, \lambda}
     \assign \sup_{0 \leqslant s < t \leqslant T} \frac{\| A_{s, t} \|_{\beta,
     \lambda}}{| t - s |^{\alpha} } ; \]
  analogue definitions hold for $C^{\alpha}_t C^{\beta}_{V, W}$, $\llbracket
  \cdot \rrbracket_{\alpha, \beta}$, $\| \cdot \|_{\alpha, \beta}$.
\end{definition}

The definition can be extended to the cases $\alpha = 0$ or $\beta = 0$ by
interpreting the norm in the supremum sense: for instance $A \in C^0_t
C^{\beta}_{V, W}$ if
\[ \| A \|_{0, \beta} = \sup_{t \in [0, T]} \| A_t \|_{\beta} < \infty . \]
Given a smooth $F : V \rightarrow W$, we regard its Frech{\'e}t differential
$D^k F$ of order $k$ as a map from $V$ to $\mathcal{L}^k (V ; W)$, the set of
bounded $k$-linear forms from $V^k$ to $W$.

\begin{definition}
  We say that $A \in C^{\alpha}_t C^{n + \beta}_{V, W}$ if $A \in C^{\alpha}_t
  C^{\beta}_{V, W}$ and it is k-times Frech{\'e}t differentiable in $x$, with
  $D^k A \in C^{\alpha}_t C^{\beta}_{V, \mathcal{L}^k (V ; W)}$ for all $k
  \leqslant n$. $C^{\alpha}_t C^{n + \beta}_{V, W}$ is a Banach space with
  norm
  \[ \| A \|_{\alpha, n + \beta} = \sum_{k = 0}^n \| D^k A \|_{\alpha, \beta}
     . \]
  Analogue definitions hold for $C^{\alpha}_t C^{n + \beta}_{V, W,
  \tmop{loc}}$ and $C^{\alpha}_t C^{n + \beta, \lambda}_{V, W}$.
\end{definition}

\subsection{Construction and first properties}

We are now ready to construct nonlinear Young integrals, following the line of
proof from~{\cite{hu}},~{\cite{perkowski}}; other constructions are possible,
see Appendix~\ref{appendixB}.

\begin{theorem}
  \label{sec2 thm definition young integral}Let $\alpha, \beta, \gamma \in (0,
  1)$ such that $\alpha + \beta \gamma > 1$, $A \in C^{\alpha}_t C^{\beta}_{V,
  W, \tmop{loc}}$ and $x \in C^{\gamma}_t V$. Then for any $[s, t] \subset [0,
  T]$ and for any sequence of partitions of $[s, t]$ with infinitesimal mesh,
  the following limit exists and is independent of the chosen sequence of
  partitions:
  \[ \int_s^t A (\mathd u, x_u) \assign
     \tmcolor{blue}{\tmcolor{black}{\lim_{\tmscript{\begin{array}{c}
       | \Pi | \rightarrow 0
     \end{array}}}}} \sum_i A_{t_i, t_{t + 1}} (x_{t_i}) . \]
  The limit is usually referred as a nonlinear Young integral. Furthermore:
  \begin{enumeratenumeric}
    \item For all $(s, r, t) \in \Delta_3$ it holds $\int_s^r A (\mathd u,
    x_u) + \int_r^t A (\mathd u, x_u) = \int_s^t A (\mathd u, x_u)$.
    
    \item If $\partial_t A$ exists continuous, then $\int_s^t A (\mathd u,
    x_u) = \int_s^t \partial_t A (u, x_u) \mathd u$.
    
    \item There exists a constant $C_1 = C_1 (\alpha, \beta, \gamma)$ such
    that
    \begin{equation}
      \left\| \int_s^t A (\mathd u, x_u) - A_{s, t} (x_s) \right\|_W \leqslant
      C_1 | t - s |^{\alpha + \beta \gamma} \llbracket A \rrbracket_{\alpha,
      \beta, \| x \|_{\infty}} \llbracket x \rrbracket^{\beta}_{\gamma} .
      \label{sec2 nonlinear integral estimate1}
    \end{equation}
    \item The map $(A, x) \mapsto \int_0^{\cdot} A (\mathd u, x_u)$ is
    continuous as a function from $C^{\alpha}_t C^{\beta}_{V, W, \tmop{loc}}
    \times C^{\gamma}_t V \rightarrow C^{\alpha}_t W$. More precisely, it is a
    linear map in $A$ and there exists $C_2 = C_2 (\alpha, \beta, \gamma, T)$
    such that
    \begin{equation}
      \left\| \int_0^{\cdot} A^1 (\mathd u, x_u) - \int_0^{\cdot} A^2 (\mathd
      u, x_u) \right\|_{\alpha} \leqslant C_2 \| A^1 - A^2 \|_{\alpha, \beta,
      \| x \|_{\infty}}  (1 + \llbracket x \rrbracket_{\gamma}) ; \label{sec2
      nonlinear integral estimate2}
    \end{equation}
    it is locally $\delta$\mbox{-}H\"{o}lder continuous in $x$ for any $\delta
    \in (0, 1)$ such that $\delta < (\alpha + \beta \gamma - 1) / \gamma$ and
    there exists $C_3 = C_3 (\alpha, \beta, \gamma, \delta, T)$ such that, for
    any $R \geqslant \| x \|_{\infty} \vee \| y \|_{\infty}$, it holds
    \begin{equation}
      \left\| \int_0^{\cdot} A (\mathd u, x_u) - \int_0^{\cdot} A (\mathd u,
      y_u) \right\|_{\alpha} \leqslant C_3 \| A \|_{\alpha, \beta, R}  (1 + \|
      x \|_{\gamma} + \nobracket \| y \nobracket \|_{\gamma})  \llbracket x -
      y \rrbracket^{\delta}_{\gamma} . \label{sec2 nonlinear integral
      estimate3}
    \end{equation}
  \end{enumeratenumeric}
\end{theorem}

\begin{proof}
  In order to show convergence of the Riemann sums, it is enough to apply the
  sewing lemma to the choice $\Gamma_{s, t} \assign A_{s, t} (x_s) = A (t,
  x_s) - A (s, x_s)$. Indeed we have
  \[ \| \Gamma \|_{\alpha} = \sup_{s < t} \frac{\| A_{s, t} (x_s) \|_W}{| t -
     s |^{\alpha}} \leqslant \sup_{s < t} \frac{\| A_{s, t} \|_{0, \| x
     \|_{\infty}}}{| t - s |^{\alpha}} \leqslant \| A \|_{\alpha, 0, \| x
     \|_{\infty}} \]
  and
  
  \begin{align*}
    \| \delta \Gamma_{s, u, t} \|_W & = \| A_{u, t} (x_s) - A_{u, t} (x_u)
    \|_W \leqslant \llbracket A_{u, t} \rrbracket_{\beta, \| x \|_{\infty}} 
    \| x_{u, s} \|_V^{\beta} \leqslant | t - u |^{\alpha}  | u - s |^{\beta
    \gamma}  \llbracket A \rrbracket_{\alpha, \beta, \| x \|_{\infty}} 
    \llbracket x \rrbracket_{\gamma}^{\beta}
  \end{align*}
  
  which implies $\| \delta \Gamma \|_{\alpha + \beta \gamma} \leqslant
  \llbracket A \rrbracket_{\alpha, \beta, \| x \|_{\infty}}  \llbracket x
  \rrbracket_{\gamma}^{\beta}$. In particular $\Gamma \in C^{\alpha, \alpha +
  \beta \gamma}_2 W$ with $\alpha + \beta \gamma > 1$, therefore by the sewing
  lemma we can set
  \[ \int_0^t A (\mathd s, x_s) \assign \mathcal{J} (\Gamma)_t =
     \tmcolor{blue}{\tmcolor{black}{\lim_{\tmscript{\begin{array}{c}
       | \Pi | \rightarrow 0
     \end{array}}}}} \sum_i \Gamma_{t_i, t_{t + 1}} . \]
  Property~1. then follows from $\mathcal{J} (\Gamma)_{s, t} = \mathcal{J}
  (\Gamma)_{s, r} + \mathcal{J} (\Gamma)_{r, t}$ and Property~3. from the
  above estimates on $\| \delta \Gamma \|_{\alpha + \beta \gamma}$. Similarly
  estimate~{\eqref{sec2 nonlinear integral estimate2}} is obtained by the
  previous estimates applied to $A = A^1 - A^2$. Property~2. follows from the
  fact that if $\partial_t A$ exists continuous, then necessarily
  \[ \tmcolor{blue}{\tmcolor{black}{\lim_{\tmscript{\begin{array}{c}
       | \Pi | \rightarrow 0
     \end{array}}}}} \sum_i A_{t_i, t_{t + 1}} (x_{t_i}) = \int_0^t \partial_t
     A (u, x_u) \mathd u. \]
  It remains to show estimate~{\eqref{sec2 nonlinear integral estimate3}}. To
  this end, for fixed $x, y \in C^{\gamma}_t V$ and $R$ as above, we need to
  provide estimates for $\| \delta \tilde{\Gamma} \|_{1 + \varepsilon}$ for
  $\tilde{\Gamma}_{s, t} : = A_{s, t} (x_s) - A_{s, t} (y_s)$ and suitable
  $\varepsilon > 0$. It holds
  
  \begin{align*}
    | \delta \tilde{\Gamma}_{s, u, t} | & \leqslant | A_{u, t} (x_u) - A_{u,
    t} (x_s) | + | A_{u, t} (y_u) - A_{u, t} (y_s) | \leqslant \| A
    \|_{\alpha, \beta, R}  (\llbracket x \rrbracket^{\beta}_{\gamma} +
    \llbracket y \rrbracket^{\beta}_{\gamma})  | t - s |^{\alpha + \beta
    \gamma},\\
    | \delta \tilde{\Gamma}_{s, u, t} | & \leqslant | A_{u, t} (x_u) - A_{u,
    t} (y_u) | + | A_{u, t} (x_s) - A_{u, t} (y_s) | \lesssim \| A \|_{\alpha,
    \beta, R}  \| x - y \|^{\beta}_0  | t - s |^{\alpha}
  \end{align*}
  
  which interpolated together give
  \[ \| \delta \Gamma \|_{(1 - \theta) (\alpha + \beta \gamma) + \theta
     \alpha} \lesssim \| A \|_{\alpha, \beta, R}  (1 + \llbracket x
     \rrbracket_{\gamma} + \llbracket y \rrbracket_{\gamma})  \| x - y
     \|_0^{\beta \theta} \]
  for any $\theta \in (0, 1)$ such that $(1 - \theta) (\alpha + \beta \gamma)
  + \theta \alpha = 1 + \varepsilon > 1$, namely such that
  \[ \beta \theta < \frac{\alpha + \beta \gamma - 1}{\gamma} . \]
  The sewing lemma then implies that
  
  \begin{align*}
    \left\| \int_s^t A (\mathd r, x_r) - \int_s^t A (\mathd r, y_r) \right\|_W
    & \lesssim_{\theta} \left\| \int_s^t A (\mathd r, x_r) - \int_s^t A
    (\mathd r, y_r) - \tilde{\Gamma}_{s, t} \right\|_W + \| \tilde{\Gamma}_{s,
    t} \|_W\\
    & \lesssim \| \delta \tilde{\Gamma} \|_{1 + \varepsilon} | t - s |^{1 +
    \varepsilon}_{} + \| A \|_{\alpha, \beta, R} | t - s |^{\alpha} \| x - y
    \|^{\beta}_0\\
    & \lesssim_{\theta, T} | t - s |^{\alpha}  \| A \|_{\alpha, \beta, R}  (1
    + \| x \|_{\gamma} + \| y \|_{\gamma})  \| x - y \|^{\beta \theta}_0 .
  \end{align*}
  
  Dividing by $| t - s |^{\alpha}$ and taking the supremum we
  obtain~{\eqref{sec2 nonlinear integral estimate3}}.
\end{proof}

\begin{remark}
  \label{sec2 remark other integrals}Several other variants of the nonlinear
  Young integral can be constructed. For instance, for $A$ and $x$ as above,
  we can also define
  \[ \int_0^{\cdot} A (s, \mathd x_s) \in C^{\beta \gamma}_t W \]
  as the sewing of $\Gamma_{s, t} \assign A_s (x_t) - A_s (x_s)$. Another
  possibility are integrals of the form
  \[ \int_0^{\cdot} y_s A (\mathd s, x_s) \]
  for $y \in C^{\delta}_t \mathbb{R}$ such that $\alpha + \delta > 1$ and $A,
  x$ as above. This can be either interpreted as a more classical Young
  integral of the form $\int_0^{\cdot} y_t \mathd \left( \int_0^t A (\mathd s,
  x_s) \right) = \mathcal{J} (\Gamma)$ for $\Gamma_{s, t} = y_s  \int_s^t A
  (\mathd r, x_r)$, or as the sewing of $\tilde{\Gamma}_{s, t} = y_s A_{s, t}
  (x_s) ;$it is immediate to check equivalence of the two definitions. This
  case can be further extended to consider a bilinear map $G : W \times U
  \rightarrow Z$, where $U$ and $Z$ are other Banach spaces, so that
  \[ \int_0^{\cdot} G (y_s, A (\mathd s, x_s)) \in C^{\alpha}_t Z \]
  is well defined for $y \in C^{\delta}_t U$, $A$ and $x$ as above, as the
  sewing of $\Gamma_{s, t} = G (y_s, A_{s, t} (x_s)) \in C^{\alpha, \alpha +
  \delta}_2 Z$, since
  \begin{eqnarray*}
    \| \Gamma_{s, t} \| & \leqslant & | t - s |^{\alpha} \| G \|  \| y
    \|_{\infty} \| A \|_{\alpha, \beta},\\
    \| \delta \Gamma_{s, u, t} \| & \lesssim & | t - s |^{\alpha + \delta}  \|
    G \|  \| y \|_{\delta}  \| A \|_{\alpha, \beta}  (1 + \llbracket x
    \rrbracket_{\gamma}) .
  \end{eqnarray*}
\end{remark}

Nonlinear Young integrals are a generalisation of classical ones, as the next
example shows.

\begin{example}
  \label{sec2 example}Let $f \in C^{\beta} (\mathbb{R}^d ; \mathbb{R}^{d
  \times m})$ and $y \in C^{\alpha}_t \mathbb{R}^m$, then $A (t, x) \assign f
  (x) y_t$ is an element of $C^{\alpha}_t C^{\beta}_{\mathbb{R}^d}$, since
  \[ | A_{s, t} (x) - A_{s, t} (y) | = | [f (x) - f (y)] y_{s, t} | \leqslant
     | f (x) - f (y) |  | y_{s, t} | \leqslant \llbracket f \rrbracket_{\beta}
     \llbracket y \rrbracket_{\alpha} | t - s |^{\alpha} | x - y |^{\beta} .
  \]
  In particular, for any $x \in C^{\gamma}_t \mathbb{R}^d$ with $\alpha +
  \beta \gamma > 1$, we can consider $\int_0^{\cdot} A (\mathd s, x_s)$; this
  corresponds to the classical Young integral $\int_0^{\cdot} f (x_s) \mathd
  y_s$, since both are defined as sewings of
  \[ A_{s, t} (x_s) = f (x_s) y_t - f (x_s) y_s = f (x_s) y_{s, t} . \]
  The previous example generalizes an infinite sum of Young integrals, i.e.
  considering sequences $f^n \in C^{\beta} (\mathbb{R}^d ; \mathbb{R}^d)$,
  $y^n \in C^{\alpha}_t ([0, T] ; \mathbb{R})$ such that (possibly locally)
  \[ \sum_n \| f^n \|_{\beta} \| y^n \|_{\alpha} < \infty . \]
  In this case we can define $A (t, x) : = \sum_n f^n (x) y^n_t$, which
  satisfies $\| A \|_{\alpha, \beta} \leqslant \sum_n \| f^n \|_{\beta} \| y^n
  \|_{\alpha}$ and for any $x \in C^{\delta}_t \mathbb{R}^d$ it holds
  \[ \int_0^{\cdot} A (\mathd s, x_s) = \sum_n \int_0^{\cdot} f^n (x_s) \mathd
     y^n_s . \]
\end{example}

\begin{remark}
  In the classical setting (let us take $d = 1$ for simplicity), if $f : [0,
  T] \times \mathbb{R} \rightarrow \mathbb{R}$ satisfies
  \begin{equation}
    | f (t, z_1) - f (s, z_2) | \leqslant C (| t - s |^{\beta \gamma} + | z_1
    - z_2 |^{\beta}), \label{sec2 final remark eq1}
  \end{equation}
  $x \in C^{\gamma}_t$ and $y \in C^{\alpha}_t$ with $\alpha + \beta \gamma >
  1$, then one can define the Young integral $\int_0^{\cdot} f (s, x_s) \mathd
  y_s$. However, $\int_0^{\cdot} f (s, x_s) \mathd y_s$ does not coincide with
  $\int A (\mathd s, x_s)$ for the choice $A (t, x) : = f (t, x) y_t$.
  
  This is partially because the domain of definition of the two integrals is
  different, since condition~{\eqref{sec2 final remark eq1}} (which is locally
  equivalent to $f \in C^{\beta \gamma}_t C^0_x \cap C^0_t C^{\beta}_x$) is
  not enough to ensure that $A \in C^{\alpha}_t C^{\beta}_x$; however, if we
  additionally assume $f \in C^{\alpha}_t C^{\beta}_x$, then so does $A$, and
  the relation between the two integrals is given by
  \begin{equation}
    \int_0^t A (\mathd s, x_s) = \int_0^t f (s, x_s) \mathd y_s + \int_0^t y_s
    f (\mathd s, x_s) . \label{sec2 final remark eq2}
  \end{equation}
  To derive~{\eqref{sec2 final remark eq2}}, define $\Gamma^A_{s, t} = A_{s,
  t} (x_s)$; then
  \[ \Gamma^A_{s, t} = f (t, x_s) y_t - f (s, x_s) y_s = f (s, x_s) y_{s, t} +
     y_s f_{s, t} (x_s) + R_{s, t} \backassign \Gamma^y_{s, t} + \Gamma^f_{s,
     t} + R_{s, t} \]
  where $| R_{s, t} | = | f_{s, t} (x_t) - f_{s, t} (x_s) | \lesssim | t - s
  |^{\alpha + \beta \gamma}$. This implies $\mathcal{J} (\Gamma^A) =
  \mathcal{J} (\Gamma^y) + \mathcal{J} (\Gamma^f)$, namely~{\eqref{sec2 final
  remark eq2}}.
\end{remark}

\subsection{Nonlinear Young calculus}

Theorem~\ref{sec2 thm definition young integral} establishes continuity of the
map $(A, x) \mapsto \int_0^{\cdot} A (\mathd s, x_s)$; if $A$ is sufficiently
regular, then we can even establish its differentiability.

\begin{proposition}
  \label{sec2 prop frechet differentiability}Let $\alpha, \beta, \gamma \in
  (0, 1)$ such that $\alpha + \beta \gamma > 1$, $A \in C^{\alpha}_t C^{1 +
  \beta}_{V, W, \tmop{loc}}$. Then the nonlinear Young integral, seen as a map
  $F : C^{\gamma}_t V \rightarrow C^{\alpha}_t W$, $F (x) = \int_0^{\cdot} A
  (\mathd s, x_s)$, is Frech{\'e}t differentiable with
  \begin{equation}
    D F (x) : y \mapsto \int_0^{\cdot} D A (\mathd s, x_s) y_s . \label{sec2
    frechet derivative young}
  \end{equation}
\end{proposition}

\begin{proof}
  For notational simplicity we will assume $A \in C^{\alpha}_t C^{1 +
  \beta}_{V, W}$. It is enough to show that, for any $x, y \in C^{\gamma}_t
  V$, the Gateaux derivative of $F$ at $x$ in the direction $y$ is given by
  the expression above, i.e.
  \begin{equation}
    \lim_{\varepsilon \rightarrow 0} \frac{F (x + \varepsilon y) - F
    (x)}{\varepsilon} = \int_0^{\cdot} D A (\mathd s, x_s) y_s \label{sec2
    proof frechet eq1}
  \end{equation}
  where the limit is in the $C^{\alpha}_t W$-topology. Indeed, once this is
  shown, it follows easily from reasoning as in Theorem~\ref{sec2 thm
  definition young integral} that the map $(x, y) \mapsto \int D A (\mathd s,
  x_s) y_s$ is jointly uniformly continuous in bounded balls and linear in the
  second variable; Frech{\'e}t differentiability then follows from existence
  and continuity of the Gateaux differential.
  
  In order to show~{\eqref{sec2 proof frechet eq1}}, setting for any
  $\varepsilon > 0$
  \[ \Gamma^{\varepsilon}_{s, t} \assign \frac{A_{s, t} (x_s + \varepsilon
     y_s) - A_{s, t} (x_s)}{\varepsilon} - D A_{s, t} (x_s) y_s, \]
  it suffices to show that $\mathcal{J} (\Gamma^{\varepsilon}) \rightarrow 0$
  in $C^{\alpha}_t W$. In particular by Lemma~\ref{appendix lemma
  interpolation} from the Appendix, we only need to check that $\|
  \Gamma^{\varepsilon} \|_{\alpha} \rightarrow 0$ as $\varepsilon \rightarrow
  0$ while $\| \delta \Gamma^{\varepsilon} \|_{\alpha + \beta \gamma}$ stays
  uniformly bounded. It holds
  
  \begin{align*}
    \| \Gamma^{\varepsilon}_{s, t} \|_W & = \left\| \int_0^1 [D A_{s, t} (x_s
    + \lambda \varepsilon y_s) - D A_{s, t} (x_s)] y_s \mathd \lambda
    \right\|_W\\
    & \leqslant \varepsilon^{\beta} \| D A_{s, t} \|_{\beta}  \| y_s
    \|_V^{\beta + 1} \leqslant \varepsilon^{\beta} | t - s |^{\alpha}  \| A
    \|_{\alpha, 1 + \beta}  \| y \|_{\delta}^{\beta + 1}
  \end{align*}
  
  which implies that $\| \Gamma^{\varepsilon} \|_{\alpha} \lesssim
  \varepsilon^{\beta} \rightarrow 0$; similar calculations show that
  
  \begin{align*}
    \| \Gamma^{\varepsilon}_{s, u, t} \|_W & = \left\| \int_0^1 [D A_{u, t}
    (x_s + \lambda \varepsilon y_s) - D A_{u, t} (x_s)] y_s \mathd \lambda -
    \int_0^1 [D A_{u, t} (x_u + \lambda \varepsilon y_u) - D A_{u, t} (x_u)]
    y_u \mathd \lambda \right\|_W\\
    & = \| - \int_0^1 [D A_{u, t} (x_s + \lambda \varepsilon y_s) - D
    A_{u, t} (x_s)] y_{s, u} \mathd \lambda\\
    &  \quad + \int_0^1 [D A_{u, t} (x_s + \lambda \varepsilon y_s) - D
    A_{u, t} (x_s) - D A_{u, t} (x_u + \lambda \varepsilon y_u) + D A_{u, t}
    (x_u)] y_u \mathd \lambda \|_W\\
    & \lesssim | t - s |^{\alpha + \gamma} \| D A \|_{\alpha, \beta}  \| y
    \|_{\gamma}^{1 + \beta} + | t - s |^{\alpha + \beta \gamma}  \| D A
    \|_{\alpha, \beta} \| y \|_{\gamma} (\llbracket x
    \rrbracket_{\gamma}^{\beta} + \llbracket y \rrbracket_{\gamma}^{\beta})
  \end{align*}
  
  which implies that $\| \delta \Gamma \|_{\alpha + \beta \gamma} \lesssim 1$
  uniformly in $\varepsilon > 0$. The conclusion the follows.
\end{proof}

Proposition~\ref{sec2 prop frechet differentiability} allows to give an
alternative proof of Lemma~6 from~{\cite{galeatigubinelli1}}.

\begin{corollary}
  \label{sec2 corollary frechet}Let $\alpha, \beta, \gamma \in (0, 1)$ such
  that $\alpha + \beta \gamma > 1$, $A \in C^{\alpha}_t C^{1 + \beta}_{V, W,
  \tmop{loc}}$, $x^1, x^2 \in C^{\gamma}_t V$. Then
  \begin{equation}
    \int_0^{\cdot} A (\mathd s, x^1_s) - \int_0^{\cdot} A (\mathd s, x^2_s) =
    \int_0^{\cdot} v_{\mathd s} (x^1_s - x^2_s) \label{sec2 corollary frechet
    separability}
  \end{equation}
  with $v$ given by
  \begin{equation}
    v_t \assign \int_0^t \int_0^1 D A (\mathd s, x^2_s + \lambda (x_s^1 -
    x^2_s)) \mathd \lambda ; \label{sec2 corollary frechet eq1}
  \end{equation}
  the above formula meaningfully defines an element of $C^{\alpha}_t
  \mathcal{L} (V, W)$ which satisfies
  \begin{equation}
    \llbracket v \rrbracket_{\alpha} \leqslant C \| D A \|_{\alpha, \beta, R}
    (1 + \llbracket x^1 \rrbracket_{\gamma} + \llbracket x^2
    \rrbracket_{\gamma}) \label{sec2 corollary frechet estimate}
  \end{equation}
  where $R \geqslant \| x \|_{\infty} \vee \| y \|_{\infty}$ and $C = C
  (\alpha, \beta, \gamma, T)$.
\end{corollary}

\begin{proof}
  It follows from the hypothesis on $A$ that the map
  \begin{equation}
    y \in V \mapsto \int_0^1 \left[ \int_0^t D A (\mathd s, x_s^2 + \lambda
    (x^1_s - x^2_s)) y \right] \mathd \lambda \in W \label{sec2 corollary
    frechet eq2}
  \end{equation}
  is well defined, the outer integral being in the Bochner sense, and it is
  linear in $y$; moreover estimate~{\eqref{sec2 nonlinear integral estimate1}}
  combined with the trivial inequality $1 + \llbracket x^2 + \lambda (x^1_s -
  x^2_s) \rrbracket_{\gamma}^{\beta} \lesssim 1 + \llbracket x^1
  \rrbracket_{\gamma} + \llbracket x^2 \rrbracket_{\gamma}$, valid for any
  $\lambda, \beta \in [0, 1]$, yields
  \[ \left\| \int_0^1 \left[ \int_0^t D A (\mathd s, x^2 + \lambda (x^1_s -
     x^2_s)) y \right] \mathd \lambda \right\|_W \lesssim \| D A \|_{\alpha,
     \beta, R}  (1 + \llbracket x^1 \rrbracket_{\gamma} + \llbracket x^2
     \rrbracket_{\gamma})  \| y \|_V . \]
  In particular, if we define $v_t$ as the linear map appearing~{\eqref{sec2
  corollary frechet eq2}}, it is easy to check that similar estimates yield $v
  \in C^{\alpha}_t \mathcal{L} (V, W)$. The fact that this definition coincide
  with the one from~{\eqref{sec2 corollary frechet eq1}}, i.e. that we can
  exchange integration in $\mathd \lambda$ and in ``$\mathd s$'', follows from
  the Fubini theorem for the sewing map, see Lemma~\ref{appendix lemma fubini}
  in the Appendix. Inequality~{\eqref{sec2 corollary frechet estimate}} then
  follows from estimates analogue to the ones obtained above.
  Identity~{\eqref{sec2 corollary frechet separability}} is an application of
  the more abstract classical identity
  \[ F (x^1) - F (x^2) = \left[ \int_0^1 D F (x^2 + \lambda (x^1 - x^2))
     \mathd \lambda \right] (x^1 - x^2) \]
  applied to $F (x) = \int_0^{\cdot} A (\mathd s, x_s)$, for which the exact
  expression for $D F$ is given by Proposition~\ref{sec2 prop frechet
  differentiability}.
\end{proof}

The following It\^{o}-type formula is taken from~{\cite{hu}}, Theorem~3.4.

\begin{proposition}
  \label{sec2 prop ito formula}Let $F \in C^{\alpha}_t C^{\beta}_{V, W,
  \tmop{loc}}$ and $x \in C^{\gamma}_t V$ with $\alpha + \beta \gamma > 1$,
  then it holds
  \begin{equation}
    F (t, x_t) - F (0, x_0) = \int_0^t F (\mathd s, x_s) + \int_0^t F (s,
    \mathd x_s) ; \label{sec2 ito formula1}
  \end{equation}
  if in addition $F \in C^0_t C^{1 + \beta'}_{V, W, \tmop{loc}}$ with $\beta'
  \in (0, 1)$ s.t. $\gamma (1 + \beta') > 1$, then
  \begin{equation}
    F (t, x_t) - F (0, x_0) = \int_0^t F (\mathd s, x_s) + \int_0^t D F (s,
    x_s) (\mathd x_s) . \label{sec2 ito formula2}
  \end{equation}
  In particular, if $x = \int_0^{\cdot} A (\mathd s, y_s)$ for some $A \in
  C^{\gamma}_t C^{\delta}_V$, $y \in C^{\eta}_t V$ with $\gamma + \eta \delta
  > 1$, then~{\eqref{sec2 ito formula2}} becomes
  \begin{equation}
    F (t, x_t) - F (0, x_0) = \int_0^t F (\mathd s, x_s) + \int_0^t D F (s,
    x_s) (A (\mathd s, y_s)) . \label{sec2 ito formula3}
  \end{equation}
\end{proposition}

\begin{proof}
  Let $0 = t_0 < t_1 < \cdots < t_n = t$, then it holds
  
  \begin{align*}
    F (t, x_t) - F (0, x_0) & = \sum_i [F (t_{i + 1}, x_{t_{i + 1}}) - F (t_i,
    x_{t_i})]\\
    & = \sum_i F_{t_i, t_{i + 1}} (x_{t_i}) + \sum_i [F_{t_i} (x_{t_{i + 1}})
    - F_{t_i} (x_{t_i})] + \sum_i R_{t_i, t_{i + 1}} = : I^n_1 + I^n_2 + I^n_3
  \end{align*}
  
  where $R_{t_i, t_{i + 1}} = F_{t_i, t_{i + 1}} (x_{t_i + 1}) - F_{t_i, t_{i
  + 1}} (x_{t_i})$ satisfies $\| R_{t_i, t_{i + 1}} \| \leqslant \| F
  \|_{\alpha, \beta, \| x \|_{\infty}}  \llbracket x
  \rrbracket_{\gamma}^{\beta} | t_{i + 1} - t_i |^{\alpha + \beta \gamma}$,
  while $I^n_1$ and $I^n_2$ are Riemann-Stjeltes sums associated to
  $\Gamma^1_{s, t} = F_{s, t} (x_s)$ and $\Gamma^2_{s, t} = F_s (x_t) - F_s
  (x_s)$. Taking a sequence of partitions $\Pi_n$ with $| \Pi_n | \rightarrow
  0$, by the above estimate we have $I^n_3 \rightarrow 0$ and by the sewing
  lemma we obtain
  \[ F (t, x_t) - F (0, x_0) = \mathcal{J} (\Gamma^1)_t + \mathcal{J}
     (\Gamma^2)_t, \]
  which is exactly~{\eqref{sec2 ito formula1}}. If $F \in C^0_t C^{1 +
  \beta'}_{V, W, \tmop{loc}}$, then setting $\Gamma^3_{s, t} \assign D F (s,
  x_s) (x_{s, t})$, it holds
  
  \begin{align*}
    \| \Gamma^2_{s, t} - \Gamma^3_{s, t} \|_V & = \| F (s, x_t) - F (s, x_s) -
    D F (s, x_s) (x_{s, t}) \|_V\\
    & = \left\| \int_0^1 [D F (s, x_s + \lambda x_{s, t}) - D F (s, x_s)]
    (x_{s, t}) \mathd \lambda \right\|_V\\
    & \lesssim \| D F (s, \cdot) \|_{\beta', \| x \|_{\infty}}  \| x_{s, t}
    \|^{1 + \beta'}_{} \lesssim \| F \|_{0, 1 + \beta', \| x \|_{\infty}}
    \llbracket x \rrbracket_{\gamma}^{\beta'} | t - s |^{\gamma (1 + \beta')}
  \end{align*}
  
  which under the assumption $\gamma (1 + \beta') > 1$ implies by the sewing
  lemma that $\mathcal{J} (\Gamma^2) = \mathcal{J} (\Gamma^3)$ and
  thus~{\eqref{sec2 ito formula2}}. The proof of~{\eqref{sec2 ito formula3}}
  is analogue, only this time consider $\Gamma^4_{s, t} \assign D F (s, x_s)
  (A_{s, t} (y_s))$, then it's easy to check that $\| \Gamma^3_{s, t} -
  \Gamma^4_{s, t} \|_V \lesssim | t - s |^{\gamma + \eta \delta}$ which
  implies that $\mathcal{J} (\Gamma^3) = \mathcal{J} (\Gamma^4)$.
\end{proof}

\begin{remark}
  The above formulas admit further variants. For instance for any $F \in
  C^{\alpha}_t C^{\beta}_{V, W}$, $x \in C^{\gamma}_t V$ and $g \in
  C^{\delta}_t \mathbb{R}$ with $\alpha + \beta \gamma > 1$, $\alpha + \delta
  > 1$ and $\beta \gamma + \delta > 1$ it holds
  \[ \int_0^t g_s \mathd [F (s, x_s)] = \int_0^t g_s F (\mathd s, x_s) +
     \int_0^t g_s F (s, \mathd x_s) \]
  and we have the product rule formula
  \[ g_t F (t, x_t) - g_0 F (0, x_0) = \int_0^t F (s, x_s) \mathd g_s +
     \int_0^t g_s F (\mathd s, x_s) + \int_0^t g_s F (s, \mathd x_s) . \]
  Also observe that, whenever $\partial_t F$ exists continuous, it holds
  \[ \int_0^t g_s F (\mathd s, x_s) = \int_0^t g_s \partial_t F (s, x_s)
     \mathd s \quad \forall \, g \in C^{\delta}_t \mathbb{R}. \]
\end{remark}

\section{Existence, uniqueness, numerical schemes}\label{sec3}

This section is devoted to the study of nonlinear Young differential equations
(YDE for short), defined below; it provides sufficient conditions for
existence and uniqueness of solutions, as well as convergence of numerical
schemes.

\begin{definition}
  Let $A \in C^{\alpha}_t C^{\beta}_{V, \tmop{loc}}$, $x_0 \in V$. We say that
  $x$ is a solution to the YDE associated to $(x_s, A)$ on an interval $[s, t]
  \subset [0, T]$ if $x \in C^{\gamma} ([s, t] ; V)$ for some $\gamma$ such
  that $\alpha + \beta \gamma > 1$ and it satisfies
  \begin{equation}
    x_r = x_s + \int_s^r A (\mathd u, x_u) \quad \forall \, r \in [s, t] .
    \label{sec3 defn solution}
  \end{equation}
\end{definition}

Before proceeding further, let us point out that by Example~\ref{sec2 example}
any Young differential equation
\[ x_t = x_0 + \int f (x_s) \mathd y_s \]
can be reinterpreted as a nonlinear YDE associated to $A \assign f \otimes y$.
Nonlinear YDEs therefore are a natural extension of the standard ones; most
results regarding their existence and uniqueness which will be presented are
perfect analogues (in terms of regularity requirements) to the well known
classical ones (which can be found for instance in~{\cite{lejay}} or Section~8
of~{\cite{frizhairer}}).

\

Throughout this section, for $x : [0, T] \rightarrow V$ and $I \subset [0,
T]$, we set
\[ \llbracket x \rrbracket_{\gamma ; I} \assign
   \sup_{\tmscript{\begin{array}{c}
     s, t \in I\\
     s \neq t
   \end{array}}} \frac{\| x_{s, t} \|_V}{| t - s |^{\gamma}} \]
as well as $\llbracket x \rrbracket_{\gamma ; s, t}$ in the case $I = [s, t]$;
similarly for $\| x \|_{\infty ; I}$ and $\| x \|_{\gamma ; I}$. For any
$\Delta > 0$ we also define
\[ \llbracket x \rrbracket_{\gamma, \Delta, V} = \llbracket x
   \rrbracket_{\gamma, \Delta} \assign \sup_{\tmscript{\begin{array}{c}
     s, t \in [0, T]\\
     | s - t | \in (0, \Delta]
   \end{array}}} \frac{\| x_{s, t} \|_V}{| t - s |^{\gamma}} . \]
\subsection{Existence}\label{sec3.1}

We provide here sufficient conditions for the existence of either local or
global solutions to the YDE, under suitable compactness assumptions on $A$.
The proof is based on an Euler scheme for the YDE, in the style of those
from~{\cite{davie}},~{\cite{lejay}}; its rate of convergence will be studied
later on. Other proofs, based on a priori estimates and compactness techniques
or an application of Leray--Schauder--Tychonoff fixed point theorem, are
possible, see~{\cite{catelliergubinelli}},~{\cite{hu}}.

\begin{theorem}
  \label{sec3.1 thm existence}Let $A \in C^{\alpha}_t C^{\beta}_{V, W}$ where
  $W$ is compactly embedded in $V$ and $\alpha (1 + \beta) > 1$. Then for any
  $s > 0$ and $x_s \in V$ there exists a solution to the YDE
  \begin{equation}
    x_t = x_s + \int_s^t A (\mathd s, x_s) \quad \forall \, t \in [s, T] .
    \label{sec3.1 yde}
  \end{equation}
\end{theorem}

\begin{proof}
  The proof is based on the application of an Euler scheme. Up to rescaling
  and shifting, we can assume for simplicity $T = 1$ and $s = 0$.
  
  Fix $N \in \mathbb{N}$, set $t^n_k = k / n$ for $k \in \{ 0, \ldots, n \}$
  and define recursively $(x^n_k)_{k = 1}^n$ by $x^n_0 = x_0$ and
  \[ x^n_{k + 1} = x^n_k + A_{t^n_k, t^n_{k + 1}} (x^n_k) . \]
  We can embed $(x^n_k)_{k = 1}^n$ into $C^0_t V$ by setting
  \[ x^n_t \assign x_0 + \sum_{0 \leqslant k \leqslant \lfloor n t \rfloor}
     A_{t^n_k, t \wedge t^{n + 1}_k} (x_k^n) ; \]
  note that by construction $x^n - x_0$ is a path in $C^{\alpha}_t W$. Using
  the identity
  \[ A_{s, t} (x^n_s) = \int_s^t A (\mathd r, x^n_r) + \int_s^t [A (\mathd r,
     x^n_s) - A (\mathd r, x^n_r)] \]
  we deduce that $x^n$ satisfies a YDE of the form
  \begin{equation}
    x^n_t = x_0 + \int_0^t A (\mathd s, x_s^n) + \psi^n_t \label{sec3.1 proof
    euler eq1}
  \end{equation}
  where
  \[ \psi^n_t = \sum_{0 \leqslant k \leqslant n} \psi_t^{n, k} = \sum_{0
     \leqslant k \leqslant n} \int_{t^n_k}^{(t \wedge t^n_{k + 1}) \vee t^n_k}
     [A (\mathd r, x^n_{t^n_k}) - A (\mathd r, x^n_r)] . \]
  By the properties of Young integrals, $\psi^n$ satisfies
  \begin{equation}
    \| \psi^n_{t^n_k, t^n_{k + 1}} \|_W = \left\| \int_{t^n_k}^{t^n_{k + 1}}
    [A (\mathd r, x^n_{t^n_k}) - A (\mathd r, x^n_r)] \right\|_W \lesssim n^{-
    \alpha (1 + \beta)} \| A \|_{\alpha, \beta}  \llbracket x^n
    \rrbracket^{\beta}_{\alpha, 1 / n, V} . \label{sec3.1 proof euler eq2}
  \end{equation}

  We first want to obtain a bound for $\llbracket \psi^n \rrbracket_{\gamma,
  \Delta, W}$; we can assume wlog $\Delta > 1 / n$, since we want to take $n
  \rightarrow \infty$. Estimates depend on whether $s$ and $t$ lie on the same
  interval $[t^n_k, t^n_{k + 1}]$ or not; assume first this is the case, then
  
  \begin{align*}
    \| \psi^n_{s, t} \|_W & = \left\| \int_s^t [A (\mathd r, x^n_{t^n_k}) - A
    (\mathd r, x^n_r)] \right\|_W\\
    & \lesssim \| A_{s, t} (x^n_{t^n_k}) - A_{s, t} (x_s^n) \|_W + | t - s
    |^{\alpha (1 + \beta)} \| A \|_{\alpha, \beta}  \llbracket x^n
    \rrbracket^{\beta}_{\alpha, \Delta, V}\\
    & \lesssim \, n^{- \alpha \beta} | t - s |^{\alpha} \| A \|_{\alpha,
    \beta}  \llbracket x^n \rrbracket^{\beta}_{\alpha, \Delta, V} .
  \end{align*}
  
  Next, given $s < t$ such that $| t - s | < \Delta$ which are not in the same
  interval, there are around $n | t - s |$ intervals separating them, i.e.
  there exist $l < m$ such that $m - l \sim n | t - s |$ and $s \leqslant
  t^n_l < \cdots < t^n_m \leqslant t$. Therefore in this case we have
  
  \begin{align*}
    \| \psi^n_{s, t} \|_W & \leqslant \| \psi_{s, t^n_l}^n \|_W + \sum_{k =
    l}^{m - 1} \| \psi^n_{t^n_k, t^n_{k + 1}} \|_W + \| \psi^n_{t_m^n, t_{}}
    \|_W\\
    & \lesssim \| A \|_{\alpha, \beta} \llbracket x^n
    \rrbracket^{\beta}_{\alpha, \Delta, V} [| t - s |^{\alpha} n^{- \alpha
    \beta} + (m - l) n^{- \alpha (1 + \beta)}]\\
    & \lesssim \| A \|_{\alpha, \beta} \llbracket x^n
    \rrbracket^{\beta}_{\alpha, \Delta, V} [| t - s |^{\alpha} n^{- \alpha
    \beta} + | t - s | n^{1 - \alpha (1 + \beta)}]\\
    & \lesssim \| A \|_{\alpha, \beta} \llbracket x^n
    \rrbracket^{\beta}_{\alpha, \Delta, V} | t - s |^{\alpha} n^{1 - \alpha (1
    + \beta)}
  \end{align*}
  
  where in the second line we used both~{\eqref{sec3.1 proof euler eq2}} and
  the previous bound for $\psi^n_{s, t^n_l}$ and $\psi^n_{t^n_m, t}$, while in
  the last one the fact that $- \alpha \beta \leqslant 1 - \alpha (1 +
  \beta)$. Overall we conclude that
  \begin{equation}
    \llbracket \psi^n \rrbracket_{\alpha, \Delta, W} \leqslant \kappa_1 n^{1 -
    \alpha (1 + \beta)} \| A \|_{\alpha, \beta}  \llbracket x^n
    \rrbracket_{\alpha, \Delta, V}^{\beta} \label{sec3.1 proof euler eq3}
  \end{equation}
  for a suitable constant $\kappa_1 = \kappa_1 (\alpha, \beta)$ independent of
  $\Delta$ and $n$.
  
  Our next goal is a uniform bound for $\llbracket x^n \rrbracket_{\alpha,
  \Delta, W}$. Since $x^n$ solves~{\eqref{sec3.1 proof euler eq1}}, it holds
  
  \begin{align*}
    \| x^n_{s, t} \|_W & \lesssim \| A_{s, t} (x^n_s) \|_W + | t - s |^{\alpha
    (1 + \beta)} \| A \|_{\alpha, \beta}  \llbracket x^n
    \rrbracket^{\beta}_{\alpha, \Delta, W} + \| \psi^n_{s, t} \|_W\\
    & \lesssim | t - s |^{\alpha} \| A \|_{\alpha, \beta} + | t - s
    |^{\alpha} \Delta^{\alpha \beta} \| A \|_{\alpha, \beta} \llbracket x^n
    \rrbracket^{\beta}_{\alpha, \Delta, W} + | t - s |^{\alpha} \llbracket
    \psi^n \rrbracket_{\alpha, \Delta, W}\\
    & \lesssim | t - s |^{\alpha} \| A \|_{\alpha, \beta} + | t - s
    |^{\alpha} \| A \|_{\alpha, \beta} \llbracket x^n
    \rrbracket^{\beta}_{\alpha, \Delta, W} (\Delta^{\alpha \beta} + n^{1 -
    \alpha (1 + \beta)})
  \end{align*}
  
  and so dividing by $| t - s |$ and taking the supremum over all $| t - s | <
  \Delta$, choosing $\Delta$ such that $\Delta^{\alpha \beta}  \| A
  \|_{\alpha, \beta} \leqslant 1 / 4$, then for all $n$ big enough such that
  $n^{1 - \alpha (1 + \beta)} \| A \|_{\alpha, \beta} \leqslant 1 / 4$ it
  holds
  \[ \llbracket x^n \rrbracket_{\alpha, \Delta, W} \lesssim \| A \|_{\alpha,
     \beta} + \frac{1}{2} \llbracket x^n \rrbracket^{\beta}_{\alpha, \Delta,
     W} \lesssim \| A \|_{\alpha, \beta} + \frac{1}{2} + \frac{1}{2}
     \llbracket x^n \rrbracket_{\alpha, \Delta, W} \]
  by the trivial bound $a^{\beta} \leqslant 1 + a$, which holds for all $\beta
  \in [0, 1]$ and $a \geqslant 0$. This implies the uniform bound $\llbracket
  x^n \rrbracket_{\alpha, \Delta, W} \lesssim 1 + \| A \|_{\alpha, \beta}$ for
  all $n$ big enough.
  
  The subspace $\{ y \in C^{\alpha} ([0, 1] ; W) : y_0 = 0 \}$ is a Banach
  space endowed with the seminorm $\llbracket y \rrbracket_{\alpha, \Delta,
  W}$, which in this case is equivalent to the norm $\| y \|_{\alpha, W}$; $\{
  x_n - x_0 \}_{n \in \mathbb{N}}$ is a uniformly bounded sequence in this
  space. By Ascoli--Arzel{\`a}, since $W$ compactly embeds in $V$, we can
  extract a subsequence (not relabelled for simplicity) such that $x_n - x_0
  \rightarrow x - x_0$ in $C^{\alpha - \varepsilon}_t V$ for any $\varepsilon
  > 0$, for some $x \in C^{\alpha}_t V$ such that $x (0) = x_0$. Observe that
  $\psi^n$ satisfy~{\eqref{sec3.1 proof euler eq3}} and $\llbracket x^n
  \rrbracket^{\beta}_{\alpha, \Delta, V}$ are uniformly bounded, therefore
  $\psi^n \rightarrow 0$ in $C^{\alpha}_t W$ as $n \rightarrow \infty$;
  choosing $\varepsilon$ small enough s.t. $\alpha + \beta (\alpha -
  \varepsilon) > 1$, by continuity of the non-linear Young integral it holds
  \[ \int_0^{\cdot} A (\mathd s, x_s^n) \rightarrow \int_0^{\cdot} A (\mathd
     s, x_s) \quad \text{in } C^{\alpha}_t W \]
  and therefore passing to the limit in~{\eqref{sec3.1 proof euler eq1}} we
  obtain the conclusion.
\end{proof}

\begin{remark}
  If $V$ is finite dimensional, the compactness condition is trivially
  satisfied by taking $V = W$. The proof also works for non uniform partitions
  $\Pi_n$ of $[0, T]$, under the condition that their mesh $| \Pi_n |
  \rightarrow 0$ and that there exists $c > 0$ such that $| t^n_{i + 1} -
  t^n_i | \geqslant c | \Pi_n |$ for all $n \in \mathbb{N}$, $i \in \{ 0,
  \ldots, N_n \}$.
\end{remark}

\begin{remark}
  \label{sec3.1 remark euler a priori estimates}The proof provides several
  estimates, some of which are true even without the compactness assumption.
  For instance, by $\llbracket x^n \rrbracket_{\alpha, \Delta} \lesssim 1 + \|
  A \|_{\alpha, \beta}$ and Exercise~4.24 from~{\cite{frizhairer}}, choosing
  $\Delta$ s.t. $\Delta^{\alpha \beta} \| A \|_{\alpha, \beta} \sim 1$, we
  deduce that there exists $C_1 = C_1 (\alpha, \beta, T)$ such that
  \[ \llbracket x^n \rrbracket_{\alpha} \leqslant C_1 \left( 1 + \| A
     \|_{\alpha, \beta}^{1 + \frac{1 - \alpha}{\alpha \beta}} \right)  \quad
     \forall \, n \in \mathbb{N}. \]

  Estimate~{\eqref{sec3.1 proof euler eq3}} is true for any choice of $\Delta
  > 0$, in particular for $\Delta = T$, which gives a global bound; combining
  it with the above one, we deduce that
  \[ \llbracket \psi^n \rrbracket_{\alpha} \leqslant C_2 n^{1 - \alpha (1 +
     \beta)} \left( 1 + \| A \|_{\alpha, \beta}^{\frac{1 + \alpha
     \beta}{\alpha}} \right) \quad \forall \, n \in \mathbb{N} \]
  for some $C_2 = C_2 (\alpha, \beta, T)$. Also observe that from the
  assumptions on $\alpha$ and $\beta$ it always holds
  \[ 1 + \frac{1 - \alpha}{\alpha \beta} \leqslant 2, \qquad \frac{1 + \alpha
     \beta}{\alpha} \leqslant 3. \]

  Under the compactness assumption, since $x^n \rightarrow x$ in $C^0_t V$,
  the solution $x$ obtained also satisfies
  \begin{equation}
    \llbracket x \rrbracket_{\alpha} \leqslant \liminf_{n \rightarrow \infty}
    \llbracket x^n \rrbracket_{\alpha} \leqslant C_1 \left( 1 + \| A
    \|_{\alpha, \beta}^{1 + \frac{1 - \alpha}{\alpha \beta}} \right) \leqslant
    2 C_1 (1 + \| A \|_{\alpha, \beta}^2) . \label{sec3.1 euler a priori
    estimate1}
  \end{equation}

  Finally observe that by going through the same proof of~{\eqref{sec3.1 proof
  euler eq3}}, for any $T > 0$ and $\alpha, \beta, \gamma$ such that $\alpha +
  \beta \gamma > 1$, there exists $C_3 = C_3 (\alpha, \beta, \gamma, T)$ such
  that
  \begin{equation}
    \llbracket \psi^n \rrbracket_{\alpha, \Delta, V} \leqslant C_3 n^{1 -
    \alpha - \beta \gamma} \| A \|_{\alpha, \beta}  \llbracket x^n
    \rrbracket_{\gamma, \Delta, V}^{\beta} \quad \forall \, n \in \mathbb{N}.
    \label{sec3.1 euler a priori estimate2}
  \end{equation}
  This estimate is rather useful when $A$ enjoys different space-time
  regularity at different scales, see the discussion at Section~\ref{sec3.4}.
\end{remark}

\begin{corollary}
  \label{sec3.1 cor local existence}Let $A \in C^{\alpha}_t C^{\beta}_{V, W,
  \tmop{loc}}$ where $W$ is compactly embedded in $V$ and $\alpha (1 + \beta)
  > 1$. Then for any $s \in [0, T)$ and any $x_s \in V$, there exists
  $\tau^{\ast} \in (s, T]$ and a solution to the YDE~{\eqref{sec3.1 yde}}
  defined on $[s, T^{\ast})$, with the property that either $T^{\ast} = T$ or
  \[ \lim_{t \uparrow T^{\ast}}  \| x_t \|_V = + \infty . \]
\end{corollary}

\begin{proof}
  As before it is enough to treat the case $s = 0, T = 1$. Fix $R > 0$ and
  consider $A^R \in C^{\alpha}_t C^{\beta}_{V, W}$ such that $A^R (t, x) = A
  (t, x)$ for any $(t, x)$ with $\| x \|_V \leqslant 2 R$ and $A^R (t, x)
  \equiv 0$ for $\| x \|_V \geqslant 3 R$; let $C_R \assign C (1 + \| A
  \|_{\alpha, \beta, 3 R}^2)$, where $C$ is the constant appearing
  in~{\eqref{sec3.1 euler a priori estimate1}}.
  
  For any $x_0 \in V$ with $\| x_0 \| \leqslant R$, by Theorem~\ref{sec3.1 thm
  existence} there exists a solution $x_{\cdot}$ to the YDE associated to
  $(x_0, A^R)$ on the interval $[0, 1]$; setting $\tau_1 \assign \inf \{ t \in
  [0, 1] : \| x_t \|_V \geqslant 2 R \}$, by~{\eqref{sec3.1 euler a priori
  estimate1}} it holds $\llbracket x \rrbracket_{\alpha ; [0, \tau_1]}
  \leqslant C_R$, and so
  \[ 2 R = \| x_{\tau_1} \|_V \leqslant \| x_0 \|_V + \tau_1^{\alpha}
     \llbracket x \rrbracket_{\alpha ; [0, \tau_1]} \leqslant R +
     \tau_1^{\alpha} C_R \]
  which implies
  \begin{equation}
    \tau_1 \geqslant \left( \frac{C_R}{R} \right)^{- \alpha} . \label{sec3.1
    blowup estimate}
  \end{equation}
  In particular, since $A = A^R$ on $[0, T] \times B_{2 R}$, we conclude that
  $x_{\cdot}$ is also a solution to the YDE associated to $(x_0, A)$ on the
  interval $[0, \tau_1]$.
  
  We can now iterate this procedure, i.e. set $x^1 \assign x_{\tau_1}$ and
  construct another solution to~{\eqref{sec3.1 yde}}, defined on an interval
  $[\tau_1, \tau_2]$, and so on; by ``gluing'' these solutions together, we
  obtain an increasing sequence $\{ \tau_n \} \subset [0, 1]$ and a solution
  $x_{\cdot}$ defined on $[0, T^{\ast})$, where $T^{\ast} = \lim_n \tau_n$.
  
  Now suppose that $T^{\ast} < T$ and $\liminf_{t \rightarrow T^{\ast}} \| x_t
  \|_V < \infty$, then we can find a sequence $t_n \rightarrow T^{\ast}$ such
  that $\| x_{t_n} \|_V \leqslant M$ for some $M > 0$; but then starting from
  any of this $x_{t_n}$ we can construct another solution $y^n$ defined on
  $[t_n, t_n + \varepsilon]$, where $\varepsilon$ is uniform in $n$ since $\|
  x_{t_n} \| \leqslant M$ and $\varepsilon$ can be estimated by~{\eqref{sec3.1
  blowup estimate}} with $R$ replaced by $M$. By replacing the solution
  $x_{\cdot}$ on $[t_n, T^{\ast})$ with $y^n$, choosing $n$ big enough, we can
  construct a solution defined on $[0, T^{\ast} + \varepsilon / 2)$.
  Reiterating this procedure we obtain the conclusion.
\end{proof}

\subsection{A priori estimates}\label{sec3.2}

A classical way to pass from local to global solutions is to establish
suitable a priori estimates, which are also of fundamental importance for
compactness arguments. Throughout this section, we assume that a solution $x$
to the YDE is already given and focus exclusively on obtainig bounds on it;
for simplicity we work on $[0, T]$, but all the statements immediately
generalise to $[s, T]$.

\begin{proposition}
  \label{sec3.3 prop a priori estimate 1}Let $\alpha > 1 / 2$, $\beta \in (0,
  1)$ such that $\alpha (1 + \beta) > 1$, $A \in C^{\alpha}_t C^{\beta}_V$,
  $x_0 \in V$ and $x \in C^{\alpha}_t V$ be a solution to the associated YDE.
  Then there exists $C = C (\alpha \comma \beta, T)$ such that
  \begin{equation}
    \llbracket x \rrbracket_{\alpha} \leqslant C (1 + \| A \|_{\alpha,
    \beta}^2), \qquad \| x \|_{\alpha} \leqslant C (1 + \| x_0 \|_V + \| A
    \|_{\alpha, \beta}^2) . \label{sec3.3 a priori estimate1}
  \end{equation}
\end{proposition}

\begin{proof}
  \tmcolor{blue}{}Let $\Delta \in (0, T]$ be a parameter to be chosen later.
  For any $s < t$ such that $| s - t | \leqslant \Delta$, using the fact that
  $x$ is a solution, it holds
  
  \begin{align*}
    \| x_{s, t} \|_V & = \, \left\| \int_s^t A (\mathd u, x_u) \right\|_V\\
    & \leqslant \, \| A_{s, t} (x_s) \|_V + \kappa_1 | t - s |^{\alpha (1 +
    \beta)}  \llbracket A \rrbracket_{\alpha, \beta}^{}  \llbracket x
    \rrbracket^{\beta}_{\alpha, \Delta}\\
    & \leqslant | t - s |^{\alpha} \| A \|_{\alpha, \beta}  (1 + \kappa_1
    \Delta^{\alpha \beta}  \llbracket x \rrbracket^{\beta}_{\alpha, \Delta})\\
    & \leqslant | t - s |^{\alpha} \| A \|_{\alpha, \beta} (1 + \kappa_1
    \Delta^{\alpha \beta} + \kappa_1 \Delta^{\alpha \beta}  \llbracket x
    \rrbracket_{\alpha, \Delta})
  \end{align*}
  
  were we used the trivial inequality $a^{\beta} \leqslant 1 + a$. Dividing
  both sides by $| t - s |^{\alpha}$ and taking the supremum over $| s - t |
  \leqslant \Delta$, we get
  \[ \llbracket x \rrbracket_{\alpha, \Delta} \leqslant \| A \|_{\alpha,
     \beta} (1 + \kappa_1 \Delta^{\alpha \beta}) + \kappa_1 \Delta^{\alpha
     \beta} \| A \|_{\alpha, \beta} \llbracket x \rrbracket_{\alpha, \Delta} .
  \]
  Choosing $\Delta$ small enough such that $\kappa_1 \Delta^{\alpha \beta} \|
  A \|_{\alpha, \beta} \leqslant 1 / 2$, we obtain
  \[ \llbracket x \rrbracket_{\alpha, \Delta} \leqslant 2 \| A \|_{\alpha,
     \beta}  (1 + \kappa_1 \Delta^{\alpha \beta}) \lesssim 1 + \| A
     \|_{\alpha, \beta} . \]
  If we can take $\Delta = T$, we get an estimate for $\llbracket x
  \rrbracket_{\alpha}$, which gives the conclusion. If this is not the case,
  we can choose $\Delta$ such that in addition $\kappa_1 \Delta^{\alpha \beta}
  \| A \|_{\alpha, \beta} \geqslant 1 / 4$ and then as before, by
  Exercise~4.24 from~{\cite{frizhairer}} it holds $\llbracket x
  \rrbracket_{\alpha} \lesssim_T \Delta^{\alpha - 1} \llbracket x
  \rrbracket_{\alpha, \Delta}$, so that
  
  \begin{align*}
    \llbracket x \rrbracket_{\alpha} & \lesssim \, (1 + \| A \|_{\alpha,
    \beta}) \Delta^{\alpha - 1}\\
    & \lesssim \, (1 + \| A \|_{\alpha, \beta}) \| A \|_{\alpha, \beta}^{(1 -
    \alpha) / (\alpha \beta)}\\
    & \lesssim \, 1 + \| A \|_{\alpha, \beta}^2
  \end{align*}
  
  where we used the fact that $\alpha (1 + \beta) > 1$ implies $(1 - \alpha) /
  (\alpha \beta) < 1$. The conclusion follows by the standard inequality $\| x
  \|_{\alpha} \lesssim_T \| x_0 \|_V + \llbracket x \rrbracket_{\alpha}$.
\end{proof}

The assumption of a global bound on $A$ of the form $A \in C^{\alpha}_t
C^{\beta}_V$ is sometimes too strong for practical applications. It can be
relaxed to suitable growth conditions, as the next result shows; it is taken
from~{\cite{hu}}, Theorem~3.1 (see also Theorem~2.9
from~{\cite{catelliergubinelli}}).

\begin{proposition}
  \label{sec3.3 proposition bounds growth condition}Let $A \in C^{\alpha}_t
  C^{\beta, \lambda}_V$ with $\alpha (1 + \beta) > 1$, $\beta + \lambda
  \leqslant 1$. Then there exists a constant $C = C (\alpha, \beta, T)$ such
  that any solution $x$ on $[0, T]$ to the YDE associated to $(x_0, A)$
  satisfies
  \begin{equation}
    \| x \|_{\alpha} \leqslant C \exp \left( \| A \|^{1 + \frac{1 -
    \alpha}{\alpha \beta}}_{\alpha, \beta, \lambda} \right) (1 + \| x_0 \|_V)
    . \label{sec3.3 a priori estimates growth}
  \end{equation}
\end{proposition}

\begin{proof}
  Fix an interval $[s, t] \subset [0, T]$, set $R = \| x \|_{\infty ; s, t}$.
  Since $x$ is a solution, for any $[u, r] \subset [s, t]$ it holds
  \begin{eqnarray*}
    \| x_{u, r} \|_V & \lesssim & \| A_{u, r} (x_u) \|_V + | r - u |^{\alpha
    (1 + \beta)} \llbracket A \rrbracket_{\alpha, \beta, R} \llbracket x
    \rrbracket^{\beta}_{\alpha ; s, t}\\
    & \lesssim & \| A_{u, r} (x_u) - A_{u, r} (x_s) \|_V + | r - u |^{\alpha}
    \| A \|_{\alpha, \beta, \lambda} (1 + \| x_s \|_V)\\
    &  & + | r - u |^{\alpha} | t - s |^{\alpha \beta} \| A \|_{\alpha,
    \beta, \lambda}  (1 + \| x \|_{\infty ; s, t}^{\lambda}) \llbracket x
    \rrbracket^{\beta}_{\alpha ; s, t}\\
    & \lesssim & | r - u |^{\alpha} \| A \|_{\alpha, \beta, \lambda} [1 + \|
    x_s \|_V + | t - s |^{\alpha \beta} (1 + \| x \|_{\infty ; s,
    t}^{\lambda}) \llbracket x \rrbracket^{\beta}_{\alpha ; s, t}]
  \end{eqnarray*}
  which implies, dividing by $| r - u |^{\alpha}$ and taking the supremum,
  that
  
  \begin{align*}
    \llbracket x \rrbracket_{\alpha ; s, t} & \lesssim \| A \|_{\alpha, \beta,
    \lambda} (1 + \| x_s \|_V) + | t - s |^{\alpha \beta} \| A \|_{\alpha,
    \beta, \lambda}  (1 + \| x \|_{\infty ; s, t}^{\lambda}) \llbracket x
    \rrbracket^{\beta}_{\alpha ; s, t} .
  \end{align*}
  
  By an application of Young's inequality, for any $a, b \geqslant 0$ it holds
  $a^{\lambda} b^{\beta} \leqslant a^{\beta + \lambda} + b^{\beta + \lambda}$;
  moreover $\beta + \lambda \leqslant 1$ so that $a^{\beta + \lambda}
  \leqslant 1 + a$ for any $\theta \in [0, 1]$, therefore we obtain
  
  \begin{align*}
    \llbracket x \rrbracket_{\alpha ; s, t} & \lesssim \| A \|_{\alpha, \beta,
    \lambda} (1 + \| x_s \|_V) + | t - s |^{\alpha \beta} \| A \|_{\alpha,
    \beta, \lambda} (1 + \| x \|_{\infty ; s, t} + \llbracket x
    \rrbracket_{\alpha ; s, t})\\
    & \lesssim \| A \|_{\alpha, \beta, \lambda} (1 + \| x_s \|_V) + \| A
    \|_{\alpha, \beta, \lambda} | t - s |^{\alpha \beta} \llbracket x
    \rrbracket_{\alpha ; s, t}
  \end{align*}
  
  where in the second passage we used the estimate $\| x \|_{\infty ; s, t}
  \lesssim_T \| x_s \|_V + \llbracket x \rrbracket_{\alpha ; s, t}$. Overall
  we deduce the existence of a constant $\kappa_1 = \kappa_1 (\alpha, \beta,
  T)$ such that
  \[ \llbracket x \rrbracket_{\alpha ; s, t} \leqslant \frac{\kappa_1}{2} \| A
     \|_{\alpha, \beta, \lambda} (1 + \| x_s \|_V) + \frac{\kappa_1}{2} \| A
     \|_{\alpha, \beta, \lambda} | t - s |^{\alpha \beta} \llbracket x
     \rrbracket_{\alpha ; s, t} . \]
  Choosing $[s, t]$ such that $| t - s | = \Delta$ satisfies $\kappa_1 \| A
  \|_{\alpha, \beta, \lambda} \Delta^{\alpha \beta} \leqslant 1$, we obtain
  \begin{equation}
    \llbracket x \rrbracket_{\alpha ; s, t} \leqslant \kappa_1 \| A
    \|_{\alpha, \beta, \lambda} (1 + \| x_s \|_V) . \label{sec3.3 apriori
    proof eq1}
  \end{equation}
  If $T$ satisfies $\kappa_1 \| A \|_{\alpha, \beta, \lambda} T^{\alpha \beta}
  \leqslant 1$, then we can take $\Delta = T$, which gives a global estimate
  and thus the conclusion. If this is not the case, then we can choose $\Delta
  < T$ s.t. $\kappa_1 \| A \|_{\alpha, \beta, \lambda} \Delta^{\alpha \beta} =
  1$ and~{\eqref{sec3.3 apriori proof eq1}} implies that
  \begin{equation}
    \llbracket x \rrbracket_{\alpha, \Delta} \leqslant \kappa_1 \| A
    \|_{\alpha, \beta, \lambda} (1 + \| x \|_{\infty}) \label{sec3.3 apriori
    proof eq2}
  \end{equation}
  and thus
  \[ \llbracket x \rrbracket_{\alpha} \lesssim \, \Delta^{\alpha - 1}
     \llbracket x \rrbracket_{\alpha, \Delta} \lesssim \| A \|_{\alpha, \beta,
     \lambda}^{\frac{1 - \alpha}{\alpha \beta}} \| A \|_{\alpha, \beta,
     \lambda} (1 + \| x \|_{\infty}) . \]
  Therefore
  \[ \llbracket x \rrbracket_{\alpha} \leqslant \kappa_2 \| A \|_{\alpha,
     \beta, \lambda}^{1 + \frac{1 - \alpha}{\alpha \beta}} (1 + \| x
     \|_{\infty}) \]
  where again $\kappa_2 = \kappa_2 (\alpha, \beta, T)$. In particular, in
  order to obtain the final estimate, we only need to focus on $\| x
  \|_{\infty}$. Let us consider, for $\Delta$ as above, the intervals $I_n
  \assign [(n - 1) \Delta, n \Delta]$ and set $J_n \assign 1 + \| x \|_{\infty
  ; I_n}$, with the convention $J_0 = 1 + \| x_0 \|_V$. Then estimates
  analogue to~{\eqref{sec3.3 apriori proof eq1}} yield
  
  \begin{align*}
    J_n & \leqslant 1 + \| x_{(n - 1) \Delta} \|_V + \Delta^{\alpha}
    \llbracket x \rrbracket_{\alpha ; I_n}\\
    & \leqslant (1 + \kappa_1 \Delta^{\alpha} \| A \|_{\alpha, \beta,
    \lambda}) (1 + \| x_{(n - 1) \Delta} \|_V)\\
    & \leqslant (1 + \kappa_1 \Delta^{\alpha} \| A \|_{\alpha, \beta,
    \lambda}) J_{n - 1}
  \end{align*}
  
  which iteratively implies
  \[ J_n \leqslant [1 + \kappa_1 \Delta^{\alpha} \| A \|_{\alpha, \beta,
     \lambda}]^n J_0 \leqslant \exp (\kappa_1 n \Delta^{\alpha} \| A
     \|_{\alpha, \beta, \lambda})  (1 + \| x_0 \|_V), \]
  where we used the basic inequality $1 + x \leqslant e^x$. Since $[0, T]$ is
  covered by $N \sim T \Delta^{- 1}$ intervals and we chose $\Delta^{- 1} \sim
  \| A \|^{1 / \alpha \beta}$, up to relabelling $\kappa_1$ into a new
  constant $\kappa_3$ we obtain
  
  \begin{align*}
    1 + \| x \|_{\infty} & = \sup_{n \leqslant N} J_n \leqslant \exp \left(
    \kappa_3  \| A \|^{1 + \frac{1 - \alpha}{\alpha \beta}}_{\alpha, \beta,
    \lambda} \right) (1 + \| x_0 \|_V) .
  \end{align*}
  
  Finally, combining this with the estimate for $\llbracket x
  \rrbracket_{\alpha}$ above we obtain
  
  \begin{align*}
    \llbracket x \rrbracket_{\alpha} & \leqslant \kappa_2  \| A \|_{\alpha,
    \beta, \lambda}^{1 + \frac{1 - \alpha}{\alpha \beta}} \exp \left( \kappa_3
    \| A \|^{1 + \frac{1 - \alpha}{\alpha \beta}}_{\alpha, \beta, \lambda}
    \right) (1 + \| x_0 \|_V)\\
    & \leqslant \kappa_4 \exp \left( \kappa_4 \| A \|^{1 + \frac{1 -
    \alpha}{\alpha \beta}}_{\alpha, \beta, \lambda} \right) (1 + \| x_0 \|_V)
  \end{align*}
  
  where we used the inequality $x e^{\lambda x} \leqslant \lambda^{- 1} e^{2
  \lambda x}$. The conclusion follows.
\end{proof}

\begin{remark}
  Since $\alpha (1 + \beta) > 1$, it holds $1 + \| A \|^{1 + (1 - \alpha) /
  (\alpha \beta)}_{\alpha, \beta, \lambda} \lesssim 1 + \| A \|_{\alpha,
  \beta, \lambda}^2$ and so
  \begin{equation}
    \| x \|_{\alpha} \leqslant C \exp (C \| A \|^2_{\alpha, \beta, \lambda})
    (1 + \| x_0 \|_V) \label{sec3.3 a priori estimates growth2}
  \end{equation}
  up to possibly changing constant $C = C (\alpha, \beta, T)$.
  
  The dependence of $C$ on $T$ can be established by a rescaling argument: if
  $x$ is a solution on $[0, T]$ to the YDE associated to $(x_0, A)$, then $x_t
  = \tilde{x}_{t / T}$ where $\tilde{x}$ is a solution on $[0, 1]$ to the YDE
  associated to $(x_0, \tilde{A})$, $\tilde{A} (t, z) = A (T t, z)$. Therefore
  one can apply the estimates to $\tilde{x}$, $\tilde{A}$ and $T = 1$ and then
  write explicitly how $\| x \|_{\alpha}$, $\| A \|_{\alpha, \beta, \lambda}$
  depend on $\| \tilde{x} \|_{\alpha}$, $\| \tilde{A} \|_{\alpha, \beta,
  \lambda}$. The same reasoning applies to several other estimates appearing
  later on, for which the dependence of $C$ on $T$ is not made explicit.
\end{remark}

In classical ODEs, a key role in establishing a priori estimates (as well as
uniqueness) is played by Gronwall's lemma; the following result can be
regarded as a suitable replacement in the Young setting. One of the main cases
of applicability is for $A \in C^{\alpha}_t L (V ; V)$.

\begin{theorem}
  \label{sec3.3 thm gronwall estimate YDE}Let $\alpha > 1 / 2$, $A \in
  C^{\alpha}_t \tmop{Lip}_V$ such that $A (t, 0) = 0$ for all $t \in [0, T]$
  and $h \in C^{\alpha}_t V$. Then there exists a constant $C = C (\alpha)$
  such that any solution $x$ to the YDE
  \begin{equation}
    x_t = x_0 + \int_0^t A (\mathd s, x_s) + h_t \label{sec3.3 affine YDE}
  \end{equation}
  satisfies the a priori bounds
  \begin{equation}
    \llbracket x \rrbracket_{\alpha} \leqslant C (\llbracket A
    \rrbracket_{\alpha, 1} \| x \|_{\infty} + \llbracket h
    \rrbracket_{\alpha}) ; \label{sec3.3 affine YDE bound1}
  \end{equation}
  \begin{equation}
    \| x \|_{\infty} \leqslant C \exp (C T \llbracket A \rrbracket_{\alpha,
    1}^{1 / \alpha} )  (\| x_0 + h_0 \|_V + T^{\alpha} \llbracket h
    \rrbracket_{\alpha}) ; \label{sec3.3 affine YDE bound2}
  \end{equation}
  \begin{equation}
    \| x \|_{\alpha} \leqslant C \exp (C T (1 + \llbracket A
    \rrbracket_{\alpha, 1}^2)) [\| x_0 + h_0 \|_V + (1 + T^{\alpha})
    \llbracket h \rrbracket_{\alpha}] . \label{sec3.3 affine YDE bound3}
  \end{equation}
\end{theorem}

\begin{proof}
  We can assume without loss of generality that $T = 1$, as the general case
  follows by rescaling. It is also clear that, up to changing constant $C$,
  inequality~{\eqref{sec3.3 affine YDE bound3}} follows from combining
  together~{\eqref{sec3.3 affine YDE bound1}} and~{\eqref{sec3.3 affine YDE
  bound2}} and using the fact that $\llbracket A \rrbracket_{\alpha, 1}^{1 /
  \alpha} \lesssim 1 + \llbracket A \rrbracket_{\alpha, 1}^2$ since $\alpha >
  1 / 2$. Up to renaming $x_0$, we can also assume $h_0 = 0$. The proof is
  similar to that of Proposition~\ref{sec3.3 proposition bounds growth
  condition}, but we provide it for the sake of completeness.
  
  Let $\Delta > 0$ to be chosen later, $s < t$ such that $| t - s | \leqslant
  \Delta$, then by~{\eqref{sec3.3 affine YDE}} it holds
  
  \begin{align*}
    \| x_{s, t} \|_V & \leqslant \, \left\| \int_s^t A (\mathd u, x_u)
    \right\|_V + \| h_{s, t} \|_V\\
    & \leqslant \, \| A_{s, t}  (x_s) \|_V + \kappa_1 | t - s |^{2 \alpha}
    \llbracket A \rrbracket_{\alpha, 1}  \llbracket x \rrbracket_{\alpha,
    \Delta} + | t - s |^{\alpha} \llbracket h \rrbracket_{\alpha}\\
    & \leqslant | t - s |^{\alpha} (\llbracket A \rrbracket_{\alpha, 1} \| x
    \|_{\infty} + \llbracket h \rrbracket_{\alpha} + \kappa_1 \Delta^{\alpha}
    \llbracket A \rrbracket_{\alpha, 1} \llbracket x \rrbracket_{\alpha,
    \Delta})
  \end{align*}
  
  and so dividing both sides by $| t - s |^{\alpha}$, taking the supremum over
  $s, t$ and choosing $\Delta$ such that $\kappa_1 \Delta^{\alpha} \llbracket
  A \rrbracket_{\alpha, 1} \leqslant 1 / 2$ we obtain
  \begin{equation}
    \llbracket x \rrbracket_{\alpha, \Delta} \leqslant 2 (\llbracket A
    \rrbracket_{\alpha, 1} \| x \|_{\infty} + \llbracket h
    \rrbracket_{\alpha}) . \label{appendixA1 linear YDE eq1}
  \end{equation}
  As usual, if $\kappa_1  \llbracket A \rrbracket_{\alpha, 1} \leqslant 1 /
  2$, then the conclusion follows from~{\eqref{appendixA1 linear YDE eq1}}
  with the choice $\Delta = 1$ and the trivial estimate $\| x \|_{\infty}
  \leqslant \| x_0 \|_V + \llbracket x \rrbracket_{\alpha}$. Suppose instead
  the opposite, choose $\Delta < 1$ such that $\kappa_1 \Delta^{\alpha}
  \llbracket A \rrbracket_{\alpha, 1} = 1 / 2$; define $I_n = [(n - 1) \Delta,
  n \Delta]$, $J_n = \| x \|_{\infty ; I_n}$, then estimates similar to the
  ones done above show that
  
  \begin{align*}
    J_{n + 1} & \leqslant \, \| x_{n \Delta} \|_V + \Delta^{\alpha} \,
    \llbracket x \rrbracket_{\alpha ; I_n}\\
    & \leqslant \, \| x_{n \Delta} \|_V  (1 + 2 \Delta^{\alpha} \llbracket A
    \rrbracket_{\alpha, 1}) + 2 \llbracket h \rrbracket_{\alpha}\\
    & \lesssim \, J_n + \llbracket h \rrbracket_{\alpha}
  \end{align*}
  
  which implies recursively that for a suitable constant $\kappa_2$ it holds
  $J_n \lesssim e^{\kappa_2 n} (\| x_0 \|_V + \llbracket h
  \rrbracket_{\alpha})$. Since $n \sim \Delta^{- 1} \sim \llbracket A
  \rrbracket_{\alpha, 1}^{1 / \alpha}$ we deduce that
  \[ \| x \|_{\infty} = \sup_n J_n \lesssim \exp (\kappa_3  \llbracket A
     \rrbracket^{1 / \alpha}_{\alpha, 1}) (\| x_0 \|_V + \llbracket h
     \rrbracket_{\alpha}) \]
  which gives~{\eqref{sec3.3 affine YDE bound2}}; combined with $\Delta^{-
  \alpha} \sim \llbracket A \rrbracket_{\alpha, 1}$,
  estimate~{\eqref{appendixA1 linear YDE eq1}} and the basic inequality
  \[ \llbracket x \rrbracket_{\alpha} \lesssim \Delta^{- \alpha} \| x
     \|_{\infty} + \llbracket x \rrbracket_{\alpha, \Delta} \]
  it also yields estimate~{\eqref{sec3.3 affine YDE bound1}}.
\end{proof}

Another way to establish that solutions don't blow-up in finite time is to the
show that the YDE admits (coercive) invariants. The next lemma gives simple
conditions to establish their existence.

\begin{lemma}
  Let $A \in C^{\alpha}_t C^{\beta}_V$ with $\alpha (1 + \beta) > 1$, $x \in
  C^{\alpha}_t V$ be a solution to the YDE associated to $(x_0, A)$ and assume
  $F \in C^2 (V ; \mathbb{R})$ is such that
  \[ D F (z) (A_{s, t} (z)) = 0 \quad \forall \, z \in V, 0 \leqslant s
     \leqslant t \leqslant T. \]
  Then $F$ is constant along $x$, i.e. $F (x_t) = F (x_0)$ for all $t \in [0,
  T]$.
\end{lemma}

\begin{proof}
  It follows immediately from the It\^{o}-type formula~{\eqref{sec2 ito
  formula3}}, since it holds
  \[ F (x_t) - F (x_0) = \int_0^t D F (x_s) (A (\mathd s, x_s)) = \mathcal{J}
     (\Gamma) \]
  for the choice $\Gamma_{s, t} = D F (x_s) (A_{s, t} (x_s)) \equiv 0$ by
  hypothesis.
\end{proof}

\begin{remark}
  If $V$ is an Hilbert space with $\| z \|_V^2 = \langle z, z \rangle_V$, then
  $\| \cdot \|_V$ is constant along solutions of the YDE under the condition
  $\langle z, A_{s, t} (z) \rangle_V = 0$ for all $z \in V$ and $s \leqslant
  t$. In this case, blow up cannot occurr, thus under the hypothesis of
  Corollary~\ref{sec3.1 cor local existence}, global existence of solutions
  holds. Similarly, if in addition $A \in C^{\alpha}_t C^{1 + \beta}_{V,
  \tmop{loc}}$, then by Corollary~\ref{sec3.2 cor local uniqueness} below,
  global existence and uniqueness holds. 
\end{remark}

\subsection{Uniqueness}\label{sec3.3}

We now turn to sufficient conditions for uniqueness of solutions; some of the
results below also establish existence under different sets of assumptions
than those from Section~\ref{sec3.1}.

\begin{theorem}
  \label{sec3 thm wellposedness}Let $A \in C^{\alpha}_t C^{1 + \beta}_V$,
  $\alpha (1 + \beta) > 1$. Then for any $x_0 \in V$ there exists a unique
  global solution to the YDE associated to $(x_0, A)$.
\end{theorem}

\begin{proof}
  The proof is based on an application of Banach fixed point theorem. Let $M$,
  $\tau$ be positive parameters to be fixed later and set
  \[ E : = \left\{ x \in C^{\alpha} ([0, \tau] ; V) \, : \, x (0) = x_0, \,
     \llbracket x \rrbracket_{\alpha} \leqslant M \right\}, \]
  which is complete metric space with the metric $d (x, y) = \llbracket x - y
  \rrbracket_{\alpha}$; define the map $\mathcal{I}$ by
  \[ x \mapsto \mathcal{I} (x)_{\cdot} = x_0 + \int_0^{\cdot} A (\mathd s,
     x_s) . \]

  We want to show that $\mathcal{I}$ is a contraction from $E$ to itself, for
  suitable choice of $M$ and $\tau$. It holds
  
  \begin{align*}
    \| \mathcal{I} (x)_{s, t} \|_V & \leqslant \| A_{s, t} (x_s) \|_V +
    \kappa_1 \llbracket A \rrbracket_{\alpha, 1} \llbracket x
    \rrbracket_{\alpha}  | t - s |^{2 \alpha}\\
    & \leqslant \| A_{s, t} (x_s) - A_{s, t} (x_0) \|_V + \| A_{s, t} (x_0)
    \|_V + \kappa_1 \llbracket A \rrbracket_{\alpha, 1} \llbracket x
    \rrbracket_{\alpha}  | t - s |^{2 \alpha}\\
    & \leqslant \| A \|_{\alpha, 1} \llbracket x \rrbracket_{\alpha}
    s^{\alpha} | t - s |^{\alpha} + \| A \|_{\alpha, 1} | t - s |^{\alpha} +
    \kappa_1 \llbracket A \rrbracket_{\alpha, 1} \llbracket x
    \rrbracket_{\alpha}  | t - s |^{2 \alpha}\\
    & \leqslant \tau^{\alpha} (1 + \kappa_1) \| A \|_{\alpha, 1} \llbracket x
    \rrbracket_{\alpha}  | t - s |^{\alpha} + \| A \|_{\alpha, 1} | t - s
    |^{\alpha} .
  \end{align*}
  
  Choosing $\tau$ and $M$ such that
  \[ \tau^{\alpha} (1 + \kappa_1) \| A \|_{\alpha, 1} \leqslant \frac{1}{2},
     \quad M \geqslant 2 \| A \|_{\alpha, 1}, \]
  for any $x \in V$ it holds
  \[ \| \mathcal{I} (x) \|_{\alpha} \leqslant \tau^{\alpha}  \| A \|_{\alpha,
     1} (1 + \kappa_1) \llbracket x \rrbracket_{\alpha} + \| A \|_{\alpha, 1}
     \leqslant M / 2 + M / 2 \leqslant M \]
  which shows that $\mathcal{I}$ maps $E$ into itself.
  
  By the hypothesis and Corollary~\ref{sec2 corollary frechet}, for any $x, y
  \in V$ it holds
  
  \begin{align*}
    \| \mathcal{I} (x)_{s, t} - \mathcal{I} (y)_{s, t} \|_V & = \left\|
    \int_s^t v_{\mathd u} (x_u - y_u) \right\|_V\\
    & \leqslant \| v_{s, t} (x_s - y_s) \|_V + \kappa_1 \llbracket v
    \rrbracket_{\alpha} \llbracket x - y \rrbracket_{\alpha} | t - s |^{2
    \alpha}\\
    & \leqslant \llbracket v \rrbracket_{\alpha} \llbracket x - y
    \rrbracket_{\alpha} (s^{\alpha} + \kappa_1 | t - s |^{\alpha}) | t - s
    |^{\alpha}\\
    & \leqslant \kappa_2 \| A \|_{\alpha, 1 + \beta}  (1 + \llbracket x
    \rrbracket_{\alpha} + \llbracket y \rrbracket_{\alpha}) \llbracket x - y
    \rrbracket_{\alpha} \tau^{\alpha} | t - s |^{\alpha},
  \end{align*}
  
  which implies
  \[ \llbracket \mathcal{I} (x) - \mathcal{I} (y) \rrbracket_{\alpha}
     \leqslant \kappa_2 \| A \|_{\alpha, 1 + \beta} (1 + 2 M) \tau^{\alpha}
     \llbracket x - y \rrbracket_{\alpha} < \llbracket x - y
     \rrbracket_{\alpha} \]
  as soon as we choose $\tau$ such that $\kappa_2 \| A \|_{\alpha, 1 + \beta}
  (1 + 2 M) \tau^{\alpha} < 1$. Therefore in this case $\mathcal{I}$ is a
  contraction from $E$ to itself; for any $x_0 \in V$ there exists a unique
  solution $x \in C^{\alpha} ([0, \tau] ; V)$ starting from $x_0$. The same
  procedure allows to show existence and uniqueness of solutions $x \in
  C^{\alpha} ([s, s + \tau] \cap [0, T] ; V)$ for any $s \in [0, T]$ and any
  $x_s \in V$, where $\tau$ does not depend on $(s, x_s)$; by iteration,
  global existence and uniqueness follows.
\end{proof}

\begin{corollary}
  \label{sec3.2 cor local uniqueness}Let $A \in C^{\alpha}_t C^{1 + \beta}_{V,
  \tmop{loc}}$, $\alpha (1 + \beta) > 1$. Then for any $x_0 \in V$ there
  exists a unique maximal solution $x$ to the YDE associated to $(x_0, A)$,
  defined on $[0, T^{\ast}) \subset [0, T]$, such that either $T^{\ast} = T$
  or
  \[ \lim_{t \rightarrow T^{\ast}} \| x_t \|_V = + \infty . \]
  In particular if $A \in C^{\alpha}_t C^{\beta, \lambda}_V \cap C^{\alpha}_t
  C^{1 + \beta}_{V, \tmop{loc}}$ with $\alpha (1 + \beta) > 1$, $\beta +
  \lambda \leqslant 1$, then global existence and uniqueness holds.
\end{corollary}

\begin{proof}
  We only sketch the proof, as it follows from classical ODE arguments and is
  similar to that of Corollary~\ref{sec3.1 cor local existence}.
  
  By localization, given any $s \in [0, T)$ and any $x_s \in V$, there exists
  $\tau = \tau (s, x_s)$ such that there exists a unique solution to the YDE
  associated to $(x_s, A)$ on the interval $[s, s + \tau]$. Therefore given
  two solutions $x^i$ defined on intervals $[s, T_i]$ with $x_s^1 = x_s^2$,
  they must coincide on $[s, T_1 \wedge T_2]$; \ in particular, any extension
  procedure of a given solution to a larger interval is consistent, which
  allows to define the maximal solution as the maximal extension of any
  solution starting from $x_0$ at $t = 0$.
  
  The blow-up alternative can be established reasoning by contradiction as in
  Corollary~\ref{sec3.1 cor local existence}. If $A \in C^{\alpha}_t C^{\beta,
  \lambda}_V$, then by the a priori estimate~{\eqref{sec3.3 a priori estimates
  growth}} blow-up cannot occur and so global well-posedness follows.
\end{proof}

Once existence of solutions is established, their uniqueness can be
alternatively shows by means of a Comparison Principle, which is the analogue
of a Gronwall type estimate for classical ODEs. Such results are of
independent interest as they also allow to compare solutions to different
YDEs; they were first introduced in~{\cite{catelliergubinelli}} and later
revisited in~{\cite{galeatigubinelli1}}.

\begin{theorem}
  \label{sec3.3 comparison principle}Let $R, M > 0$ fixed. For $i = 1, 2$, let
  $x_0^i \in V$ such that $\| x^i_0 \|_V \leqslant R$, $A^i \in C^{\alpha}_t
  C^{\beta, \lambda}_V$ with $\alpha (1 + \beta) > 1$, $\beta + \lambda
  \leqslant 1$ and $\| A^i \|_{\alpha, \beta, \lambda} \leqslant M$, as well
  as $A^1 \in C^{\alpha}_t C^{1 + \beta, \lambda}_V$ with $\| A^1 \|_{\alpha,
  1 + \beta, \lambda} \leqslant M$; let $x^i$ be two given solutions
  associated respectively to $(x_0^i, A^i)$. Then it holds
  \[ \llbracket x^1 - x^2 \rrbracket_{\alpha} \leqslant C (\| x^1_0 - x^2_0
     \|_V + \| A^1 - A^2 \|_{\alpha, \beta, \lambda}) \]
  for a constant $C = C (\alpha, \beta, T, R, M)$ increasing in the last two
  variables.
\end{theorem}

\begin{proof}
  Let $x^i$ be the two given solutions and set $e_t \assign x_t^1 - x_t^2$,
  then $e$ satisfies
  
  \begin{align*}
    e_t & = \, e_0 + \int_0^t A^1 (\mathd s, x^1_s) - \int_0^t A^2 (\mathd s,
    x_s^2)\\
    & = \, e_0 + \int_0^t A^1 (\mathd s, x^1_s) - \int_0^t A^1 (\mathd s,
    x_s^2) + \int_0^t (A^1 - A^2) (\mathd s, x^2_s)\\
    & = \, e_0 + \int_0^t v_{\mathd s} (e_s) + \psi_t
  \end{align*}
  
  for the choice
  \[ v_t : = \int_0^t \int_0^1 D A^1 (\mathd s, x^2_s + \lambda (x^1_s -
     x^2_s)) \mathd \lambda, \quad \psi_t \assign \int_0^t (A^1 - A^2) (\mathd
     s, x^2_s) \]
  where we applied Corollary~\ref{sec2 corollary frechet}. By the same result,
  combined with estimate~{\eqref{sec3.3 a priori estimates growth2}}, it holds
  
  \begin{align*}
    \llbracket v \rrbracket_{\alpha, 1} & \leqslant \kappa_1 \| D A^1
    \|_{\alpha, \beta, \lambda} (1 + \| x^1 \|_{\alpha} + \| x^2
    \|_{\alpha})\\
    & \leqslant \kappa_2 \exp (\kappa_2 (\| A^1 \|_{\alpha, 1 + \beta,
    \lambda}^2 + \| A^2 \|_{\alpha, \beta, \lambda}^2)) (1 + R)\\
    & \leqslant \kappa_2 \exp (2 \kappa_2 M^2) (1 + R) ;
  \end{align*}
  
  similarly, by Point~4. of Theorem~\ref{sec2 thm definition young integral},
  
  \begin{align*}
    \llbracket \psi \rrbracket_{\alpha} & \leqslant \kappa_3  \| A^1 - A^2
    \|_{\alpha, \beta, \lambda} (1 + \| x^2 \|_{\infty}^{\lambda}) (1 +
    \llbracket x^2 \rrbracket_{\alpha})\\
    & \leqslant \kappa_4  \| A^1 - A^2 \|_{\alpha, \beta, \lambda} \exp
    (\kappa_4 (1 + M^2)) (1 + R) .
  \end{align*}
  
  Applying Theorem~\ref{sec3.3 thm gronwall estimate YDE} to $e$, we have
  \[ \llbracket x^1 - x^2 \rrbracket_{\alpha} \leqslant \kappa_5 e^{\kappa_5 
     \llbracket v \rrbracket^2_{\alpha, 1} }  (\| x_0^1 - x_0^2 \|_V +
     \llbracket \psi \rrbracket_{\alpha}) \]
  which combined with the previous estimates implies the conclusion.
\end{proof}

\begin{remark}
  \label{sec3.3 remark lipschitz solution map}If $A \in C^{\alpha}_t C^{1 +
  \beta}_V$ and we consider solutions $x^i$ associated to $(x_0^i, A)$, going
  through the same proof but applying instead estimate~{\eqref{sec3.3 a priori
  estimate1}}, we obtain
  
  \begin{align*}
    \llbracket v \rrbracket_{\alpha, 1} & \lesssim \| D A \|_{\alpha, \beta}
    (1 + \| x^1 \|_{\alpha} + \| x^2 \|_{\alpha}) \lesssim 1 + \| A
    \|_{\alpha, 1 + \beta}^3
  \end{align*}
  
  which combined with~{\eqref{sec3.3 affine YDE bound3}} implies the existence
  of a constant $C = C (\alpha, \beta, T)$ such that
  \begin{equation}
    \llbracket x^1 - x^2 \rrbracket_{\alpha} \leqslant C \exp (C \| A
    \|^6_{\alpha, 1 + \beta}) \, \| x_0^1 - x_0^2 \|_V . \label{sec3.3 eq
    lipschitz map}
  \end{equation}
  As a consequence, the solution map $F [A] : x_0 \mapsto x$ associated to
  $A$, seen as a map from $V$ to $C^{\alpha}_t V$, is globally Lipschitz.
  Similar estimates show that, if $\{ A_n \}_n$ is a sequence such that $A_n
  \rightarrow A$ in $C^{\alpha}_t C^{1 + \beta}_V$, then $F [A_n] \rightarrow
  F [A]$ uniformly on bounded sets.
\end{remark}

As a corollary, we obtain convergence of the Euler scheme introduced in
Section~\ref{sec3.1}, with rate $2 \alpha - 1$. For simplicity we state the
result in the case $A \in C^{\alpha}_t C^{1 + \beta}_V$, but the same results
follow for $A \in C^{\alpha}_t C^{1 + \beta, \lambda}_V$ by the usual
localization procedure.

\begin{corollary}
  \label{sec3.3 cor euler convergence}Given $A \in C^{\alpha}_t C^{1 +
  \beta}_V$ with $\alpha (1 + \beta) > 1$ and $x_0 \in V$, denote by $x^n$ the
  element of $C^{\alpha}_t V$ constructed by the $n$-step Euler approximation
  from Theorem~\ref{sec3.1 thm existence}, and by $x$ the unique solution
  associated to $(x_0, A)$. Then there exists a constant $C = C (\alpha,
  \beta, T)$ such that
  \[ \| x - x^n \|_{\alpha} \leqslant C \exp (C \| A \|_{\alpha, 1 + \beta}^6)
     n^{1 - 2 \alpha} \quad \text{as } n \rightarrow \infty . \]
\end{corollary}

\begin{proof}
  Recall that by Theorem~\ref{sec3.1 thm existence}, $x^n$ satisfies the YDE
  \[ x^n_t = x_0 + \int_0^t A (\mathd s, x^n_s) + \psi^n_t, \]
  where by Remark~\ref{sec3.1 remark euler a priori estimates}, for the choice
  $\beta = 1$, it holds
  \[ \llbracket \psi^n \rrbracket \lesssim \, (1 + \| A \|_{\alpha, 1}^{1 + 1
     / \alpha}) n^{1 - 2 \alpha} . \]
  Define $e^n \assign x - x^n$, then by Corollary~\ref{sec2 corollary frechet}
  it satisfies
  \[ e^n_t = \int_0^t A (\mathd s, x^n_s) - A (\mathd s, x_s) + \psi^n_t =
     \int_0^t v^n_{\mathd s} (e^n_s) + \psi^n_t \]
  where again by Remark~\ref{sec3.1 remark euler a priori estimates} it holds
  \[ \llbracket v^n \rrbracket_{\alpha, 1} \lesssim \| A \|_{\alpha, 1 +
     \beta}  (1 + \llbracket x \rrbracket_{\alpha} + \llbracket x^n
     \rrbracket_{\alpha}) \lesssim 1 + \| A \|_{\alpha, 1 + \beta}^3 . \]
  Applying Theorem~\ref{sec3.3 thm gronwall estimate YDE}, we deduce the
  existence of $\kappa_1 = \kappa_1 (\alpha, \beta, T)$ such that
  \[ \| e^n \|_{\alpha} \leqslant \kappa_1 \exp (\kappa_1 \| A \|_{\alpha, 1 +
     \beta}^6)  \llbracket \psi^n \rrbracket_{\alpha}, \]
  which combined with the estimate for $\llbracket \psi^n \rrbracket_{\alpha}$
  yields the conclusion.
\end{proof}

\subsection{The case of continuous $\partial_t A$}\label{sec3.4}

In this section we study how the well-posedness theory changes when, in
addition to the regularity condition $A \in C^{\alpha}_t C^{\beta}_t$, we
impose $\partial_t A : [0, T] \times V \rightarrow V$ to exist continuous and
uniformly bounded (we assume boundedness for simplicity, but it could be
replaced by a growth condition).

The key point is that, by Point~2. from Theorem~\ref{sec2 thm definition young
integral}, any solution to the YDE is also a solution to the classical ODE
associated to $\partial_t A$; as such, it is Lipschitz continuous with
constant $\| \partial_t A \|_{\infty}$. We can exploit this additional time
regularity, combined with nonlinear Young theory, to obtain well-posedness
under weaker conditions than those from Theorem~\ref{sec3 thm wellposedness}.

While the existence of $\partial_t A$ is not a very meaningful requirement for
classical YDEs, i.e. for $A (t, x) = f (x) y_t$, as it would imply that $y \in
C^1_t$, there are other situations in which it becomes a natural assumption.
One example is for perturbed ODEs $\dot{x} = b (x) + \dot{w}$, in which the
associated $A$ is the averaged field
\[ A (t, x) = \int_0^t b (s, x + w_s) \mathd s \]
for which $\partial_t A$ exists continuous as soon as $b$ is continuous field;
still classical wellposedness is not is not guaranteed under the sole
continuity of $b$.

\begin{theorem}
  \label{sec3.4 thm wellposedness}Let $A$ be such that $A \in C^{\alpha}_t
  C^{1 + \beta}_V$ and $\partial_t A \in C_b ([0, T] \times V ; V)$ with
  $\alpha + \beta > 1$. Then for any $x_0 \in V$ there exists a unique global
  solution to the YDE associated to $(x_0, A)$.
\end{theorem}

\begin{proof}
  Similarly to Theorem~\ref{sec3 thm wellposedness}, the proof is by Banach
  fixed point theorem. For suitable values of $M, \tau > 0$ to be fixed later,
  consider the space $E \assign \left\{ x \in \tmop{Lip} ([0, \tau] ; V) \, :
  x (0) = x_0, \, \llbracket x \rrbracket_{\tmop{Lip}} \leqslant M \right\}$;
  it is a complete metric space with the metric $d (x, y) = \llbracket x - y
  \rrbracket_{\gamma}$ (the condition $\llbracket x \rrbracket_{\tmop{Lip}}
  \leqslant M$ is essential for this to be true). Define the map $\mathcal{I}$
  by
  \[ \mathcal{I} (x)_t = x_0 + \int_0^t \partial_t A (s, x_s) \mathd s = x_0 +
     \int_0^t A (\tmop{ds}, x_s) \]
  and observe that under the condition $\| \partial_t A \|_{\infty} \leqslant
  M$ it maps $E$ into itself. By the hypothesis and Corollary~\ref{sec2
  corollary frechet}, for any $x, y \in E$ it holds
  
  \begin{align*}
    \| \mathcal{I} (x)_{s, t} - \mathcal{I} (y)_{s, t} \|_V & = \left\|
    \int_s^t v_{\mathd u} (x_u - y_u) \right\|_V\\
    & \leqslant \| v_{s, t} (x_s - y_s) \|_V + \kappa_1  \llbracket v
    \rrbracket_{\alpha} \llbracket x - y \rrbracket_{\tmop{Lip}} | t - s |^{2
    \alpha}\\
    & \leqslant \llbracket v \rrbracket_{\alpha} \llbracket x - y
    \rrbracket_{\alpha} (s^{\alpha} + \kappa_1 | t - s |^{\alpha}) | t - s
    |^{\alpha}\\
    & \leqslant \kappa_2 \tau^{\alpha} \| A \|_{\alpha, 1 + \beta}  (1 +
    \llbracket x \rrbracket_{\tmop{Lip}} + \llbracket y
    \rrbracket_{\tmop{Lip}}) \llbracket x - y \rrbracket_{\alpha} | t - s
    |^{\alpha}
  \end{align*}
  
  which implies
  \[ \llbracket \mathcal{I} (x) - \mathcal{I} (y) \rrbracket_{\alpha}
     \leqslant \kappa_2 \tau^{\alpha} \| A \|_{\alpha, 1 + \beta} (1 + 2 M)
     \llbracket x - y \rrbracket_{\alpha} < \llbracket x - y
     \rrbracket_{\alpha} \]
  as soon as we choose $\tau$ small enough such that $\kappa_2 \tau^{\alpha}
  \| A \|_{\alpha, 1 + \beta} (1 + 2 M) < 1$. Therefore $\mathcal{I}$ is a
  contraction on $E$ and for any $x_0 \in V$ there exists a unique associated
  solution $x \in C^{\gamma} ([0, \tau] ; V)$. Global existence and uniqueness
  then follows from the usual iterative argument.
\end{proof}

We can also establish an analogue of Theorem~\ref{sec3.3 comparison principle}
in this setting.

\begin{theorem}
  \label{sec3.4 thm comparison principle}Let $M > 0$ fixed. For $i = 1, 2$,
  let $A^i \in C^{\alpha}_t C^{\beta}_V$ such that $\partial_t A^i \in C^0
  ([0, T] \times V ; V)$, $\alpha + \beta > 1$ and $\| A^i \|_{\alpha, \beta}
  + \| \partial_t A \|_{\infty} \leqslant M$, as well as $A^1 \in C^{\alpha}_t
  C^{1 + \beta}_V$ with $\| A^1 \|_{\alpha, 1 + \beta} \leqslant M$, and
  $x_0^i \in V$; let $x^i$ be two given solutions associated respectively to
  $(x_0^i, A^i)$. Then it holds
  \[ \llbracket x^1 - x^2 \rrbracket_{\alpha} \leqslant C (\| x^1_0 - x^2_0
     \|_V + \| A^1 - A^2 \|_{\alpha, \beta}) \]
  for a constant $C = C (\alpha, \beta, T, M)$ increasing in the last
  variable. A more explicit formula for $C$ is given by~{\eqref{sec3.4
  constant cor1}}. 
\end{theorem}

\begin{proof}
  The proof is analogous to that of Theorem~\ref{sec3.3 comparison principle},
  so we will mostly sketch it; it is based on an application of
  Corollary~\ref{sec2 corollary frechet} and Theorem~\ref{sec3.3 thm gronwall
  estimate YDE}.
  
  Given two solutions as above, their difference $e = x^1 - x^2$ satisfies the
  affine YDE
  \[ e_t = e_0 + \int_0^t v_{\mathd s} e_s + \psi_t \]
  with
  \[ v_t = \int_0^t \int_0^1 D A^1 (\mathd s, x^2_s + \lambda e_s) \mathd
     \lambda, \quad \psi_t = \int_0^t (A^1 - A^2) (\mathd s, x^2_s) . \]
  We have the estimates
  \begin{eqnarray*}
    \| v \|_{\alpha, 1} & \lesssim_{\alpha, \beta, T} & \| A^1 \|_{\alpha, 1 +
    \beta} (1 + \llbracket x^1 \rrbracket_{\tmop{Lip}} + \llbracket x^2
    \rrbracket_{\tmop{Lip}}) \lesssim \| A^1 \|_{\alpha, 1 + \beta} (1 + \|
    \partial_t A^1 \|_{\infty} + \| \partial_t A^2 \|_{\infty})\\
    \| \psi_t \|_{\alpha} & \lesssim_{\alpha, \beta, T} & \| A^1 - A^2
    \|_{\alpha, \beta} (1 + \llbracket x^2 \rrbracket_{\tmop{Lip}}) \lesssim
    \| A^1 - A^2 \|_{\alpha, \beta} (1 + \| \partial_t A^2 \|_{\infty})
  \end{eqnarray*}
  which, combined with Theorem~\ref{sec3.3 thm gronwall estimate YDE}, yield
  
  \begin{align*}
    \| e \|_{\alpha} & \leqslant \kappa_1 e^{\kappa_1 (1 + \| A^1
    \|^2_{\alpha, 1 + \beta}) (1 + \| \partial_t A^1 \|_{\infty}^2 + \|
    \partial_t A^2 \|_{\infty}^2)} (\| e_0 \|_V + \| A^1 - A^2 \|_{\alpha,
    \beta} (1 + \| \partial_t A^2 \|_{\infty}))\\
    & \leqslant \kappa_2 e^{\kappa_2 (1 + \| A^1 \|^2_{\alpha, 1 + \beta}) (1
    + \| \partial_t A^1 \|_{\infty}^2 + \| \partial_t A^2 \|_{\infty}^2)} (\|
    e_0 \|_V + \| A^1 - A^2 \|_{\alpha, \beta})
  \end{align*}
  
  for some $\kappa_2 = \kappa_2 (\alpha, \beta, T)$. In particular, $C$ can be
  taken of the form
  \begin{equation}
    C (\alpha, \beta, T, M) = \kappa_3 (\alpha, \beta, T) \exp (\kappa_3
    (\alpha, \beta, T) (1 + M^4)) . \label{sec3.4 constant cor1}
  \end{equation}
\end{proof}

\begin{corollary}
  \label{sec3.4 cor euler convergence}Given $A$ as in Theorem~\ref{sec3.4 thm
  wellposedness}, denote by $x^n$ the element of $C^{\alpha}_t V$ constructed
  by the $n$-step Euler approximation from Theorem~\ref{sec3.1 thm existence}
  and by $x$ the solution associated to $(x_0, A)$. Then there exists a
  constant $C = C (\alpha, \beta, T, \| A \|_{\alpha, 1 + \beta}, \|
  \partial_t A \|_{\infty})$ such that
  \[ \| x - x^n \|_{\alpha} \leqslant C n^{- \alpha} \quad \text{as } n
     \rightarrow \infty . \]
  A more explicit formula for $C$ is given by~{\eqref{sec3.4 constant cor2}}. 
\end{corollary}

\begin{proof}
  By Theorem~\ref{sec3.1 thm existence}, $x^n$ satisfies the YDE
  \[ x^n = x_0 + \int_0^t A (\mathd s, x^n_s) + \psi^n_t = x_0 + \int_0^t A^n
     (\mathd s, x_s^n) \]
  where $A^n (t, z) : = A (t, z) + \psi^n_t$ and that by
  estimate~{\eqref{sec3.1 euler a priori estimate2}}, for the choice $\Delta =
  T$, $\beta = \gamma = 1$, we have
  
  \begin{align*}
    \llbracket \psi^n \rrbracket_{\alpha} & \lesssim_{\alpha, T} \| A
    \|_{\alpha, 1} \llbracket x^n \rrbracket_{\tmop{Lip}} n^{- \alpha}
    \lesssim \| A \|_{\alpha, 1} \| \partial_t A \|_{\infty} n^{- \alpha} .
  \end{align*}
  
  Defining $e^n \assign x - x^n$, by the basic estimates $\| A - A^n
  \|_{\alpha, \beta} \lesssim_T \llbracket \psi^n \rrbracket_{\alpha}$ and $\|
  \partial_t A^n \|_{\infty} \lesssim \| \partial_t A \|_{\infty}$, going
  through the same proof as in Theorem~\ref{sec3.4 thm comparison principle}
  we deduce that
  \[ \| e^n \|_{\alpha} \leqslant \kappa_1 e^{\kappa_1 (1 + \| A \|_{\alpha,
     1}^2) (1 + \| \partial_t A \|_{\infty}^2)}  \| A - A^n \|_{\alpha, \beta}
  \]
  and so finally that, for a suitable constant $\kappa_2 = \kappa_2 (\alpha,
  T)$, it holds
  \begin{equation}
    \| e^n \|_{\alpha} \leqslant \kappa_2 \exp (\kappa_2 (1 + \| A \|_{\alpha,
    1}^2) (1 + \| \partial_t A \|_{\infty}^2)) n^{- \alpha} . \label{sec3.4
    constant cor2}
  \end{equation}
\end{proof}

\subsection{Further variants}\label{sec3.5}

Several other kinds of differential equations involving a nonlinear Young
integral term can be studied. In this section we focus on two cases: nonlinear
YDEs involving a classical drift term and fractional YDEs.

\subsubsection{Mixed equations}

Let us consider now an equation of the form
\begin{equation}
  x_t = x_0 + \int_0^t F (s, x_s) \mathd s + \int_0^t A (\mathd s, x_s) .
  \label{sec3.5 mixed YDE}
\end{equation}
where $F : [0, T] \times V \rightarrow V$ is continuous function; the first
integral is meaningful as a classical one.

\begin{proposition}
  \label{sec3.5.1 baby result}Let $A \in C^{\alpha}_t C^{1 + \beta}_V$ with
  $\alpha (1 + \beta) > 1$, $F$ be bounded and globally Lipschitz, namely
  \[ \| F (t, y) \|_V \leqslant C_F, \qquad \| F (t, y) - F (t, z) \|_V
     \leqslant C_F \| y - z \|_V \qquad \text{for all } \, t \in [0, T], \, y,
     z \in V \]
  for some constant $C_F > 0$. Then global well-posedness holds
  for~{\eqref{sec3.5 mixed YDE}}.
\end{proposition}

\begin{proof}
  For simplicity we will use the notation $\| A \| = \| A \|_{\alpha, 1 +
  \beta}$; the proof is analogue to that of Theorem~\ref{sec3 thm
  wellposedness}. Let $M$, $\tau$ be positive parameters to be fixed later and
  define as usual
  \[ E = \left\{ x \in C^{\alpha} ([0, \tau] ; V) : \, x (0) = x_0, \,
     \llbracket x \rrbracket_{\alpha} \leqslant M \right\} . \]
  A path $x$ solves~{\eqref{sec3.5 mixed YDE}} if and only if it belongs to
  $E$ and is a fixed point for the map
  \[ x \mapsto \mathcal{I} (x)_{\cdot} = x_0 + \int_0^{\cdot} F (s, x_s) +
     \int_0^{\cdot} A (\mathd s, x_s) . \]
  We have the estimates
  
  \begin{align*}
    \| \mathcal{I} (x)_{s, t} \|_V & \leqslant \int_s^t \| F (r, x_r) \|_V
    \mathd r + \| A_{s, t} (x_s) \|_V + \kappa_1  | t - s |^{2 \alpha} \| A \|
    \llbracket x \rrbracket_{\alpha}\\
    & \leqslant | t - s | C_F + \| A_{s, t} (x_s) - A_{s, t} (x_0) \|_V + \|
    A_{s, t} (x_0) \|_V + \kappa_1  | t - s |^{2 \alpha} \| A \| \llbracket x
    \rrbracket_{\alpha}\\
    & \leqslant | t - s |^{\alpha}  [C_F \tau^{1 - \alpha} + \| A \|
    \tau^{\alpha} M + \| A \| + \kappa_1 \| A \| \tau^{\alpha} M],
  \end{align*}
  
  which imply
  
  \begin{align*}
    \llbracket \mathcal{I} (x) \rrbracket_{\alpha} & \leqslant C_F \tau^{1 -
    \alpha} + \| A \| + [\tau + \| A \| (1 + \kappa_1) \tau^{\alpha}] M.
  \end{align*}
  
  In order for $\mathcal{I}$ to map $E$ into itself, it suffices to choose
  $\tau$ and $M$ such that
  \[ \tau \leqslant 1, \quad \tau + \| A \| (1 + \kappa_1) \tau^{\alpha}
     \leqslant 1 / 2, \quad M \geqslant 2 (C_F + \| A \|) . \]
  Next we check contractivity of $\mathcal{I}$; given $x, y \in E$, it holds
  
  \begin{align*}
    \| \mathcal{I} (x)_{s, t} - \mathcal{I} (y)_{s, t} \|_V & \leqslant
    \int_s^t \| F (r, x_r) - F (r, y_r) \|_V \, \mathd r + \left\| \int_s^t
    v_{\mathd r} (x_r - y_r) \right\|_V\\
    & \leqslant C_F | t - s | \tau^{\alpha} \llbracket x - y
    \rrbracket_{\alpha} + \| v_{s, t} (x_s - y_s) \|_V + \kappa_2  | t - s
    |^{2 \alpha} \llbracket v \rrbracket_{\alpha} \llbracket x - y
    \rrbracket_{\alpha}\\
    & \leqslant \kappa_3 \tau^{\alpha} [C_F + \| A \| (1 + \llbracket x
    \rrbracket_{\alpha} + \llbracket y \rrbracket_{\alpha})]  \llbracket x - y
    \rrbracket_{\alpha} | t - s |^{\alpha}
  \end{align*}
  
  which implies
  \[ \llbracket \mathcal{I} (x) - \mathcal{I} (y) \rrbracket_{\alpha}
     \leqslant \kappa_3 \tau^{\alpha} [C_F + \| A \| (1 + 2 M)] \]
  thus choosing $\tau$ small enough we deduce contractivity. Therefore
  existence and uniqueness of solutions holds on the interval $[0, \tau]$; as
  the choice of $\tau$ does not depend on $x_0$, we can iterate the reasoning
  to cover the whole interval $[0, T]$.
\end{proof}

\begin{theorem}
  \label{sec3.5.1 main result}Let $A \in C^{\alpha}_t C^{1 + \beta}_{V,
  \tmop{loc}}$ with $\alpha (1 + \beta) > 1$ and $F$ be a continuous locally
  Lipschitz function, in the sense that for any $R > 0$ there exist a constant
  $C_R$ such that
  \[ \| F (t, y) - F (t, z) \|_V \leqslant C_R  \| y - z \|_V \quad \text{for
     all } \, t \in [0, T]  \text{ and } y, z \in V \text{ such that } \| y
     \|_V, \| z \|_V \leqslant R. \]
  Then for any $x_0 \in V$ there exists a unique maximal solution $x$
  to~{\eqref{sec3.5 mixed YDE}}, defined on $[0, T^{\ast}) \subset [0, T]$
  such that either $T = T^{\ast}$ or
  \[ \lim_{t \rightarrow T^{\ast}} \| x_t \|_V = + \infty . \]
  If in addition $A \in C^{\alpha}_t C^{\beta, \lambda}_V$ with $\beta +
  \lambda \leqslant 1$ and $F$ has at most linear growth, i.e. there exists
  $C_F > 0$ s.t.
  \[ \| F (t, z) \|_V \leqslant C_F (1 + \| z \|_V) \quad \forall \, (t, z)
     \in [0, T] \times V, \]
  then global wellposedness holds. Moreover in this case there exists $C = C
  (\alpha, \beta, T)$ such that, setting $\theta = 1 + \frac{1 -
  \alpha}{\alpha \beta}$, any solution to~{\eqref{sec3.5 mixed YDE}} satisfies
  the a priori estimate
  \begin{equation}
    \| x \|_{\alpha} \leqslant C \exp (C (C^{\theta}_F + \| A \|_{\alpha,
    \beta, \lambda}^{\theta}))  (1 + \| x_0 \|_V) . \label{sec3.5 a priori
    estimate 1}
  \end{equation}
\end{theorem}

\begin{proof}
  The first part of the statement, regarding local wellposedness and the
  blow-up alternative, follows from the usual localisation arguments, so we
  omit its proof.
  
  The proof of a priori estimate~{\eqref{sec3.5 a priori estimate 1}} is
  analogue to that of Proposition~\ref{sec3.3 proposition bounds growth
  condition}, so we will mostly sketch it; as before $\| A \| = \| A
  \|_{\alpha, \beta, \lambda}$ for simplicity. Let $x$ be a solution
  to~{\eqref{sec3.5 mixed YDE}} defined on $[0, T^{\ast})$, then for any $[r,
  u] \subset [s, t] \subset [0, T^{\ast})$ it holds
  
  \begin{align*}
    \left\| \int_u^r F (a, x_a) \mathd a \right\|_V & \leqslant C_F | r - u |
    + C_F \int_u^r \| x_a \| \mathd a\\
    & \leqslant | r - u | C_F (1 + \| x_s \|_V) + | r - u | | t - s
    |^{\alpha} C_F \llbracket x \rrbracket_{\alpha ; s, t}\\
    & \lesssim | r - u |^{\alpha} C_F [1 + \| x_s \|_V + | t - s | \llbracket
    x \rrbracket_{\alpha ; s, t}] .
  \end{align*}
  
  Together with the estimates from the proof of Proposition~\ref{sec3.3
  proposition bounds growth condition} and the fact that $| t - s | \lesssim |
  t - s |^{\alpha \beta}$, this implies the existence of $\kappa_1 = \kappa_1
  (\alpha, \beta, T)$ such that any solution $x$ to~{\eqref{sec3.5 mixed YDE}}
  satisfies
  \[ \llbracket x \rrbracket_{\alpha ; s, t} \leqslant \frac{\kappa_1}{2} (C_F
     + \| A \|) (1 + \| x_s \|_V) + \frac{\kappa_1}{2} (C_F + \| A \|) | t - s
     |^{\alpha \beta} \llbracket x \rrbracket_{\alpha ; s, t} . \]
  The rest of the proof is identical, up to replacing $\| A \|$ with $C_F + \|
  A \|$ in all the passages. Specifically, if $T$ is such that $\kappa_1 (C_F
  + \| A \|) T^{\alpha \beta} < 2$, then we obtain a global estimate by
  choosing $s = 0$, $t = T$, which shows that $T^{\ast} = T$ and gives the
  conclusion in this case. Otherwise, taking $\Delta < T$ such that $\kappa_1
  (C_F + \| A \|) \Delta^{\alpha \beta} = 1$ and defining $J_n$ as before, we
  obtain the recurrent estimate
  \[ J_n \leqslant [1 + \kappa_1 \Delta^{\alpha} (C_F + \| A \|)] J_{n - 1} \]
  and going through the same reasoning the conclusion follows.
\end{proof}

\subsubsection{Fractional Young equations}

We restrict in this subsection to the finite dimensional case $V
=\mathbb{R}^d$ for some $d \in \mathbb{N}$; as usual we work on a finite time
interval $[0, T]$. We are interested in studying a fractional type of equation
of the form
\begin{equation}
  D_{0 +}^{\delta} x_t = A (\mathd t, x_t) \quad \forall \, t \in [0, T]
  \label{sec3.5 prototype YDE}
\end{equation}
for a suitable parameter $\delta \in (0, 1)$. Here $D^{\delta}_{0 +}$ denotes
a Riemann--Liouville type of fractional derivative on $[0, T]$; for more
details on fractional derivatives and fractional calculus we refer the reader
to~{\cite{samko}}. In the case $\delta = 1$, formally $D^{\delta} x_s = \mathd
x_s$ and we recover the class of YDEs studied so far.

In order to study~{\eqref{sec3.5 prototype YDE}}, it is more convenient to
write it in integral form, using the fact that $D^{\delta}_{0 +}$ is the
inverse operator of the fractional integral $I^{\delta}_{0 +}$ given by
\[ (I^{\delta}_{0 +} f)_t = \frac{1}{\Gamma (\delta)} \int_0^{\delta} (t -
   s)^{\delta - 1} f_s \mathd s \]
(being interpreted componentwise if $f : [0, T] \rightarrow \mathbb{R}^d$).
From now on we will for simplicity drop the constant $1 / \Gamma (\delta)$,
which can be incorporated in the drift $A$. We need the following lemma.

\begin{lemma}
  \label{sec3.5 lemma fractional}For $\delta \in (0, 1)$, consider the
  functional $\Xi$ defined for smooth $f$ by
  \[ \Xi [f]_t : = (I^{\delta}_{0 +} \dot{f})_t = \int_0^t (t - s)^{\delta -
     1} \dot{f}_s \mathd s. \]
  For any $\alpha \in (0, 1)$ such that $\alpha + \delta > 1$ and any
  $\varepsilon > 0$, $\Xi$ extends uniquely to a continuous linear map from
  $C^{\alpha} ([0, T] ; \mathbb{R}^d)$ to $C^{\alpha + \delta - 1 -
  \varepsilon} ([0, T] ; \mathbb{R}^d)$; in particular, there exists $C = C
  (\alpha, \delta, \varepsilon, T)$, which will be denoted by $\| \Xi \|$,
  such that
  \begin{equation}
    \| \Xi [f] \|_{\alpha + \delta - 1 - \varepsilon} \leqslant \| \Xi \| 
    \llbracket f \rrbracket_{\alpha} \quad \text{for all } f \in C^{\alpha}
    ([0, T] ; \mathbb{R}^d) . \label{sec3.5 continuity fractional}
  \end{equation}
\end{lemma}

\begin{proof}
  Up to multiplicative constant, $\Xi = I^{\alpha}_{0 +} D$. Recall that
  fractional integrals and fractional derivatives, on their domain of
  definition, satisfy the following properties, for $\alpha, \beta, \alpha +
  \beta \in [0, 1]$:
  \begin{enumerateroman}
    \item $I^{\alpha}_{0 +} \circ I^{\beta}_{0 +} = I^{\alpha + \beta}_{0 +}$,
    $I^0_{0 +} = \tmop{Id}$, similarly for $D^{\alpha}_{0 +}$;
    
    \item $I^{\alpha}_{0 +} \circ D^{\alpha}_{0 +} = D^{\alpha}_{0 +} \circ
    I^{\alpha}_{0 +} = \tmop{Id}$, $D^1_{0 +} = D$.
  \end{enumerateroman}
  Let $f$ be a smooth function, then $\Xi [f] = I^{\delta}_{0 +} D f = D^{1 -
  \delta}_{0 +} f$; moreover for any $\gamma < \alpha$, we can write $f$ as $f
  = I^{\gamma}_{0 +} \tilde{f}$ with $\| \tilde{f} \|_{\infty} \lesssim \| f
  \|_{\alpha}$; choosing $\gamma > 1 - \delta$, we obtain $\Xi [f] = I^{\gamma
  + \delta - 1}_{0 +} \tilde{f}$ and so overall $\Xi [f] \in I^{\gamma +
  \delta - 1}_{0 +} (L^{\infty}_t) \hookrightarrow C^{\gamma + \delta - 1 -
  \varepsilon}_t$ with
  \[ \| \Xi [f] \|_{\gamma + \delta - 1 - \varepsilon} \lesssim \| I^{\gamma +
     \delta - 1}_{0 +} \tilde{f} \|_{I^{\gamma + \delta - 1}_{0 +}
     (L^{\infty}_t)} \lesssim \| \tilde{f} \|_{\infty} \lesssim \| f
     \|_{\alpha} . \]
  The conclusion for general $f$ follows from an approximation procedure.
  Indeed, since all inequalities are strict, we can replace $\alpha$ with
  $\alpha - \varepsilon$ and use the fact that functions in $C^{\alpha}_t$ can
  be approximated by smooth functions in the $C^{\alpha -
  \varepsilon}_t$-norm.
  
  The fact that in~{\eqref{sec3.5 continuity fractional}} only the seminorm
  $\llbracket f \rrbracket$ appears is a consequence of the fact that by
  definition $\Xi [1] = 0$ and so we can always shift $f$ in such a way that
  $f_0 = 0$. 
\end{proof}

\begin{remark}
  \label{sec3.5 remark}Let us point out two properties of the operator $\Xi$.
  The first one is that, if $f \equiv g$ on $[0, \tau]$ with $\tau \leqslant
  T$, the same holds for $\Xi [f] \equiv \Xi [g]$; in particular, since we can
  always extend $f \in C^{\alpha} ([0, \tau] ; \mathbb{R}^d)$ to $C^{\alpha}
  ([0, T] ; \mathbb{R}^d)$ by setting $f_t = f_{\tau}$ for all $t \geqslant
  \tau$, we can consider $\Xi$ as an operator from $C^{\alpha} ([0, \tau] ;
  \mathbb{R}^d)$ to $C^{\alpha + \delta - 1 - \varepsilon} ([0, \tau] ;
  \mathbb{R}^d)$. As long as $\tau \leqslant T$, the operator norm of this
  restricted functional is still controlled by $\| \Xi \|$.
  
  The second one is that if $h \equiv 0$ on $[0, \tau]$, then $\Xi [h]_{\cdot
  + \tau} = \Xi [h_{\cdot + \tau}]$. Indeed for $h$ smooth it holds
  
  \begin{align*}
    \Xi [h]_{t + \tau} & = \int_0^{t + \tau} (t + \tau - s)^{\delta - 1} 
    \dot{h}_s \mathd s = \int_{\tau}^{t + \tau} (t + \tau - s)^{\delta - 1}
    \dot{h}_s \mathd s\\
    & = \int_0^t (t - s)^{\delta - 1} \dot{h}_{s + \tau} \mathd s = \Xi
    [h_{\cdot + \tau}]_t .
  \end{align*}
  
  The general case follows from an approximation procedure.
\end{remark}

Thanks to Lemma~\ref{sec3.5 lemma fractional} we can give a proper meaning to
the fractional YDE.

\begin{definition}
  \label{sec3.5 defn solution}We say that $x$ is a solution to~{\eqref{sec3.5
  prototype YDE}} if $\int_0^{\cdot} A (\mathd s, x_s)$ is well defined as a
  nonlinear Young integral in $C^{\alpha}_t$ for some $\alpha > 1 - \delta$
  and $x$ satisfies the identity
  \[ x_{\cdot} = x_0 + \Xi \left[ \int_0^{\cdot} A (\mathd s, x_s) \right] .
  \]
\end{definition}

\begin{proposition}
  \label{sec3.5 prop existence}Let $A \in C^{\alpha}_t C^{\beta}_x$ with
  $\alpha, \beta \in (0, 1)$ satisfying
  \begin{equation}
    \alpha + \delta - 1 > \frac{1 - \alpha}{\beta} . \label{sec3.5 condition
    coefficients}
  \end{equation}
  Then for any $x_0 \in \mathbb{R}^d$ and any $\gamma < \alpha + \delta - 1$
  there exists a solution $x \in C^{\gamma}_t$ to {\eqref{sec3.5 prototype
  YDE}}, in the sense of Definition~\ref{sec3.5 defn solution}.
\end{proposition}

\begin{proof}
  Due to condition~{\eqref{sec3.5 condition coefficients}}, we can find
  $\gamma \in (0, 1)$, $\varepsilon > 0$ sufficiently small satisfying
  \[ \alpha + \delta - 1 > \gamma > \gamma - \varepsilon > \frac{1 -
     \alpha}{\beta} . \]
  The existence of a solution is then equivalent to the existence of a fixed
  point in $C^{\gamma}_t$ for the map
  \[ I (x) \assign x_0 + \Xi \left[ \int_0^{\cdot} A (\mathd s, x_s) \right] .
  \]
  The above conditions imply $\alpha + \beta (\gamma - \varepsilon) > 1$, so
  by Theorem~\ref{sec2 thm definition young integral} the map $x \mapsto A
  (\mathd s, x_s)$, from $C^{\gamma - \varepsilon}_t$ to $C^{\alpha}_t$ is
  continuous and satisfies
  \[ \left\llbracket \int_0^{\cdot} A (\mathd s, x_s)
     \right\rrbracket_{\alpha} \lesssim \| A \|_{\alpha \comma \beta} (1 +
     \llbracket x \rrbracket_{\gamma - \varepsilon}^{\beta}) \]
  which together with estimate~{\eqref{sec3.5 continuity fractional}} implies
  that $I$ is continuous from $C^{\gamma - \varepsilon}_t$ to $C^{\gamma}_t$
  with
  \[ \| I (x) \|_{\gamma} \leqslant \| x_0 \| + \kappa_1 \| \Xi \| \| A
     \|_{\alpha \comma \beta} (1 + \llbracket x \rrbracket_{\gamma -
     \varepsilon}^{\beta}) \]
  for suitable $\kappa_1 = \kappa_1 (T, \alpha + \beta (\gamma -
  \varepsilon))$. It follows by Ascoli-Arzel{\`a} that $I$ is compact from
  $C^{\gamma - \varepsilon}_t$ to itself; for any $\lambda \in (0, 1)$, if $x$
  solves $x = \lambda I (x)$, then
  \[ \| x \|_{\gamma - \varepsilon} \leqslant \| x \|_{\gamma} = \lambda \| T
     (x) \|_{\gamma} \leqslant \| x_0 \| + \kappa_1 \| \Xi \| \| A \|_{\alpha,
     \beta} (1 + \| x \|_{\gamma - \varepsilon}^{\beta}) . \]
  Since $\beta < 1$, any such solution $x$ must satisfy (for instance)
  \[ \| x \|_{\gamma - \varepsilon} \leqslant \max \left\{ 2 (\| x_0 \| +
     \kappa_1 \| \Xi \| \| A \|_{\alpha, \beta}), \, (2 \kappa_1 \| \Xi \| \|
     A \|_{\alpha, \beta})^{\frac{1}{1 - \beta}} \right\} \]
  where the estimate is uniform in $\lambda \in [0, 1]$. We can thus apply
  Schaefer's theorem to deduce the existence of a fixed point for $I$ in
  $C^{\gamma - \varepsilon}_t$, which also belongs to $C^{\gamma}_t$ since $I
  (x)$ does so.
\end{proof}

\begin{theorem}
  \label{sec3.5 main thm 2}Let $A \in C^{\alpha}_t C^{1 + \beta}_x$ with
  $\alpha, \beta, \delta$ satisfying~{\eqref{sec3.5 condition coefficients}}.
  Then for any $x_0 \in \mathbb{R}^d$ there exists a unique solution $x \in
  C^{\gamma}_t$ to~{\eqref{sec3.5 prototype YDE}}, for any $\gamma$ satisfying
  \[ \alpha + \delta - 1 > \gamma > \frac{1 - \alpha}{\beta} . \]
\end{theorem}

\begin{proof}
  Existence is granted by Proposition~\ref{sec3.5 prop existence}, so we only
  need to check uniqueness. Let $x$ and $y$ be two solutions, say with $\| x
  \|_{\alpha}, \| y \|_{\alpha} \leqslant M$ for suitable $M > 0$; we are
  first going to show that they must coincide on an interval $[0, \tau]$ with
  $\tau$ sufficiently small. It holds
  
  \begin{align*}
    \llbracket x - y \rrbracket_{\gamma ; 0, \tau} & = \left\llbracket \Xi
    \left[ \int_0^{\cdot} A (\mathd s, x_s) - \int_0^{\cdot} A (\mathd s, y_s)
    \right] \right\rrbracket_{\gamma ; 0, \tau}\\
    & \leqslant \| \Xi \|  \left\llbracket \int_0^{\cdot} A (\mathd s, x_s) -
    \int_0^{\cdot} A (\mathd s, y_s) \right\rrbracket_{\alpha ; 0, \tau}\\
    & = \| \Xi \|  \left\llbracket \int_0^{\cdot} v_{\mathd s} (x_s - y_s)
    \right\rrbracket_{\alpha ; 0, \tau}
  \end{align*}
  
  where $v$ is given by
  \[ v_t = \int_0^1 \int_0^t \nabla A (\mathd s, y_s + \lambda (x_s - y_s))
     \mathd \lambda \]
  and satisfies $\| v \|_{\alpha ; 0, T} \leqslant \kappa_1  \| A \|_{\alpha,
  1 + \beta} (1 + M)$. Since $x_0 = y_0$, for any $[s, t] \subset [0, \tau]$
  it holds
  
  \begin{align*}
    \left\| \int_s^t v_{\mathd r} (x_r - y_r) \right\| & \leqslant \| v_{s, t}
    (x_s - y_s) \| + \kappa_2 | t - s |^{\alpha + \gamma} \| v \|_{\alpha}
    \llbracket x - y \rrbracket_{\gamma ; 0, \tau}\\
    & \leqslant | t - s |^{\alpha} \tau^{\gamma}  (1 + \kappa_2)  \| v
    \|_{\alpha}  \llbracket x - y \rrbracket_{\gamma ; 0, \tau} ;
  \end{align*}
  
  combined with the previous estimates we obtain
  
  \begin{align*}
    \llbracket x - y \rrbracket_{\gamma ; 0, \tau} & \leqslant \| \Xi \|
    \tau^{\gamma}  (1 + \kappa_2)  \| v \|_{\alpha}  \llbracket x - y
    \rrbracket_{\gamma ; 0, \tau}\\
    & \leqslant \kappa_3  \| \Xi \| \| A \|_{\alpha, 1 + \beta} (1 + M)
    \tau^{\gamma} \llbracket x - y \rrbracket_{\gamma ; 0, \tau} .
  \end{align*}
  
  Choosing $\tau$ small enough such that $\kappa_3  \| \Xi \| \| A \|_{\alpha,
  1 + \beta} (1 + M) \tau^{\gamma} < 1$, we conclude that $x \equiv y$ on $[0,
  \tau]$.
  
  As a consequence, $\int_0^{\cdot} A (\mathd s, x_s) = \int_0^{\cdot} A
  (\mathd s, y_s)$ on $[0, \tau]$ as well; define $v_t = x_{t + \tau} - y_{t +
  \tau}$, then applying Remark~\ref{sec3.5 remark} to $v$ we obtain
  
  \begin{align*}
    v_t & = \, \Xi \left[ \int_0^{\cdot} A (\mathd s, x_s) - A (\mathd s, y_s)
    \right]_{t + \tau}\\
    & = \, \Xi \left[ \int_{\tau}^{\cdot + \tau} A (\mathd s, x_s) -
    \int_{\tau}^{\cdot + \tau} A (\mathd s, y_s) \right]_t\\
    & = \, \Xi \left[ \int_0^{\cdot} \tilde{A} (\mathd s, x_{s + \tau}) -
    \int_0^{\cdot} \tilde{A} (\mathd s, y_{s + \tau}) \right]_t
  \end{align*}
  
  where $\tilde{A} (t, x) = A (t + \tau, x)$ has the same regularity
  properties of $A$. We can therefore iterate the previous argument, applied
  this time to $\tilde{A}$, $x_{\cdot + \tau}$ and $y_{\cdot + \tau}$, to
  deduce that $x$ and $y$ also coincide on $[\tau, 2 \tau]$; repeating this
  procedure we can cover the whole interval $[0, T]$.
\end{proof}

\section{Flow}\label{sec4}

Having established sufficient conditions for the existence and uniqueness of
solutions to the YDE associated to $(x_0, A)$, it is natural to study their
dependence on the data of the problem. This section is devoted to the study of
the flow, seen as the ensemble of all possible solutions, and its Frech{\'e}t
differentiability w.r.t. both $(x_0, A)$.

In order to avoid technicalities we will only consider the case of $A \in
C^{\alpha}_t C^{1 + \beta}_V$ with global bounds, but everything extends
easily by localisation arguments to $A \in C^{\alpha}_t C^{\beta, \lambda}_V
\cap C^{\alpha}_t C^{1 + \beta}_{V, \tmop{loc}}$; similar results can also be
established for the type of equations considered respectively in
Sections~\ref{sec3.4} and~\ref{sec3.5}.

\subsection{Flow of diffeomorphisms}

We start by giving a proper definition of a flow for the YDE associated to
$A$; recall here that $\Delta_n$ denotes the $n$-simplex on $[0, T]$.

\begin{definition}
  \label{sec4 defn flow}Given $A \in C^{\alpha}_t C^{\beta}_V$ with $\alpha (1
  + \beta) > 1$, we say that $\Phi : \Delta_2 \times V \rightarrow V$ is a
  flow of homeomorphisms for the YDE associated to $A$ if the following hold:
  \begin{enumerateroman}
    \item $\Phi (t, t, x) = x$ for all $t \in [0, T]$ and $x \in V$;
    
    \item $\Phi (s, \cdot, x) \in C^{\alpha} ([s, T] ; V)$ for all $s \in [0,
    T]$ and $x \in V$;
    
    \item for all $(s, t, x) \in \Delta_2 \times \mathbb{R}^d$ it holds
    \[ \Phi (s, t, x) = x + \int_s^t A (\mathd r, \Phi (s, r, x)) ; \]
    \item $\Phi$ satisfies the group property, namely
    \[ \Phi (u, t, \Phi (s, u, x)) = \Phi (s, t, x) \quad \text{for all } (s,
       u, t) \in \Delta_3 \text{ and } x \in V ; \]
    \item for any $(s, t) \in \Delta_2$, the map $\Phi (s, t, \cdot)$ is an
    homeomorphism of $V$, i.e. it is continuous with continuous inverse.
  \end{enumerateroman}
\end{definition}

From now on, whenever talking about a flow $\Phi$, we will use the notation
$\Phi_{s \rightarrow t} (x) = \Phi (s, t, x)$; we will denote by $\Phi_{s
\leftarrow t} (\cdot)$ the inverse of $\Phi_{s \rightarrow t} (\cdot)$ as a
map from $V$ to itself.

\begin{definition}
  \label{sec4 defn regularity flow}Given $A$ as above, $\gamma \in (0, 1)$, we
  say that it admits a locally $\gamma$\mbox{-}H\"{o}lder continuous flow
  $\Phi$, $\Phi$ is $C^{\gamma}_{\tmop{loc}}$ for short, if for any $(s, t)
  \in \Delta_2$ it holds $\Phi_{s \rightarrow t}, \Phi_{s \leftarrow t} \in
  C^{\gamma}_{\tmop{loc}} (V ; V)$; we say that $\Phi$ is a flow of
  diffeomorphisms if $\Phi_{s \rightarrow t}, \Phi_{s \leftarrow t} \in
  C^1_{\tmop{loc}} (V ; V)$ for any $(s, t) \in \Delta_2$. Similar definitions
  hold for a locally Lipschitz flow, or a $C^{n +
  \gamma}_{\tmop{loc}}$\mbox{-}flow with $\gamma \in [0, 1)$ and $n \in
  \mathbb{N}$.
  
  If $V =\mathbb{R}^d$, we say that $\Phi$ is a Lagrangian flow if there
  exists a constant $C$ such that
  \[ C^{- 1} \lambda_d (E) \leqslant \lambda_d (\Phi_{s \leftarrow t} (E))
     \leqslant C \lambda_d (E) \qquad \forall \, E \in \mathcal{B}
     (\mathbb{R}^d), \, \forall \, (s, t) \in \Delta_2, \]
  where $\lambda_d$ denotes the Lebesgue measure on $\mathbb{R}^d$ and
  $\mathcal{B} (\mathbb{R}^d)$ the collection of Borel sets.
\end{definition}

It follows from Remark~\ref{sec3.3 remark lipschitz solution map} that, if $A
\in C^{\alpha}_t C^{1 + \beta}_V$ with $\alpha (1 + \beta) > 1$, then the
solution map $(x_0, t) \mapsto x_t$ is Lipschitz in space, uniformly in time.
However we cannot yet talk about a flow, as we haven't shown the invertibility
of the solution map, nor the flow property; this is accomplished by the
following lemma.

\begin{lemma}
  \label{sec4.1 invariance shift inversion}Let $A \in C^{\alpha}_t
  C^{\beta}_V$ and $x \in C^{\alpha}_t V$ such that $\alpha (1 + \beta) > 1$,
  $x$ be a solution of the YDE associated to $(x_0, A)$. Then setting
  $\tilde{A} (t, z) : = A (T - t, z)$ and $\tilde{x}_t \assign x_{T - t}$,
  $\tilde{x}$ is a solution to the time-reversed YDE
  \[ \tilde{x}_t = \tilde{x}_0 + \int_0^t \tilde{A} (\mathd s, \tilde{x}_s) .
  \]
  Similarly, setting $\tilde{x}_t = x_{t - s}$, $\tilde{A} (t, x) = A (t - s,
  x)$ for $t \in [s, T]$, then $\tilde{x}$ is a solution to the time-shifted
  YDE
  \[ \tilde{x}_t = \tilde{x}_0 + \int_0^t \tilde{A} (\mathd r, \tilde{x}_r)
     \quad \forall \, t \in [s, T] . \]
\end{lemma}

The proof is elementary but a bit tedious, so we omit it; we refer the
interested reader to Lemma~2, Section~6.1 from~{\cite{lejay}} or Lemmas~11
and~12, Section~4.3.1 from~{\cite{galeatigubinelli1}}.

As a consequence, we immediately deduce conditions for the existence of a
Lipschitz flow.

\begin{corollary}
  \label{sec4.1 corollary lipschitz flow}Let $A \in C^{\alpha}_t C^{1 +
  \beta}_V$ with $\alpha (1 + \beta) > 1$, then the associated YDE admits a
  locally Lipschitz flow $\Phi$. Moreover there exists $C = C (\alpha, \beta,
  T, \| A \|_{\alpha, 1 + \beta})$ such that
  \begin{equation}
    \| \Phi_{s \rightarrow \cdot} (x) - \Phi_{s \rightarrow \cdot} (y)
    \|_{\alpha ; s, T} \leqslant C \| x - y \|_V, \quad \llbracket \Phi_{s
    \rightarrow \cdot} (x) \rrbracket_{\alpha ; s, T} \leqslant C \quad
    \forall \, s \in [0, T], \, x, y \in V \label{sec4.1 lipschitz flow}
  \end{equation}
  together with a similar estimate for $\Phi_{\cdot \leftarrow t} (\cdot)$.
\end{corollary}

\begin{proof}
  The proof is a straightforward application of Remark~\ref{sec3.3 remark
  lipschitz solution map} and Lemma~\ref{sec4.1 invariance shift inversion}.
  In both cases of time reversal and translation we have $\| \tilde{A}
  \|_{\alpha, 1 + \beta} \leqslant \| A \|_{\alpha, 1 + \beta}$ so that
  uniqueness holds also for the reversed/translated YDE, with the same
  continuity estimates; this provides respectively invertibility of the
  solution map and flow property.
\end{proof}

Actually, under the same hypothesis it is possible to prove that the YDE
admits a flow of diffeomorphisms, which satisfies a variational equation.

\begin{theorem}
  \label{sec4.1 main thm infinite dim}Let $A \in C^{\alpha}_t C^{1 + \beta}_V$
  with $\alpha (1 + \beta) > 1$, then the YDE associated to $A$ admits a flow
  of diffeomorphisms. For any $x \in V$, $D_x \Phi_{s \rightarrow t} (x) =
  J^x_{s \rightarrow t}$, where $J^x_{s \rightarrow \cdot} \in C^{\alpha}_t
  \mathcal{L} (V ; V)$ is the unique solution to the variational equation
  \begin{equation}
    J^x_{s \rightarrow t} = I + \int_s^t D A (\mathd r, \Phi_{s \rightarrow r}
    (x)) \circ J^x_{s \rightarrow r} \qquad \forall \, t \in [s, T]
    \label{sec4.1 variational eq infinite dim}
  \end{equation}
  where $\circ$ denotes the composition of linear operators.
\end{theorem}

We postpone the proof of this result to Section~\ref{sec4.2}, as the variation
equation will follow from a more general result on the differentiability of
the It\^{o} map. Following~{\cite{hu}}, we give an alternative proof in the
case of finite dimensional $V$, in which more precise information on $\Phi$ is
known.

\begin{theorem}
  \label{sec4.1 main thm finite dim}Let $A$ satisfy the hypothesis of
  Theorem~\ref{sec4.1 main thm infinite dim}, $V =\mathbb{R}^d$ for some $d
  \in \mathbb{N}$; then the associated YDE admits a flow of diffeomorphisms
  and the following hold:
  \begin{enumerateroman}
    \item For any $x \in \mathbb{R}^d$ and $s \in [0, T]$, $D_x \Phi_{s
    \rightarrow \cdot} (x)$ corresponds to $J^x_{s \rightarrow \cdot} \in
    C^{\alpha} ([s, T] ; \mathbb{R}^{d \times d})$ satisfying
    \begin{equation}
      J^x_{s \rightarrow t} = I + \int_s^t D A (\mathd r, \Phi_{s \rightarrow
      r} (x)) J^x_{s \rightarrow r} . \label{sec4.1 finite dim variational
      equation}
    \end{equation}
    \item The Jacobian $\jmath_{s \rightarrow t} (x) : = \det (D_x \Phi_{s
    \rightarrow t} (x))$ satisfies the identity
    \begin{equation}
      \jmath_{s \rightarrow t} (x) = \exp \left( \int_s^t \tmop{div} A (\mathd
      r, \Phi_{s \rightarrow r} (x)) \right) \label{sec4.1 formula jacobian}
    \end{equation}
    and there exists a constant $C = C (\alpha, \beta, T, \| A \|_{\alpha, 1 +
    \beta}) > 0$ such that
    \[ C^{- 1} \leqslant \jmath_{s \rightarrow t} (x) \leqslant C \quad
       \forall \, (s, t, x) \in \Delta_2 \times \mathbb{R}^d . \]
    In particular, $\Phi$ is a Lagrangian flow of diffeomorphisms.
  \end{enumerateroman}
\end{theorem}

\begin{proof}
  For simplicity we will prove all the statements for $s = 0$, the general
  case being similar. By Corollary~\ref{sec4.1 corollary lipschitz flow}, the
  existence of a locally Lipschitz flow $\Phi$ is known; to show
  differentiability, it is enough to establish existence and continuity of the
  Gateaux derivatives.
  
  Fix $x, v \in \mathbb{R}^d$ and consider for any $\varepsilon > 0$ the map
  $\eta^{\varepsilon}_t \assign \varepsilon^{- 1} (\Phi_{0 \rightarrow \cdot}
  (x + \varepsilon_n v) - \Phi_{0 \rightarrow \cdot} (x))$; by
  estimate~{\eqref{sec4.1 lipschitz flow}}, the family $\{ \eta^{\varepsilon}
  \}_{\varepsilon > 0}$ is bounded in $C^{\alpha}_t \mathbb{R}^d$. Thus by
  Ascoli-Arzel{\`a} we can extract a subsequence $\varepsilon_n \rightarrow 0$
  such that $\eta^{\varepsilon} \rightarrow \eta$ in $C^{\alpha - \delta}_t$
  for some $\eta \in C^{\alpha}_t$ and any $\delta > 0$. Choose $\delta > 0$
  small enough such that $(\alpha - \delta) (1 + \beta) > 1$; using the fact
  that the map $F (y) = \int_0^{\cdot} A (\mathd s, y_s)$ is differentiable
  from $C^{\alpha - \delta}_t$ to itself by Proposition~\ref{sec2 prop frechet
  differentiability}, with $D F$ given by~{\eqref{sec2 frechet derivative
  young}}, by chain rule we deduce that
  
  \begin{align*}
    \eta_{\cdot} & = \lim_{\varepsilon_n \rightarrow 0} \frac{\Phi_{0
    \rightarrow \cdot} (x + \varepsilon_n v) - \Phi_{0 \rightarrow \cdot}
    (x)}{\varepsilon_n}\\
    & = \, v + \lim_{\varepsilon_n \rightarrow 0} \frac{F (\Phi_{0
    \rightarrow \cdot} (x + \varepsilon_n v)) - F (\Phi_{0 \rightarrow \cdot}
    (x))}{\varepsilon_n}\\
    & = v + D F (\Phi_{0 \rightarrow \cdot} (x)) (\eta_{\cdot}) ;
  \end{align*}
  
  namely, $\eta$ satisfies the YDE
  \begin{equation}
    \eta_t = v + \int_0^t D_x A (\mathd r, \Phi_{0 \rightarrow r} (x)) \eta_r
    \label{sec4.1 proof flow eq1}
  \end{equation}
  whose meaning was defined in Remark~\ref{sec2 remark other integrals}.
  Equation~{\eqref{sec4.1 proof flow eq1}} is an affine YDE, which admits a
  unique solution by Corollary~\ref{sec3.2 cor local uniqueness}; moreover
  it's easy to check that the unique solution must have the form $\eta_t =
  J_{0 \rightarrow t}^x v$, where $J_{0 \rightarrow \cdot}^x \in C^{\alpha}_t
  \mathbb{R}^{d \times d}$ is the unique solution to the affine $\mathbb{R}^{d
  \times d}$-valued YDE
  \[ J^x_{0 \rightarrow t} = I + \int_0^t D_x A (\mathd r, \Phi_{0 \rightarrow
     r} (x)) J^x_{0 \rightarrow r}, \]
  whose global existence and uniqueness follows from Corollary~\ref{sec3.2 cor
  local uniqueness} and Theorem~\ref{sec3.3 thm gronwall estimate YDE}. As the
  reasoning holds for any subsequence $\varepsilon_n$ we can extract and any
  $v \in \mathbb{R}^d$, we conclude that $\Phi_{0 \rightarrow t} (\cdot)$ is
  Gateaux differentiable with $D \Phi_{0 \rightarrow t} (x) = J^x_{0
  \rightarrow t}$ which satisfies~{\eqref{sec4.1 finite dim variational
  equation}}. A similar argument shows that $J_{0 \rightarrow t}^x$ depends
  continuously on $x$, from which Frech{\'e}t differentiability follows.
  
  Part~\tmtextit{ii.} can be established for instance by means of an
  approximation procedure; indeed by Lemma~\ref{sec2 remark approximation A},
  given $A \in C^{\alpha}_t C^{1 + \beta}_x$, we can find $A^n \in C^1_t C^{1
  + \beta}_x$ such that $A^n \rightarrow A$ in $C^{\alpha -}_t C^{1 + \beta
  -}_x$ and by Theorem~\ref{sec3.3 comparison principle}, the solutions
  $y^n_{\cdot} = \Phi_{0 \rightarrow \cdot}^n (x)$ associated to $(x, A^n)$
  converge to $\Phi_{0 \rightarrow \cdot} (x)$ associated to $(x, A)$.
  Moreover for $A^n$ the YDE is meaningful as the more classical ODE
  associated to $\partial_t A^n$, so we can apply to it all the classical
  results from ODE theory; the Jacobian associated to $A^n$ is given by
  \[ \det (D_x \Phi_{0 \rightarrow t}^n (x)) = \exp \left( \int_0^t \tmop{div}
     \partial_t A^n (r, \Phi_{0 \rightarrow r}^n (x)) \mathd r \right) = \exp
     \left( \int_0^t \tmop{div} A^n (\mathd r, \Phi_{0 \rightarrow r}^n (x))
     \right) . \]
  Passing to the limit as $n \rightarrow \infty$, by the continuity of
  nonlinear Young integrals, we obtain~{\eqref{sec4.1 formula jacobian}}.
  Moreover by equation~{\eqref{sec4.1 lipschitz flow}} we have the estimate
  \[ \sup_{t \in [0, T]} \left| \int_0^t \tmop{div} A (\mathd r, \Phi_{0
     \rightarrow r} (x)) \right| \lesssim \| \tmop{div} A \|_{\alpha, \beta}
     (1 + \llbracket \Phi_{0 \rightarrow \cdot} (x) \rrbracket_{\alpha})
     \lesssim \| A \|_{\alpha, 1 + \beta}, \]
  which gives Lagrangianity.
\end{proof}

It's possible to show that the flow inherits regularity from the drift, namely
that to a spatially more regular $A$ corresponds a more regular $\Phi$.

\begin{theorem}
  \label{sec4.1 thm higher regularity}Let $n \in \mathbb{N}$, $\alpha, \beta
  \in (0, 1)$ be such that $\alpha (1 + \beta) > 1$ and assume $A \in
  C^{\alpha}_t C^{n + \beta}_V$. Then the flow $\Phi$ associated to $A$ is
  locally $C^n$-regular.
\end{theorem}

We omit the proof, which follows similar lines to those of
Theorems~\ref{sec4.1 main thm infinite dim} and~\ref{sec4.1 main thm finite
dim} and is mostly technical; we refer the interested reader to
{\cite{galeatigubinelli1}},~{\cite{perkowski}} and the discussion at the end
of Section~3 from~{\cite{lejay}}.

\begin{remark}
  In line with Section~\ref{sec3.4}, one can obtain sufficient conditions for
  the existence of a regular flow under the additional assumption $\partial_t
  A \in C ([0, T] \times V ; V)$; in this case if $A \in C^{\alpha}_t C^{n +
  \beta}_V$, then it has a locally $C^n$-regular flow, see the discussion in
  Section~4.3 from~{\cite{galeatigubinelli1}}. Similar reasonings allow to
  establish existence of a flow also for the equations treated in
  Section~\ref{sec3.5}.
\end{remark}

\subsection{Differentiability of the It\^{o} map}\label{sec4.2}

Denote by $\Phi^A_{s \rightarrow \cdot} (x)$ the solution to the YDE
associated to $(x, A)$; the aim of this section is to study the dependence of
the flow $\Phi^A$ as a function of $A \in C^{\alpha}_t C^{1 + \beta}_V$,
namely to identify $D_A \Phi^A_{s \rightarrow \cdot} (x)$.

For simplicity we will restrict to the case $s = 0$; we will actually fix $A
\in C^{\alpha}_t C^{1 + \beta}_V$, consider $\Phi^{A + \varepsilon B}$ with
$B$ varying and set $X^x_t \assign \Phi^A_{0 \rightarrow t} (x)$.

\begin{theorem}
  \label{sec4 differentiability ito map}Let $\alpha (1 + \beta) > 1$, $x_0 \in
  V$ and consider the It\^{o} map $\Phi^{\cdot}_{0 \rightarrow \cdot} (x) :
  C^{\alpha}_t C^{1 + \beta}_V \rightarrow C^{\alpha}_t V$, $A \mapsto
  \Phi^A_{0 \rightarrow \cdot} (x)$. Then $\Phi^{\cdot}_{0 \rightarrow \cdot}
  (x)$ is Frech{\'e}t differentiable and for any $B \in C^{\alpha}_t C^{1 +
  \beta}_V$ the Gateaux derivative
  \[ D_A \Phi^A_{0 \rightarrow \cdot} (x) (B) = \lim_{\varepsilon \rightarrow
     0} \frac{1}{\varepsilon} (\Phi^{A + \varepsilon B}_{0 \rightarrow \cdot}
     (x) - \Phi_{0 \rightarrow \cdot}^A (x)) \in C^{\alpha}_t V \]
  satisfies the affine YDE
  \begin{equation}
    Y^x_t = \int_0^t D A (\mathd s, X^x_s) (Y^x_s) + \int_0^t B (\mathd s,
    X^x_s) \quad \forall \, t \in [0, T]  \label{sec4 directional derivatives
    ito}
  \end{equation}
  and is given explicitly by
  \begin{equation}
    D_A \Phi^A_{0 \rightarrow t} (x) (B) = J^x_{0 \rightarrow t} \int_0^t
    (J^x_{0 \rightarrow s})^{- 1} B (\mathd s, X^x_s) \quad \forall \, t \in
    [0, T] \label{sec4 directional derivatives2}
  \end{equation}
  where $J^x_{0 \rightarrow \cdot}$ is the unique solution to~{\eqref{sec4.1
  variational eq infinite dim}} and $(J^x_{0 \rightarrow s})^{- 1}$ denotes
  its inverse as an element of $L (V)$.
\end{theorem}

The proof requires the following preliminary lemma.

\begin{lemma}
  \label{sec4.2 techlem}For any $L \in C^{\alpha}_t L (V)$, there exists a
  unique solution $M \in C^{\alpha}_t L (V)$ to the YDE
  \begin{equation}
    M_t = \tmop{Id}_V + \int_0^t L_{\mathd s} \circ M_s \qquad \forall \, t
    \in [0, T] ; \label{sec4.2 lemma eq1}
  \end{equation}
  moreover $M_t$ is invertible for any $t \in [0, T]$ and $N_{\cdot} \assign
  (M_{\cdot})^{- 1} \in C^{\alpha}_t L (V)$ is the unique solution to
  \begin{equation}
    N_t = \tmop{Id}_V - \int_0^t N_s \circ L_{\mathd s} \qquad \forall \, t
    \in [0, T] . \label{sec4.2 lemma eq2}
  \end{equation}
  Finally, for any $y_0 \in V$ and any $\psi \in C^{\alpha}_t V$, the unique
  solution to the affine YDE
  \begin{equation}
    y_t = y_0 + \int_0^t L_{\mathd s} y_s + \psi_t \label{sec4.2 lemma eq3}
  \end{equation}
  is given by
  \begin{equation}
    y_t = M_t y_0 + M_t \int_0^t N_s \mathd \psi_s . \label{sec4.2 lemma eq4}
  \end{equation}
\end{lemma}

\begin{proof}
  Setting $A (t, M) \assign L_t \circ M$, $A \in C^{\alpha}_t C^2_{L (V),
  \tmop{loc}}$ and so existence and uniqueness of a global solution
  to~{\eqref{sec4.2 lemma eq1}} follows from Corollary~\ref{sec3.2 cor local
  uniqueness} and Theorem~\ref{sec3.3 thm gronwall estimate YDE}; similarly
  for~{\eqref{sec4.2 lemma eq2}} with $\tilde{A} (t, N) = N \circ L_t$. Let
  $M_{\cdot}, N_{\cdot} \in C^{\alpha}_t L (V)$ be solution respectively
  to~{\eqref{sec4.2 lemma eq1}},~{\eqref{sec4.2 lemma eq2}}, we claim that
  they are inverse of each other. Indeed by the product rule for Young
  integrals it holds
  \[ \mathd (N_t \circ M_t) = (\mathd N_t) \circ M_t + N_t \circ (\mathd M_t)
     = - N_t \circ L_{\mathd t} \circ M_t + N_t \circ L_{\mathd t} \circ M_t =
     0 \]
  which implies $N_t \circ M_t = N_0 \circ M_0 = \tmop{Id}_V$ and thus $N_t =
  (M_t)^{- 1}$. Let $y_{\cdot} \in C^{\alpha}_t V$ be the unique solution
  to~{\eqref{sec4.2 lemma eq3}}, whose global existence and uniqueness follows
  as above, and set $z_t = N_t y_t$; then again by Young product rule it holds
  $\mathd z_t = N_t \mathd \psi_t$ and thus
  \[ N_t y_t = z_t = z_0 + \int_0^t \mathd z_s = y_0 + \int_0^t N_s \mathd
     \psi_s \]
  which gives~{\eqref{sec4.2 lemma eq4}}.
\end{proof}

\begin{proof}[of Theorem~\ref{sec4 differentiability ito map}]
  Given $A, B \in C^{\alpha}_t C^{1 + \beta}_V$, it is enough to show that
  \[ \lim_{\varepsilon \rightarrow 0} \frac{\Phi^{A + \varepsilon B}_{0
     \rightarrow \cdot} (x) - \Phi_{0 \rightarrow \cdot}^A (x)}{\varepsilon} 
     \text{ exists in } C_t^{\alpha} V \]
  and that it is a solution to~{\eqref{sec4 directional derivatives ito}}.
  Once this is shown, we can apply Lemma~\ref{sec4.2 techlem} for the choice
  $L_t = \int_0^t D_x A (\mathd s, X^x_s)$, $y_0 = 0$ and $\psi_t = \int_0^t B
  (\mathd s, X^x_s)$ to deduce that the limit is given by formula~{\eqref{sec4
  directional derivatives2}}, which is meaningful since $J^x_{0 \rightarrow
  \cdot}$ is defined as the solution to~{\eqref{sec4.2 lemma eq1}} for such
  choice of $L$ and is therefore invertible. The explicit formula~{\eqref{sec4
  directional derivatives2}} for the Gateaux derivatives readily implies
  existence and continuity of the Gateux differential $D_A \Phi^A_{0
  \rightarrow \cdot} (x)$ and thus also Frech{\'e}t differentiability.
  
  In order to prove the claim, let $Y^x \in C^{\alpha}_t V$ be the solution
  to~{\eqref{sec4 directional derivatives ito}}, which exists and is unique by
  Lemma~\ref{sec4.2 techlem}; then we need to show that
  \[ \lim_{\varepsilon \rightarrow 0} \left\| \frac{\Phi^{A + \varepsilon
     B}_{0 \rightarrow \cdot} (x) - X^x_{\cdot}}{\varepsilon} - Y^x_{\cdot}
     \right\|_{\alpha} = 0. \]
  Set $X^{\varepsilon, x}_{\cdot} \assign \Phi^{A + \varepsilon B}_{0
  \rightarrow \cdot} (x)$; recall that by the Comparison Principle
  (Theorem~\ref{sec3.3 comparison principle}), we have
  \begin{equation}
    \| X^{\varepsilon, x} - X^x \|_{\alpha} \lesssim \varepsilon \| B
    \|_{\alpha, \beta} . \label{sec4 frechet proof eq1}
  \end{equation}
  Setting $e^{\varepsilon} \assign \varepsilon^{- 1} [X^{\varepsilon, x} -
  X^x] - Y^x$, it holds
  
  \begin{align*}
    e^{\varepsilon}_t & = \, \frac{1}{\varepsilon} \left[ \int_0^t (A +
    \varepsilon B) (\mathd s, X^{\varepsilon, x}_s) - A (\mathd s, X^x_s)
    \right] - \int_0^t D A (\mathd s, X^x_s) (Y^x_s) - \int_0^t B (\mathd s,
    X^x_s)\\
    & = \int_0^t \left[ \frac{A (\mathd s, X^{\varepsilon, x}_s) - A (\mathd
    s, X^x_s)}{\varepsilon} - D A (\mathd s, X^x_s) (Y_s) \right] + \int_0^t
    [B (\mathd s, X^{\varepsilon, x}_s) - B (\mathd s, X^x_s)]\\
    & = \int_0^t D A (\mathd s, X^x_s) (e^{\varepsilon}_s) +
    \psi^{\varepsilon}_t
  \end{align*}
  
  where $\psi^{\varepsilon}$ is given by
  
  \begin{align*}
    \psi^{\varepsilon}_t & = \int_0^t \frac{A (\mathd s, X^{\varepsilon, x}_s)
    - A (\mathd s, X^x_s) - D A (\mathd s, X^x_s) (X^{\varepsilon, x}_s -
    X^x_s)}{\varepsilon} + \int_0^t B (\mathd s, X^{\varepsilon, x}_s) - B
    (\mathd s, X^x_s)\\
    & \backassign \psi^{\varepsilon, 1}_t + \psi^{\varepsilon, 2}_t .
  \end{align*}
  
  In order to conclude, it is enough to show that $\| \psi^{\varepsilon}
  \|_{\alpha} \rightarrow 0$ as $\varepsilon \rightarrow 0$, since then we can
  apply the usual a priori estimates from Theorem~\ref{sec3.3 thm gronwall
  estimate YDE} to $e^{\varepsilon}$, which solves an affine YDE starting at
  $0$. We already know that $X^{\varepsilon, x} \rightarrow X^x$ as
  $\varepsilon \rightarrow 0$, which combined with the continuity of nonlinear
  Young integrals implies that $\psi^{\varepsilon, 2}_t \rightarrow 0$ as
  $\varepsilon \rightarrow 0$. Observe that $\psi^{\varepsilon, 1} =
  \mathcal{J} (\Gamma^{\varepsilon})$ for
  \[ \Gamma^{\varepsilon}_{s, t} = \varepsilon^{- 1} [A_{s, t}
     (X^{\varepsilon, x}_s) - A_{s, t} (X^x_s) - D A_{s, t} (X^x_s)
     (X^{\varepsilon, x}_s - X^x_s)] \]
  which by virtue of~{\eqref{sec4 frechet proof eq1}} satisfies
  \[ \| \Gamma^{\varepsilon}_{s, t} \|_V \lesssim \varepsilon^{- 1} \| A_{s,
     t} \|_{C^{1 + \beta}_V}  \| X^{\varepsilon, x}_s - X^x_s \|_V^{1 + \beta}
     \lesssim \varepsilon^{\beta} | t - s |^{\alpha} \| A \|_{\alpha, 1 +
     \beta} \]
  which implies that $\| \Gamma^{\varepsilon} \|_{\alpha} \rightarrow 0$ as
  $\varepsilon \rightarrow 0$. On the other hand we have
  \begin{eqnarray*}
    \| \delta \Gamma^{\varepsilon}_{s, u, t} \|_V & = & \varepsilon^{- 1}
    \| \int_0^1 [D A_{u, t} (X^x_s + \lambda (X^{\varepsilon, x}_s -
    X^x_s)) - D A_{u, t} (X^x_s)] (X^{\varepsilon, x}_s - X^x_s) \mathd
    \lambda\\
    &  & - \int_0^1 [D A_{u, t} (X^x_u + \lambda (X^{\varepsilon, x}_u -
    X^x_u)) - D A_{u, t} (X^x_u)] (X^{\varepsilon, x}_u - X^x_u) \mathd
    \lambda \|_V\\
    & \leqslant & \varepsilon^{- 1} \left\| \int_0^1 [D A_{u, t} (X^x_s +
    \lambda (X^{\varepsilon, x}_s - X^x_s)) - D A_{u, t} (X^x_s)]
    (X^{\varepsilon, x}_{s, u} - X^x_{s, u}) \mathd \lambda \right\|_V\\
    &  & + \varepsilon^{- 1} \left\| \int_0^1 [D A_{u, t} (X^x_u + \lambda
    (X^{\varepsilon, x}_u - X^x_u)) - D A_{u, t} (X^x_s + \lambda
    (X^{\varepsilon, x}_s - X^{\varepsilon}_s))] (X^{\varepsilon, x}_u -
    X^x_u) \mathd \lambda \right\|_V\\
    &  & + \varepsilon^{- 1} \left\| \int_0^1 [D A_{u, t} (X^x_u) - D A_{u,
    t} (X^x_s)] (X^{\varepsilon, x}_u - X^x_u) \mathd \lambda \right\|_V\\
    & \lesssim & \varepsilon^{- 1} | t - s |^{\alpha (1 + \beta)} \| A
    \|_{\alpha, 1 + \beta}  \llbracket X^{\varepsilon, x} - X^x
    \rrbracket_{\alpha} (1 + \llbracket X^{\varepsilon, x} - X^x
    \rrbracket_{\alpha} + \llbracket X^x \rrbracket_{\alpha})\\
    & \lesssim & | t - s |^{\alpha (1 + \beta)} \| A \|_{\alpha, 1 + \beta} 
    (1 + \llbracket X^x \rrbracket_{\alpha})
  \end{eqnarray*}
  which implies that $\| \delta \Gamma^{\varepsilon} \|_{\alpha (1 + \beta)}$
  are uniformly bounded in $\varepsilon$. We can therefore apply
  Lemma~\ref{appendix lemma interpolation} from the Appendix to conclude.
\end{proof}

\begin{remark}
  Although $A \mapsto \Phi^A$ is defined only on $C^{\alpha}_t C^{1 +
  \beta}_V$, observe that $(A, B) \mapsto D_A \Phi^A_{0 \rightarrow \cdot} (x)
  (B)$ as given by formula~{\eqref{sec4 directional derivatives2}} is well
  defined and continuous for any $(A, B) \in C^{\alpha}_t C^{1 + \beta}_V
  \times C^{\alpha}_t C^{\beta}_V$.
\end{remark}

We can use Theorem~\ref{sec4 differentiability ito map} to complete the proof
of Theorem~\ref{sec4.1 main thm infinite dim}.

\begin{proof}[of Theorem~\ref{sec4.1 main thm infinite dim}]
  The existence of a Lipschitz flow $\Phi$ is granted by Corollary~\ref{sec4.1
  corollary lipschitz flow}, so it suffices to show its differentiability and
  the variational equation; for simplicity we take $s = 0$. Existence of a
  unique solution $J^x_{0 \rightarrow \cdot} \in C^{\alpha}_t L (V)$
  to~{\eqref{sec4.1 variational eq infinite dim}} follows from
  Lemma~\ref{sec4.2 techlem} applied to
  \[ L_t = \int_0^t D A (\mathd r, \Phi_{0 \rightarrow r} (x)) \]
  and by linearity it's easy to check that for any $h \in V$, $Y^h_t \assign
  J^x_{0 \rightarrow t} (h)$ is the unique solution to
  \begin{equation}
    Y^h_t = h + \int_0^t D A (\mathd r, \Phi_{0 \rightarrow r} (x)) (Y^h_r) .
    \label{sec4.2 proof eq2}
  \end{equation}
  Therefore in order to conclude it suffices to show that the directional
  derivatives
  \[ D_x \Phi^A_{0 \rightarrow \cdot} (x) (h) = \lim_{\varepsilon \rightarrow
     0} \frac{\Phi^A_{0 \rightarrow \cdot} (x + \varepsilon h) - \Phi^A_{0
     \rightarrow \cdot} (x)}{\varepsilon} \]
  exist in $C^{\alpha}_t V$ and are solutions to~{\eqref{sec4.2 proof eq2}},
  as this implies that $D_x \Phi^A_{0 \rightarrow \cdot} (x) = J^x_{0
  \rightarrow \cdot}$. Now fix $x, h \in V$ and let $y^{\varepsilon} =
  \Phi^A_{0 \rightarrow \cdot} (x + \varepsilon h)$, then $z^{\varepsilon}
  \assign y^{\varepsilon} - \varepsilon h$ solves
  \[ z^{\varepsilon}_t = x + \int_0^t A^{\varepsilon} (\mathd s,
     z^{\varepsilon}_s) \]
  with $A^{\varepsilon} (t, v) = A (t, v + \varepsilon h)$, i.e.
  $z^{\varepsilon}_{\cdot} = \Phi^{A^{\varepsilon}}_{0 \rightarrow \cdot}
  (x)$. It's easy to see that, if the first limit below exists, then
  \[ \lim_{\varepsilon \rightarrow 0} \frac{z^{\varepsilon} -
     z^0}{\varepsilon} = \lim_{\varepsilon \rightarrow 0}
     \frac{y^{\varepsilon} - y^0}{\varepsilon} - h, \quad \lim_{\varepsilon
     \rightarrow 0} \frac{A^{\varepsilon} - A}{\varepsilon} = B, \quad B (t,
     x) = D A (t, x) (h) . \]
  By the Frech{\'e}t differentiability of $A \mapsto \Phi^A_{0 \rightarrow
  \cdot} (x)$ and the chain rule, it holds
  \[ \lim_{\varepsilon \rightarrow 0} \frac{z^{\varepsilon} -
     z^0}{\varepsilon} = \lim_{\varepsilon \rightarrow 0}
     \frac{\Phi^{A^{\varepsilon}}_{0 \rightarrow \cdot} (x) - \Phi^A_{0
     \rightarrow \cdot} (x)}{\varepsilon} = D_A \Phi^A_{0 \rightarrow \cdot}
     (x) (B) \]
  which is characterized as the unique solution $Z^h$ to
  \[ Z^h_t = \int_0^t D A (\mathd r, \Phi^A_{0 \rightarrow r} (x)) (Z^h_r) +
     \int_0^t D A (\mathd r, \Phi^A_{0 \rightarrow r} (x)) (h) . \]
  This implies by linearity that $Y^h = Z^h_t + h = \lim_{\varepsilon}
  \varepsilon^{- 1} (y^{\varepsilon} - y) = D_x \Phi^A_{0 \rightarrow \cdot}
  (x) (h)$ solves exactly~{\eqref{sec4.2 proof eq2}}. The conclusion follows.
\end{proof}

\begin{example}
  Here are some examples of applications of Theorem~\ref{sec4
  differentiability ito map}.
  \begin{enumerateroman}
    \item Consider the simple case of an additive perturbation, i.e. for fixed
    $(x_0, A)$ we want to understand how the solution $x$ of
    \[ x_t = x_0 + \int_0^t A (\mathd s, x_s) + \psi_t \]
    depends on $\psi$, where $\psi \in C^{\alpha}_t V$ with $\psi_0 = 0$.
    Identifying $\psi$ with $B^{\psi} (t, z) = \psi_t$ for all $z \in V$, it
    holds $x_{\cdot} = \Phi^{A + B^{\psi}}_{0 \rightarrow \cdot} (x_0)
    \backassign F (\psi)$, which implies that $F$ is Frech{\'e}t
    differentiable in $0$ with
    \[ D F (0) (\psi)_{\cdot} = J^x_{0 \rightarrow \cdot} \int_0^{\cdot}
       (J^x_{0 \rightarrow s})^{- 1} \mathd \psi_s . \]

    \item Consider the classical Young case, namely $V =\mathbb{R}^d$, with
    \[ A (t, z) = A [\omega] (t, z) = \sigma (z) \omega_t = \sum_{i = 1}^m
       \sigma_i (z) \omega_t^i, \quad (t, z) \in [0, T] \times \mathbb{R}^d \]
    for regular vector fields $\sigma_i : \mathbb{R}^d \rightarrow
    \mathbb{R}^d$ and $\omega \in C^{\alpha}_t \mathbb{R}^m$, $\alpha > 1 /
    2$; assume $\sigma_i$ are fixed and we are interested in the dependence on
    the drivers $\omega$, namely the map $\Phi^{\omega}_{0 \rightarrow \cdot}
    (x) \assign \Phi^{A [\omega]}_{0 \rightarrow \cdot} (x)$. For fixed
    $\omega \in C_t^{\alpha} \mathbb{R}^m$ and $x \in \mathbb{R}^d$, setting
    $X^x_t : = \Phi^{A [\omega]}_{0 \rightarrow t} (x)$, $J^x_{0 \rightarrow
    t} : = D_x \Phi^{A [\omega]}_{0 \rightarrow t} (x)$, $\Phi^{A [\cdot]}_{0
    \rightarrow \cdot} (x)$ is Frech{\'e}t differentiable at $\omega$ with
    directional derivatives
    \begin{equation}
      D_{\omega} \Phi^{A [\cdot]}_{0 \rightarrow t} (x) (\psi) = J^x_{0
      \rightarrow t} \int_0^t \sum_{i = 1}^m (J^x_{0 \rightarrow r})^{- 1}
      \sigma_i (X^x_r) \mathd \psi^i_r . \label{sec4.2 malliavin eq1}
    \end{equation}
    The above formula uniquely extends by continuity to the case $\psi \in
    W^{1, 1}_t$, in which case we can write it in compact form as
    \begin{equation}
      D_{\omega} \Phi^{A [\cdot]}_{0 \rightarrow t} (x) (\psi) = \int_0^T K
      (t, r) \dot{\psi}_r \mathd r, \quad K (t, r) = 1_{r \leqslant t} J^x_{0
      \rightarrow t} (J^x_{0 \rightarrow r})^{- 1} \sigma (X^x_r) .
      \label{sec4.2 malliavin eq2}
    \end{equation}
    Formulas~{\eqref{sec4.2 malliavin eq1}} and~{\eqref{sec4.2 malliavin eq2}}
    are well known by Malliavin calculus, mostly in the case $\omega$ is
    sampled as an fBm of parameter $H > 1 / 2$, see Section 11.3
    from~{\cite{frizhairer}}; formula~{\eqref{sec4 directional derivatives2}}
    can be regarded as a generalisation of them.
  \end{enumerateroman}
\end{example}

\section{Conditional uniqueness}\label{sec5}

This section provides several criteria for uniqueness of the YDE, under
additional assumptions on the properties of the associated solutions.
Typically such properties can't be established directly, at least not under
mild regularity assumptions on $A$; yet the criteria are rather useful in
application to SDEs, where the analytic theory can be combined with more
probabilistic techniques.

\subsection{A Van Kampen type result for YDEs}\label{sec5.1}

The following result is inspired by the analogue results for ODEs in the style
of van Kampen and Shaposhnikov, see~{\cite{vankampen}},~{\cite{shaposhnikov}}.

\begin{theorem}
  Suppose $A \in C^{\alpha}_t C^{\beta, \lambda}_V$ with $\alpha (1 + \beta) >
  1$, $\beta + \lambda \leqslant 1$ and that the associated YDE admits a
  spatially locally $\gamma$-H\"{o}lder continuous flow. If
  \[ \alpha \gamma (1 + \beta) > 1, \]
  then for any $x_0 \in V$ there exists a unique solution to the YDE in the
  class $x \in C^{\alpha}_t V$.
\end{theorem}

\begin{proof}
  Let $x_0 \in V$ and $x$ be a given solution to the YDE starting at $x_0$. By
  the a priori estimate~{\eqref{sec3.3 a priori estimates growth}}, we can
  always find $R = R (x_0)$ big enough such that
  \[ \sup_{s \in [0, T]} \{ \| x \|_{\alpha} + \| \Phi (s, \cdot, x_s)
     \|_{\alpha ; s, T} \} \leqslant R ; \]
  therefore in the following computations, up to a localisation argument, we
  can assume without loss of generality that $A \in C^{\alpha}_t C^{\beta}_V$
  and that $\Phi$ is globally $\gamma$-H\"{o}lder.
  
  It suffices to show that $f_t \assign \Phi (t, T, x_t) - \Phi (0, T, x_0)$
  satisfies $\| f_{s, t} \|_V \lesssim | t - s |^{1 + \varepsilon}$ for some
  $\varepsilon > 0$; if that's the case, then $f \equiv 0$, $\Phi (t, T, x_t)
  = \Phi (0, T, x_0)$ for all $t \in [0, T]$ and so inverting the flow $x_t =
  \Phi (0, t, x_0)$, which implies that $\Phi (0, \cdot, x_0)$ is the unique
  solution starting from $x_0$.
  
  By the flow property
  
  \begin{align*}
    \| f_{s, t} \|_V & = \| \Phi (t, T, x_t) - \Phi (s, T, x_s) \|_V\\
    & = \| \Phi (t, T, x_t) - \Phi (t, T, \Phi (s, t, x_s)) \|_V\\
    & \lesssim \| x_t - \Phi (s, t, x_s) \|_V^{\gamma} .
  \end{align*}
  
  Since both $x$ and $\Phi (s, \cdot, x_s)$ are solutions to the YDE starting
  from $x_s$, it holds
  
  \begin{align*}
    \| x_t - \Phi (s, t, x_s) \|_V & = \left\| \int_s^t A (\mathd r, x_r) -
    \int_s^t A (\mathd r, \Phi (s, r, x_s)) \right\|_V\\
    & \lesssim \| A_{s, t} (x_s) - A_{s, t} (\Phi (s, s, x_s)) \|_V + | t - s
    |^{\alpha (1 + \beta)} \| A \|_{\alpha, \beta} (1 + \llbracket x
    \rrbracket_{\alpha} + \llbracket \Phi (s, \cdot, x_s)
    \rrbracket_{\alpha})\\
    & \lesssim | t - s |^{\alpha (1 + \beta)}
  \end{align*}
  
  and so overall we obtain $\| f_{s, t} \|_V \lesssim | t - s |^{\gamma \alpha
  (1 + \beta)}$, which implies the conclusion.
\end{proof}

\begin{remark}
  The assumption can be weakened in several ways. For instance, the existence
  of a $\gamma$-H\"{o}lder regular semiflow is enough to establish that $\Phi
  (t, T, x_t) = \Phi (0, T, x_0)$, even when $\Phi$ is not invertible.
  Uniqueness only requires $\Phi (t, T, \cdot)$ to be invertible for $t \in
  D$, $D$ dense subset of $[0, T]$; indeed this implies $x_t = \Phi (0, t,
  x_0)$ on $D$ and then by continuity the equality can be extended to the
  whole $[0, T]$. Similarly, it is enough to require
  \[ \sup_{t \in D} \| \Phi (t, T, \cdot) \|_{\gamma, R} < \infty \quad
     \text{for all } R \geqslant 0 \]
  for $D$ dense subset of $[0, T]$ as before.
\end{remark}

\subsection{Averaged translations and Conditional Comparison
Principle}\label{sec5.2}

The concept of averaged translation has been introduced
in~{\cite{catelliergubinelli}}, Definition~2.13. We provide here a different
construction based on the sewing lemma (although with the same underlying
idea).

\begin{definition}
  Let $A \in C^{\alpha}_t C^{\beta}_V$, $y \in C^{\gamma}_t V$ with $\alpha +
  \beta \gamma > 1$. The averaged translation $\tau_x A$ is defined as
  \[ \tau_y A (t, x) = \int_0^t A (\mathd s, z + y_s) \quad \forall \, t \in
     [0, T], \, z \in V. \]
\end{definition}

\begin{lemma}
  Let $A \in C^{\alpha}_t C^{n + \beta}_V$, $y \in C^{\gamma}_t V$ with
  $\alpha + \beta \gamma > 1$, $\eta \in (0, 1)$ satisfying $\eta < n +
  \beta$, $\alpha + \eta \gamma > 1$. The operator $\tau_y$ is continuous from
  $C^{\alpha}_t C^{n + \beta}_V$ to $C^{\alpha}_t C^{n + \beta - \eta}_V$ and
  there exists $C = C (\alpha, \beta, \gamma, \eta, T)$ s.t.
  \begin{equation}
    \| \tau_y A \|_{\alpha, n + \beta - \eta} \leqslant C \| A \|_{\alpha, n +
    \beta} (1 + \llbracket y \rrbracket_{\gamma}) . \label{sec conditional
    average translation}
  \end{equation}
\end{lemma}

\begin{proof}
  Observe that $\tau_y A$ corresponds to the sewing of $\Gamma : \Delta_2
  \rightarrow C^{n + \beta}_V$ given by
  \[ \Gamma_{s, t} : = A_{s, t} \left( \, \cdot \, + y_s \right) . \]
  It holds $\| \Gamma_{s, t} \|_{n + \beta} \leqslant | t - s |^{\alpha}  \| A
  \|_{\alpha, n + \beta}$; moreover by Lemma~\ref{appendix conditional lemma
  translations} in Appendix~\ref{appendixA} it holds
  
  \begin{align*}
    \| \delta \Gamma_{s, u, t} \|_{n + \beta - \eta} & = \left\| A_{u, t}
    \left( \, \cdot \, + y_s \right) - A_{u, t} \left( \, \cdot \, + y_u
    \right) \right\|_{n + \beta - \eta}\\
    & \lesssim \| y_s - y_u \|_V^{\eta}  \| A_{u, t} \|_{n + \beta}\\
    & \lesssim | t - s |^{\alpha + \gamma \eta}  \llbracket y
    \rrbracket_{\gamma}  \| A \|_{\alpha, n + \beta} .
  \end{align*}
  
  Since $\alpha + \gamma \eta > 1$, by the sewing lemma we deduce that
  $\mathcal{J} (\Gamma) = \tau_y A \in C^{\alpha}_t C^{n + \beta - \eta}_V$,
  together with estimate~{\eqref{sec conditional average translation}}.
\end{proof}

Young integrals themselves can indeed be regarded as averaged translations
evaluated at $z = 0$. Moreoveor iterating translations is a consistent
procedure, as the following lemma shows.

\begin{lemma}
  \label{sec5.2 basic lemma}Assume that $\alpha + \beta \gamma > 1$ and $A \in
  C^{\alpha}_t C^{\beta}_V$, $x \in C^{\gamma}_t V$ and $\tau_x A \in
  C^{\alpha}_t C^{\beta}_V$. Then for any $y \in C^{\gamma}_t V$ it holds
  \[ \int_0^t (\tau_x A) (\mathd s, y_s) = \int_0^t A (\mathd s, x_s + y_s)
     \quad \forall \, t \in [0, T] . \]
\end{lemma}

\begin{proof}
  The statement follows immediately from the observation that for any $s
  \leqslant t$ it holds
  
  \begin{align*}
    \left\| \int_s^t (\tau_x A) (\mathd r, y_r) - \int_s^t A (\mathd r, x_r +
    y_r) \right\| & \lesssim \| (\tau_x A)_{s, t} (y_s) - A_{s, t} (x_s + y_s)
    \| + | t - s |^{\alpha + \beta \gamma}\\
    & \lesssim \left\| \left( A_{s, t} \left( \cdot \, + x_s \right) \right)
    (y_s) - A_{s, t} (x_s + y_s) \right\| + | t - s |^{\alpha + \beta
    \gamma}\\
    & \lesssim | t - s |^{\alpha + \beta \gamma}
  \end{align*}
  
  so that the two integrals must coincide.
\end{proof}

The main reason for introducing averaged translations is the following key
result.

\begin{theorem}[Conditional Comparison Principle]
  \label{sec5.2 condition comparison thm}Let $A^1, A^2 \in C^{\alpha}_t
  C^{\beta}_V$ with $\alpha (1 + \beta) > 1$ for some $\alpha, \beta \in (0,
  1)$ and let $x^i \in C^{\alpha}_t V$ be given solutions respectively to the
  YDE associated to $(x_0^i, A^i)$. Suppose in addition that $x^1$ is such
  that $\tau_{x^1} A^1 \in C^{\alpha}_t \tmop{Lip}_V$. Then there exists $C =
  C (\alpha, \beta, T)$ s.t.
  \begin{equation}
    \| x^1 - x^2 \|_{\alpha} \leqslant C \exp (C \| \tau_{x^1} A^1 \|^{1 /
    \alpha}_{\alpha, 1}) (1 + \| A^2 \|_{\alpha, \beta}^2)  (\| x_0^1 - x_0^2
    \| + \| A^1 - A^2 \|_{\alpha, \beta}) . \label{sec conditional comparison
    inequality}
  \end{equation}
  In particular, uniqueness holds in the class $C^{\alpha}_t V$ to the YDE
  associated to $(x_0^1, A^1)$.
\end{theorem}

\begin{proof}
  The final uniqueness claim immediately follows from inequality~{\eqref{sec
  conditional comparison inequality}}, since in that case we can consider $A^1
  = A^2$, $x^1_0 = x^2_0$. Now let $x^i$ be two solutions as above, then their
  difference $v = x^1 - x^2$ satisfies
  
  \begin{align*}
    v_t & = v_0 + \int_0^t A^1 (\mathd s, x^1_s) - \int_0^t A^2 (\mathd s,
    x^2_s)\\
    & = v_0 + \int_0^t A^1 (\mathd s, x^1_s) - \int_0^t A^1 (\mathd s, v_s +
    x^1_s) + \int_0^t (A^2 - A^1) (\mathd s, x^2_s)\\
    & = v_0 - \int_0^t \tau_{x^1} A^1 (\mathd s, v_s) + \int_0^t \tau_{x^1}
    A^1 (\mathd s, 0) + \int_0^t (A^2 - A^1) (\mathd s, x^2_s)\\
    & = v_0 + \int_0^t B (\mathd s, v_s) + \psi_t
  \end{align*}
  
  where in the third line we applied Lemma~\ref{sec5.2 basic lemma} and we
  take
  \[ B (t, z) = - \tau_{x_1} A^1 (t, z) + \tau_{x_1} A^1 (t, 0), \qquad
     \psi_{\cdot} = \int_0^{\cdot} (A^2 - A^1) (\mathd s, x^2_s) . \]
  By the hypothesis, $B \in C^{\gamma}_t \tmop{Lip}_V$ with $B (t, 0) = 0$ for
  all $t \in [0, T]$, while $\psi \in C^{\alpha}_t V$. Therefore from
  Theorem~\ref{sec3.3 thm gronwall estimate YDE} applied to $v$ we deduce the
  existence of a constant $\kappa_1 = \kappa_1 (\alpha, T)$ such that
  \[ \| x^1 - x^2 \|_{\alpha} \leqslant \kappa_1 \exp (\kappa_1 \llbracket
     \tau_{x^1} A^1 \rrbracket_{1, \alpha}^{1 / \alpha}) (\| x_0^1 - x_0^2
     \|_V + \llbracket \psi \rrbracket_{\alpha}) . \]
  On the other hand, estimates~{\eqref{sec2 nonlinear integral estimate2}}
  and~{\eqref{sec3.3 prop a priori estimate 1}} imply that
  \[ \llbracket \psi \rrbracket_{\alpha} \leqslant \kappa_2  \| A^1 - A^2
     \|_{\alpha, \beta}  (1 + \| A^2 \|_{\alpha, \beta}^2) \]
  for some $\kappa_2 = \kappa_2 (\alpha, \beta, T)$. Combining the above
  estimates the conclusion follows.
\end{proof}

Remarkably, the hypothesis $\tau_x A \in C^{\alpha}_t \tmop{Lip}_V$ allows not
only to show that this is the unique solution starting at $x_0$, but also that
any other solution will not get too close to it. In the next lemma, in order
to differentiate $\| \cdot \|_V$, we assume for simplicity $V$ to be a Hilbert
space, but a uniformly smooth Banach space would suffice.

\begin{lemma}
  Let $V$ be a Hilbert space, $A \in C^{\alpha}_t C^{\beta}_V$ with $\alpha (1
  + \beta) > 1$, $x, y \in C^{\alpha}_t V$ solutions respectively to the YDEs
  associated to $(x_0, A)$, $(y_0, A)$ and assume that $\tau_x A \in
  C^{\alpha}_t \tmop{Lip}_V$. Then there exists $C = C (\alpha, T)$ s.t.
  \[ \sup_{t \in [0, T]} \frac{\| x_t - y_t \|_V}{\| x_0 - y_0 \|_V} \leqslant
     C \exp (C \| \tau_x A \|_{\alpha, 1}^{1 / \alpha}), \quad \sup_{t \in [0,
     T]} \frac{\| x_0 - y_0 \|_V}{\| x_t - y_t \|_V} \leqslant C \exp (C \|
     \tau_x A \|_{\alpha, 1}^{1 / \alpha}) . \]
\end{lemma}

\begin{proof}
  The first inequality is an immediate consequence of Theorem~\ref{sec5.2
  condition comparison thm}, so we only need to prove the second one. By the
  same computation as in Theorem~\ref{sec5.2 condition comparison thm}, the
  map $v = y - x$ satisfies
  \[ \mathd v_t = A (\mathd t, y_t) - A (\mathd t, x_t) = \tau_x A (\mathd t,
     v_t) - \tau_x A (\mathd t, 0) = B (\mathd t, v_t) \]
  where $B (t, z) : = \tau_x A (t, z) - \tau_x A (t, 0)$, which by hypothesis
  belongs to $C^{\alpha}_t \tmop{Lip}_V$ with $\llbracket B
  \rrbracket_{\alpha, 1} = \llbracket \tau_x A \rrbracket_{\alpha, 1}$;
  moreover $B (t, 0) = 0$ for all $t \in [0, T]$.
  
  Now for $0 < \varepsilon < \| x_0 - y_0 \|_V$, define $T^{\varepsilon} =
  \inf \{ t \in [0, T] : \| x_t - y_t \|_V \leqslant \varepsilon \}$, with the
  convention that $\inf \emptyset = T$; then on $[0, \tau_{\varepsilon}]$ the
  map $z_t \assign \| y_t - x_t \|_V^{- 1} = \| v_t \|_V^{- 1}$ is in
  $C^{\alpha}_t \mathbb{R}$ and by Young chain rule
  \[ \mathd z_t = - \| v_t \|_V^{- 3} \langle v_t, \tilde{A} (\mathd t, v_t)
     \rangle_V . \]
  We are going to show that $z$ satisfies a bound from above which does not
  depend on the interval $[0, T^{\varepsilon}]$; as a consequence, for all
  $\varepsilon > 0$ small enough it must hold $T^{\varepsilon} = T$, which
  yields the conclusion.
  
  For any $[u, r] \subset [s, t] \subset [0, T^{\varepsilon}]$ it holds
  
  \begin{align*}
    | z_{u, r} | & \leqslant \, \| v_u \|_V^{- 3} | \langle v_u, B_{u, r}
    (v_u) \rangle_V | + \kappa_1 \llbracket z \rrbracket_{\alpha ; s, t} 
    \llbracket B \rrbracket_{\alpha, 1} | u - r |^{2 \alpha}\\
    & \leqslant \, \| v_u \|_V^{- 1} \llbracket B \rrbracket_{\alpha, 1} | u
    - r |^{\alpha} + \kappa_1 \llbracket z \rrbracket_{\alpha ; s, t}
    \llbracket B \rrbracket_{\alpha, 1} | t - s |^{\alpha} | u - r
    |^{\alpha}\\
    & \leqslant | z_u | \llbracket \tau_x A \rrbracket_{\alpha, 1} | u - r
    |^{\alpha} + \kappa_1 \llbracket z \rrbracket_{\alpha ; s, t} \llbracket
    \tau_x A \rrbracket_{\alpha, 1} | t - s |^{\alpha} | u - r |^{\alpha}\\
    & \leqslant | u - r |^{\alpha} \llbracket \tau_x A \rrbracket_{\alpha, 1}
    [| z_s | + (1 + \kappa_1) \llbracket z \rrbracket_{\alpha ; s, t} | t - s
    |^{\alpha}] ;
  \end{align*}
  
  dividing by $| u - r |^{\alpha}$ and taking the supremum we obtain
  \[ \llbracket z \rrbracket_{\alpha ; s, s + \Delta} \leqslant \llbracket
     \tau_x A \rrbracket_{\alpha, 1} | z_s | + \kappa_2 \Delta^{\alpha}
     \llbracket \tau_x A \rrbracket_{\alpha, 1} \llbracket z
     \rrbracket_{\alpha} . \]
  The rest of the proof follows exactly the same calculations as in the proof
  of Theorem~\ref{sec3.3 thm gronwall estimate YDE}: taking $\Delta$ such that
  $\kappa_2 \Delta^{\alpha} \llbracket \tau_x A \rrbracket_{\alpha, 1}
  \leqslant 1 / 2$, $\kappa_2 \Delta^{\alpha} \llbracket \tau_x A
  \rrbracket_{\alpha, 1} \sim 1$, we deduce that
  \[ \llbracket z \rrbracket_{\alpha ; s, s + \Delta} \leqslant 2 \llbracket
     \tau_x A \rrbracket_{\alpha, 1} | z_s | ; \]
  setting $J_n = \| z \|_{\infty ; I_n}$ with $I_n = [(n - 1) \Delta, n
  \Delta] \cap [0, T^{\varepsilon}]$, $J_0 = | z_0 |$, it holds
  \[ J_n \leqslant J_{n - 1} + \Delta^{\alpha} \llbracket z \rrbracket_{\alpha
     ; I_n} \leqslant (1 + 2 \Delta^{\alpha} \llbracket \tau_x A
     \rrbracket_{\alpha, 1}) J_{n - 1}, \]
  which implies recursively
  \[ \| z \|_{\infty ; 0, T^{\varepsilon}} = \sup_n J_n \leqslant (1 + 2
     \Delta^{\alpha} \llbracket \tau_x A \rrbracket_{\alpha, 1})^N | z_0 |
     \leqslant \exp (2 N \Delta^{\alpha} \llbracket \tau_x A
     \rrbracket_{\alpha, 1}) | z_0 | . \]
  Since $T^{\varepsilon} \leqslant T$, it takes at most $N \sim T / \Delta$
  intervals of size $\Delta$ to cover $[0, T^{\varepsilon}]$, and $\Delta \sim
  \llbracket \tau_x A \rrbracket_{\alpha, 1}^{1 / \alpha}$, therefore overall
  we have found a constant $C = C (\alpha, T)$ such that
  \[ \sup_{t \in [0, T^{\varepsilon}]} \frac{1}{\| x_s - y_s \|_V} = \sup_{t
     \in [0, T^{\varepsilon}]} | z_t | \leqslant C \exp (C \llbracket \tau_x A
     \rrbracket_{\alpha, 1}^{1 / \alpha}) | z_0 | = C \exp (C \llbracket
     \tau_x A \rrbracket_{\alpha, 1}^{1 / \alpha}) \frac{1}{\| x_0 - y_0 \|_V}
     . \]
  As the estimate does not depend on $\varepsilon$, the conclusion follows.
\end{proof}

\subsection{Conditional rate of convergence for the Euler
scheme}\label{sec5.3}

Remarkably, under the assumption of regularity of $\tau_x A$, convergence of
the Euler scheme to the unique solution can be established, with the same rate
$2 \alpha - 1$ as in the more regular case of $A \in C^{\alpha}_t C^{1 +
\beta}_V$. The following results are direct analogues of
Corollaries~\ref{sec3.3 cor euler convergence} and~\ref{sec3.4 cor euler
convergence}.

\begin{corollary}
  \label{sec5.2 euler convergence1}Let $A \in C^{\alpha}_t \tmop{Lip}_V$ with
  $\alpha > 1 / 2$, $x_0 \in V$ and suppose there exists a solution $x$
  associated to $(x_0, A)$ such that $\tau_x A \in C^{\alpha}_t \tmop{Lip}_V$
  (which is therefore the unique solution); denote by $x^n$ the element of
  $C^{\alpha}_t V$ constructed by the $n$-step Euler approximation from
  Theorem~\ref{sec3.1 thm existence}. Then there exists $C = C (\alpha, T)$
  such that
  \[ \| x - x^n \|_{\alpha} \leqslant C \exp (C \| \tau_x A \|^{1 /
     \alpha}_{\alpha, 1}) (1 + \| A \|_{\alpha, 1}^3) n^{1 - 2 \alpha} \quad
     \text{as } n \rightarrow \infty . \]
\end{corollary}

\begin{proof}
  As in the proof of Corollary~\ref{sec3.3 cor euler convergence}, recall that
  $x^n$ satisfies the YDE
  \[ x^n = x_0 + \int_0^t A (\mathd s, x^n_s) + \psi^n_t, \qquad \llbracket
     \psi^n \rrbracket_{\alpha} \lesssim_{\alpha, T} \, (1 + \| A \|_{\alpha,
     1}^3) n^{1 - 2 \alpha} . \]
  Therefore $v^n = x^n - x$ satisfies
  \[ v^n_t = \int_0^t B (\mathd s, v^n_s) + \psi^n_t, \quad B (t, z) = \tau_x
     A (t, z) - \tau_x A (t, 0), \quad \llbracket B \rrbracket_{\alpha, 1} =
     \llbracket \tau_x A \rrbracket_{\alpha, 1} . \]
  Applying Theorem~\ref{sec3.3 thm gronwall estimate YDE} we obtain that, for
  suitable $\kappa = \kappa (\alpha, T)$ it holds
  \[ \| x - x^n \|_{\alpha} \leqslant \kappa \exp (\kappa \| \tau_x A \|^{1 /
     \alpha}_{\alpha, 1})  \llbracket \psi^n \rrbracket_{\alpha} \]
  which combined with the above inequality for $\llbracket \psi^n
  \rrbracket_{\alpha}$ gives the conclusion.
\end{proof}

\begin{corollary}
  \label{sec5.2 euler convergence2}Let $A$ be such that $A \in C^{\alpha}_t
  C^{\beta}_V$ and $\partial_t A \in C^0 ([0, T] \times V ; V)$ with $\alpha
  (1 + \beta) > 1$, $x_0 \in V$ and suppose there exists a solution $x$
  associated to $(x_0, A)$ such that $\tau_x A \in C^{\alpha}_t \tmop{Lip}_V$
  (which is therefore the unique solution); denote by $x^n$ the element of
  $C^{\alpha}_t V$ constructed by the $n$-step Euler approximation from
  Theorem~\ref{sec3.1 thm existence}. Then there exists $C = C (\alpha, T)$
  such that
  \[ \| x - x^n \|_{\alpha} \leqslant C \exp (C \| \tau_x A \|^{1 /
     \alpha}_{\alpha, 1}) \| A \|_{\alpha, 1} \| \partial_t A \|_{\infty} n^{-
     \alpha} \quad \text{as } n \rightarrow \infty . \]
\end{corollary}

\begin{proof}
  Recall that $x^n$ satisfies the YDE
  \[ x^n = x_0 + \int_0^t A (\mathd s, x^n_s) + \psi^n_t, \qquad \llbracket
     \psi^n \rrbracket_{\alpha} \lesssim_{\alpha, T} \| A \|_{\alpha, 1} \|
     \partial_t A \|_{\infty} n^{- \alpha} . \]
  The rest of the proof is mostly identical to that of Corollary~\ref{sec5.2
  euler convergence1}.
\end{proof}

\section{Young transport equations}\label{sec6}

This section is devoted to the study of Young transport equations of
the form
\begin{equation}
  u_{\mathd t} + A_{\mathd t} \cdot \nabla u_t + c_{\mathd t} u_t = 0.
  \label{sec6 temp YTE}
\end{equation}
which we will refer to as the YTE associated to $(A, c)$.

We restrict here to the case $V =\mathbb{R}^d$; as in Section~\ref{sec4} for
simplicity we will assume on $A$ global bounds like $A \in C^{\alpha}_t C^{1 +
\beta}_x$, but slightly more tedious localisation arguments allow to relax
them to growth conditions and local regularity requirements.

\

Classical results on weak solutions to~{\eqref{sec6 temp YTE}} in the case
$A_{\mathd t} = b_t \mathd t$, $c_{\mathd t} = \tilde{c}_t \mathd t$ can be
found in~{\cite{diperna}},~{\cite{ambrosio}}. Our approach here mostly follows
the one given in~{\cite{galeatigubinelli1}}, although slightly less based on
the method of characteristics and more on a duality approach; other works
concerning transport equations in the Young (or ``level-1'') regime are given
by~{\cite{catellier}},~{\cite{hu}} and Chapter~9 from~{\cite{maurelli}}. Let
us also mention on a different note the
works~{\cite{bailleul}}~{\cite{diehl2017}},~{\cite{bellingeri}} which treat
with different techniques and in various regularity regimes rough trasnport
equations of ``level-2'' or higher (namely corresponding to a time regularity
$\alpha \leq 1 / 2$).

\

Before explaining the meaning of~{\eqref{sec6 temp YTE}}, we need some
preparations. Given any compact $K \subset \mathbb{R}^d$, we denote by
$C^{\beta}_K = C^{\beta}_K (\mathbb{R}^d)$ the Banach space of $f \in
C^{\beta} (\mathbb{R}^d)$ with $\tmop{supp} f \subset K$; $C^{\beta}_c =
C^{\beta}_c (\mathbb{R}^d)$ is the set of all compactly supported
$\beta$-H\"{o}lder continuous functions. $C^{\beta}_c$ is a direct limit of
Banach spaces and thus it is locally convex; we denote its topological dual by
$(C^{\beta}_c)^{\ast}$. Given $\gamma, \beta \in (0, 1)$, we say that $f \in
C^{\alpha}_t C^{\beta}_c$ if there exists a compact $K$ such that $f \in
C^{\alpha}_t C_K^{\beta}$; similarly, a distribution $u \in C^{\gamma}_t
(C^{\beta}_c)^{\ast}$ if $u \in C^{\gamma}_t (C^{\beta}_K)^{\ast}$ for all
compact $K \subset \mathbb{R}^d$. We will use the bracket $\langle \cdot,
\cdot \rangle$ to denote both the classical $L^2$-pairing and the one between
$C^{\beta}_c$ and its dual. Finally, $M_{\tmop{loc}}$ denotes the space of
Radon measures on $\mathbb{R}^d$, $M_K$ the space of finite signed measure
supported on $K$; observe that the above notation is consistent with
$M_{\tmop{loc}} = (C^0_c)^{\ast}$.

We are now ready to give a notion of solution to the YTE.

\begin{definition}
  \label{sec6 temp defn solution}Let $\alpha, \beta \in (0, 1)$ such that
  $\alpha (1 + \beta) > 1$.We say that $u \in L^{\infty}_t M_{\tmop{loc}} \cap
  C^{\alpha \beta}_t (C_c^{\beta})^{\ast}$ is a weak solution to the YTE
  associated to $A \in C^{\alpha}_t C^{\beta}_x$, $c \in C^{\alpha}_t
  C^{\beta}_x$ with $\tmop{div} A \in C^{\alpha}_t C^{\beta}_x$ if
  \begin{equation}
    \langle u_t, \varphi \rangle - \langle u_0, \varphi \rangle = \int_0^t
    \langle A_{\mathd s} \cdot \nabla \varphi + (\tmop{div} A_{\mathd s} -
    c_{\mathd s}) \varphi, u_s \rangle \quad \forall \, \varphi \in
    C^{\infty}_c . \label{sec6 temp YTE2}
  \end{equation}
\end{definition}

Observe that under the above assumptions, for any $\varphi \in C^{\infty}_c$,
$A \cdot \nabla \varphi$ and $(\tmop{div} A - c) \varphi$ belong to
$C^{\alpha}_t C^{\beta}_c$; since $u \in C^{\alpha \beta}_t
(C^{\beta}_c)^{\ast}$ with $\alpha (1 + \beta) > 1$, the integral appearing
in~{\eqref{sec6 temp YTE2}} is meaningful as a functional Young integral.

\begin{remark}
  \label{sec6 temp remark defn}For practical purposes, it is useful to
  consider the following equivalent characterization of solutions: under the
  above regularity assumptions on $u$, $A$, $c$, $u$ is a solution if and only
  if for any compact $K \subset \mathbb{R}^d$ and $\varphi \in C^{\infty}_K$
  it holds
  \begin{eqnarray}
    | \langle u_{s, t}, \varphi \rangle - \langle A_{s, t} \cdot \nabla
    \varphi + (\tmop{div} A_{s, t} - c_{s, t}) \varphi, u_s \rangle | &
    \lesssim_K & \| \varphi \|_{C^{1 + \beta}_K} | t - s |^{\alpha (1 +
    \beta)} \llbracket u \rrbracket_{C^{\alpha \beta}_t (C^{\beta}_K)^{\ast}}
    \times \nonumber\\
    &  & \times (\| A \|_{\alpha, \beta} + \| \tmop{div} A - c \|_{\alpha,
    \beta}) . \label{sec6 temp YTE characterization} 
  \end{eqnarray}
  Clearly in the l.h.s. above one can replace $u_s$ with $u_t$ to get a
  similar estimate.
\end{remark}

\begin{remark}
  The presence of $c$ in~{\eqref{sec6 temp YTE}} allows to also consider
  nonlinear Young continuity equations (YCE for short) of the form
  \[ v_{\mathd t} + \nabla \cdot (A_{\mathd t} v_t) + c_{\mathd t} v_t = 0 ;
  \]
  weak solutions to the above equation must be understood as weak solutions to
  the YTE associated to $(A, \tilde{c})$ with $\tilde{c} = c + \nabla \cdot
  A$.
\end{remark}

Let us quickly recall some results from Section~\ref{sec4}: given $A \in
C^{\alpha}_t C^{1 + \beta}_x$, the YDE admits a flow of diffeomorphisms
$\Phi_{s \rightarrow t} (x)$ and there exists $C = C (\alpha, \beta, T, \| A
\|_{\alpha, 1 + \beta})$ such that
\begin{eqnarray*}
  \| \Phi_{s \rightarrow \cdot} (x) - \Phi_{s \rightarrow \cdot} (y)
  \|_{\alpha ; s, T} & \leqslant & C | x - y |\\
  | \Phi_{s \rightarrow t} (x) - x | & \leqslant & C | t - s |^{\alpha}\\
  \llbracket \Phi_{s \rightarrow \cdot} (x) \rrbracket_{\alpha ; s, T} + | D_x
  \Phi_{s \rightarrow t} (x) | & \leqslant & C
\end{eqnarray*}
for all $x, y \in \mathbb{R}^d$, $(s, t) \in \Delta_2$, together with similar
estimates for $\Phi_{\cdot \leftarrow t}$. Moreover
\[ \det D \Phi_{s \rightarrow t} (x) = \exp \left( \int_s^t \tmop{div} A
   (\mathd r, \Phi_{s \rightarrow r} (x)) \right) \]
and similarly
\[ \det D \Phi_{s \leftarrow t} (x) = (\det D \Phi_{s \rightarrow t} (\Phi_{s
   \leftarrow t} (x)))^{- 1} = \exp \left( - \int_s^t \tmop{div} A (\mathd r,
   \Phi_{r \leftarrow t} (x)) \right) . \]
\begin{proposition}
  \label{sec6 temp prop existence}Let $A \in C^{\alpha}_t C^{1 + \beta}_x$, $c
  \in C^{\alpha}_t C^{\beta}_x$. Then for any $\mu_0 \in M_{\tmop{loc}}$, a
  solution to the YTE is given by the formula
  \begin{equation}
    \langle u_t, \varphi \rangle = \int \varphi (\Phi_{0 \rightarrow t} (x))
    \exp \left( \int_0^t (\tmop{div} A - c) (\mathd s, \Phi_{0 \rightarrow s}
    (x)) \right) \mu_0 (\mathd x) \quad \forall \, \varphi \in C^{\infty}_c .
    \label{sec6 temp candidate solution}
  \end{equation}
  If $\mu_0 (\mathd x) = u_0 (x) \mathd x$ for $u_0 \in L^p_{\tmop{loc}}$,
  then $u_t$ corresponds to the measurable function
  \begin{equation}
    u (t, x) = u_0 (\Phi_{0 \leftarrow t} (x)) \exp \left( - \int_0^t c
    (\mathd s, \Phi_{s \leftarrow t} (\nobracket x)) \nobracket \right)
    \label{sec6 temp candidate solution2}
  \end{equation}
  which belongs to $L^{\infty}_t L^p_{\tmop{loc}}$ and satisfies
  \[ \int_K | u (t, x) |^p \mathd x = \int_{\Phi_{0 \leftarrow t} (K)} | u_0
     (x) |^p \exp \left( \int_0^t (\tmop{div} A - c) (\mathd s, \Phi_{0
     \rightarrow s} (x)) \right) . \]
  If in addition $c \in C^{\alpha}_t C^{1 + \beta}_x$, then for any $u_0 \in
  C^1_{\tmop{loc}}$ it holds $u \in C^{\alpha}_t C^0_{\tmop{loc}} \cap C^0_t
  C^1_{\tmop{loc}}$.
\end{proposition}

\begin{proof}
  Since $| \Phi_{0 \rightarrow t} (x) - x | \lesssim T^{\alpha}$, it is always
  possible to find $R \geqslant 0$ big enough such that $\tmop{supp} \varphi
  (\Phi_{0 \rightarrow t} (\cdot)) \subset \tmop{supp} \varphi + B_R$ for all
  $t \in [0, T]$; by estimates~{\eqref{sec2 nonlinear integral estimate2}}
  and~{\eqref{sec3.3 a priori estimate1}}, it holds
  \[ \sup_{(t, x) \in [0, T] \times \mathbb{R}^d} \left| \int_0^t (\tmop{div}
     A - c) (\mathd s, \Phi_{0 \rightarrow s} (x)) \right| \lesssim \|
     \tmop{div} A - c \|_{\alpha, \beta} \sup_{x \in \mathbb{R}^d} (1 +
     \llbracket \Phi_{0 \rightarrow \cdot} (x) \rrbracket_{\alpha}) < \infty .
  \]
  It is therefore clear that $u_t$ defined as in~{\eqref{sec6 temp candidate
  solution}} belongs to $L^{\infty}_t (C^0_c)^{\ast}$. Similarly, combining
  the estimates
  
  \begin{align*}
    | \varphi (\Phi_{0 \rightarrow t} (x)) - \varphi (\Phi_{0 \rightarrow s}
    (x)) | & \leqslant | t - s |^{\alpha \beta} \llbracket \varphi
    \rrbracket_{\beta} \llbracket \Phi_{0 \rightarrow \cdot} (x)
    \rrbracket_{\alpha}^{\beta} \lesssim | t - s |^{\alpha \beta} \llbracket
    \varphi \rrbracket_{\beta}\\
    \left| \int_s^t (\tmop{div} A - c) (\mathd s, \Phi_{0 \rightarrow s} (x))
    \right| & \lesssim | t - s |^{\alpha}  \| \tmop{div} A - c \|_{\alpha,
    \beta} (1 + \llbracket \Phi_{0 \rightarrow \cdot} (x) \rrbracket_{\alpha})
    \lesssim | t - s |^{\alpha},
  \end{align*}
  
  it is easy to check that $u \in C^{\alpha \beta}_t (C^{\beta}_c)^{\ast}$.
  
  Let us show that it is a solution to the YTE in the sense of
  Definition~\ref{sec6 temp defn solution}. Given $\varphi \in C^{\infty}_K$
  and $x \in \mathbb{R}^d$, define
  \[ z_t (x) \assign \varphi (\Phi_{0 \rightarrow t} (x)) \exp \left( \int_0^t
     (\tmop{div} A - c) (\mathd s, \Phi_{0 \rightarrow s} (x)) \right) . \]
  By It\^{o} formula, $z$ satisfies
  \begin{eqnarray*}
    z_{s, t} (x) & = & \int_s^t \varphi (\Phi_{0 \rightarrow r} (x)) \exp
    \left( \int_0^r (\tmop{div} A - c) (\mathd s, \Phi_{0 \rightarrow s} (x))
    \right) (\tmop{div} A - c) (\mathd r, \Phi_{0 \rightarrow r} (x))\\
    &  & + \int_s^t \exp \left( \int_0^r (\tmop{div} A - c) (\mathd s,
    \Phi_{0 \rightarrow s} (x)) \right) \nabla \varphi (\Phi_{0 \rightarrow r}
    (x)) \cdot A (\mathd r, \Phi_{0 \rightarrow r} (x)) .
  \end{eqnarray*}
  By the properties of Young integrals and the above estimates, which are
  uniform in $x$, it holds
  \begin{eqnarray*}
    z_{s, t} (x) & \sim & \exp \left( \int_0^s (\tmop{div} A - c) (\mathd r,
    \Phi_{0 \rightarrow r} (x)) \right) \times\\
    &  & \times [\varphi (\Phi_{0 \rightarrow s} (x)) (\tmop{div} A - c)_{s,
    t} (\Phi_{0 \rightarrow s} (x)) + \nabla \varphi (\Phi_{0 \rightarrow s}
    (x)) \cdot A_{s, t} (\Phi_{0 \rightarrow s} (x))]
  \end{eqnarray*}
  in the sense that the two quantities differ by $O (| t - s |^{\alpha (1 +
  \beta)})$, uniformly in $x \in \mathbb{R}^d$. Therefore
  
  \begin{align*}
    \langle u_{s, t}, \varphi \rangle & = \int_{K + B_R} z_{s, t} (x) \mu_0
    (\mathd x)\\
    & \sim \int_{K + B_R} [A_{s, t} \cdot \nabla \varphi + (\tmop{div} A -
    c)_{s, t} \varphi] (\Phi_{0 \rightarrow t} (x)) \exp \left( \int_0^s
    (\tmop{div} A - c) (\mathd r, \Phi_{0 \rightarrow r} (x)) \right) \mu_0
    (\mathd x)\\
    & \sim \langle u_s, A_{s, t} \cdot \nabla \varphi + (\tmop{div} A -
    c)_{s, t} \varphi \rangle
  \end{align*}
  
  where the two quantities differ by $O (\| \varphi \|_{C^{1 + \beta}_K} | t -
  s |^{\alpha (1 + \beta)})$. By Remark~\ref{sec6 temp remark defn} we deduce
  that $u$ is indeed a solution.
  
  The statements for $u_0 \in L^p_{\tmop{loc}}$ are an easy application of
  formula~{\eqref{sec4.1 formula jacobian}}; it remains to prove the claims
  for $u_0 \in C^1_{\tmop{loc}}$, under the additional assumption $c \in
  C^{\alpha}_t C^{1 + \beta}_x$. First of all observe that, for any $(s, t)
  \in \Delta_2$, it holds
  \begin{equation}
    \| \Phi_{\cdot \leftarrow t} (x) - \Phi_{\cdot \leftarrow s} (x)
    \|_{\alpha} = \| \Phi_{\cdot \leftarrow s} (\Phi_{s \leftarrow t} (x)) -
    \Phi_{\cdot \leftarrow s} (x) \|_{\alpha} \lesssim | \Phi_{s \leftarrow t}
    (x) - x | \lesssim | t - s |^{\alpha} ; \label{sec6 temp existence proof
    eq1}
  \end{equation}
  as a consequence, the map $(t, x) \mapsto u_0 (\Phi_{0 \leftarrow t} (x))$
  belongs to $C^{\alpha}_t C^0_{\tmop{loc}}$. Consider now the map
  \[ g (t, x) \assign \int_0^t c (\mathd r, \Phi_{r \leftarrow t} (x)) . \]
  It holds
  
  \begin{align*}
    \int_0^t c (\mathd r, \Phi_{r \leftarrow t} (x)) - \int_0^s c (\mathd r,
    \Phi_{r \leftarrow s} (x)) & = \int_s^t c (\mathd r, \Phi_{r \leftarrow t}
    (x)) + \int_0^s [c (\mathd r, \Phi_{r \leftarrow t} (x)) - c (\mathd r,
    \Phi_{r \leftarrow s} (x))] ;
  \end{align*}
  
  by Corollary~\ref{sec2 corollary frechet} and estimate~{\eqref{sec6 temp
  existence proof eq1}} we have
  \begin{eqnarray*}
    \left\| \int_0^{\cdot} [c (\mathd r, \Phi_{r \leftarrow t} (x)) - c
    (\mathd r, \Phi_{r \leftarrow s} (x))] \right\|_{\alpha} & \lesssim & \| c
    \|_{\alpha, 1 + \beta} (1 + \llbracket \Phi_{\cdot \leftarrow t} (x)
    \rrbracket_{\alpha} + \llbracket \Phi_{\cdot \leftarrow s} (x)
    \rrbracket_{\alpha}) \times\\
    &  & \times \| \Phi_{\cdot \leftarrow t} (x) - \Phi_{\cdot \leftarrow s}
    (x) \|_{\alpha}\\
    & \lesssim & | t - s |^{\alpha} .
  \end{eqnarray*}
  As a consequence, $g \in C^{\alpha}_t C^0_{\tmop{loc}}$ and so does $u$. The
  verification that $u \in C^0_t C^1_{\tmop{loc}}$ is similar and thus
  omitted.
\end{proof}

\begin{remark}
  \label{sec6 temp remark backward solution}Analogous computations show that a
  solution to the YTE with terminal condition $u (T, \cdot) = \mu_T (\cdot)$
  is given by
  \[ \langle u_t, \varphi \rangle = \int \varphi (\Phi_{t \leftarrow T} (x))
     \exp \left( \int_t^T (c - \tmop{div} A) (\mathd s, \Phi_{s \leftarrow T}
     (x)) \right) \mu_T (\mathd x) \quad \forall \, \varphi \in C^{\infty}_c ;
  \]
  in the case $\mu_T (\mathd x) = u_T (x) \mathd x$ with $u_T \in
  L^p_{\tmop{loc}}$ it corresponds to
  \[ u_t (x) = u_T (\Phi_{t \rightarrow T} (x)) \exp \left( \int_t^T c (\mathd
     s, \Phi_{t \rightarrow s} (x)) \right) . \]
  This solution satisfies the same space-time regularity as in
  Proposition~\ref{sec6 temp prop existence}. Moreover by the properties of
  the flow, if $\mu_0$ (resp. $\mu_T$) has compact support, then it's possible
  to find $K \subset \mathbb{R}^d$ compact such that $\tmop{supp} u_t \subset
  K$ uniformly in $t \in [0, T]$. In particular if $c \in C^{\alpha}_t C^{1 +
  \beta}_x$ and $u_0 \in C^1_c$ (resp. $u_T \in C^1_c$), then the associated
  solution belongs to $C^{\alpha}_t C^0_c \cap C^0_t C^1_c$.
\end{remark}

The following result is at the heart of the duality approach and our main tool
to establish uniqueness.

\begin{proposition}
  \label{sec6 temp prop duality}Let $u \in C^{\alpha}_t C^0_c \cap C^0_t
  C^1_c$ be a solution of the YTE
  \begin{equation}
    u_{\mathd t} + A_{\mathd t} \cdot \nabla u_t + c_{\mathd t} u_t = 0
    \label{sec6 temp YTE dual}
  \end{equation}
  and let $v \in L^{\infty}_t (C^0_c)^{\ast} \cap C^{\alpha \beta}_t
  (C^{\beta}_c)^{\ast}$ be a solution to the YCE
  \begin{equation}
    v_{\mathd t} + \nabla \cdot (A_{\mathd t} v_t) - c_{\mathd t} v_t = 0.
    \label{sec6 temp YCE dual}
  \end{equation}
  Then it holds $\langle v_t, u_t \rangle = \langle v_s, u_s \rangle$ for all
  $(s, t) \in \Delta_2$. A similar statement holds for $u \in C^{\alpha}_t
  C^0_{\tmop{loc}} \cap C^0_t C^1_{\tmop{loc}}$ and $v$ as above and compactly
  supported uniformly in time.
\end{proposition}

The proof requires some preparations. Let $\{ \rho_{\varepsilon}
\}_{\varepsilon > 0}$ be a family of standard spatial mollifiers (say $\rho_1$
supported on $B_1$ for simplicity) and define the $R^{\varepsilon}$, for
sufficiently regular $g$ and $h$, as the following bilinear operator:
\begin{equation}
  R^{\varepsilon} (g, h) = (g \cdot \nabla h)^{\varepsilon} - g \cdot \nabla
  h^{\varepsilon} = \rho^{\varepsilon} \ast (g \cdot \nabla h) - g \cdot
  \nabla (\rho^{\varepsilon} \ast h) ; \label{sec6 defn commutator}
\end{equation}
the following commutator lemma is a slight variation on Lemma~16, Section~5.2
from~{\cite{galeatigubinelli1}}, which in turn is inspired by the general
technique first introduced in~{\cite{diperna}}.

\begin{lemma}
  \label{sec6 temp commutator lemma}The operator $R^{\varepsilon} :
  C_{\tmop{loc}}^{1 + \beta} \times C_{\tmop{loc}}^1 \rightarrow
  C_{\tmop{loc}}^{\beta}$ defined by~{\eqref{sec6 defn commutator}} satisfies
  the following.
  \begin{enumerateroman}
    \item There exists a constant $C$ independent of $\varepsilon$ and $R$
    such that
    \[ \| R^{\varepsilon} (g, h) \|_{\beta, R} \leqslant C \| g \|_{1 + \beta,
       R + 1} \| h \|_{\beta, R + 1} . \]
    \item For any fixed $g \in C^{1 + \beta}_{\tmop{loc}}, h \in
    C^{\beta}_{\tmop{loc}}$ it holds $R^{\varepsilon} (g, h) \rightarrow 0$ in
    $C^{\beta'}_{\tmop{loc}}$ as $\varepsilon \rightarrow 0$, for any $\beta'
    < \beta$.
  \end{enumerateroman}
\end{lemma}

\begin{proof}
  It holds
  
  \begin{align*}
    R^{\varepsilon} (g, h) (x) & = \int_{B_1} h (x - \varepsilon z)  \frac{g
    (x - \varepsilon z) - g (x)}{\varepsilon} \cdot \nabla \rho (z) \mathd z -
    (h \tmop{div} g)^{\varepsilon} (x)\\
    & \backassign \tilde{R}^{\varepsilon} (g, h) (x) - (h \tmop{div}
    g)^{\varepsilon} (x) .
  \end{align*}
  
  Thus claim~\tmtextit{i.} follows from $\| (h \tmop{div} g)^{\varepsilon}
  \|_{\beta, R} \leqslant \| h \|_{1, R + 1} \| g \|_{1 + \beta, R + 1}$ and
  \begin{eqnarray*}
    | \tilde{R}^{\varepsilon} (g, h) (x) - \tilde{R}^{\varepsilon} (g, h) (y)
    | & \leqslant & \left| \int_{B_1} [h (x - \varepsilon z) - h (y -
    \varepsilon z)] \frac{g (x - \varepsilon z) - g (x)}{\varepsilon} \cdot
    \nabla \rho (z) \mathd z \right|\\
    &  & + \, \left| \int_{B_1} h (x - \varepsilon z) \left[ \frac{g (x -
    \varepsilon z) - g (x)}{\varepsilon} - \frac{g (y - \varepsilon z) - g
    (y)}{\varepsilon} \right] \cdot \nabla \rho (z) \mathd z \right|\\
    & \leqslant & | x - y |^{\beta} \| h \|_{\beta, R + 1} \| g \|_{1, R + 1}
    \| \nabla \rho \|_{L^1}\\
    &  & + \| h \|_{0, R + 1} \int_{B_1} \left| \int_0^1 [\nabla g (x -
    \varepsilon \theta z) - \nabla g (x) - \nabla g (y - \varepsilon \theta z)
    + \nabla g (y)] \right| \times\\
    &  & \quad \times | z | | \nabla \rho (z) | \mathd z\\
    & \lesssim & | x - y |^{\beta} \| h \|_{\beta, R + 1} \| g \|_{1 + \beta,
    R + 1}
  \end{eqnarray*}
  where the estimate is uniform in $x, y \in B_R$ and in $\varepsilon > 0$.
  Claim~\tmtextit{ii.} follows from the above uniform estimate, the fact that
  $R^{\varepsilon} (g, h) \rightarrow 0$ $C^0_{\tmop{loc}}$ by Lemma~16
  from~{\cite{galeatigubinelli1}} and an interpolation argument.
\end{proof}

\begin{proof}[of Proposition~\ref{sec6 temp prop duality}]
  We only treat the case $u \in C^{\alpha}_t C^0_c \cap C^0_t C^1_c$, $v \in
  L^{\infty}_t (C^0_c)^{\ast} \cap C^{\alpha \beta}_t (C^{\beta}_c)^{\ast}$,
  the other one being similar. Applying a mollifier $\rho^{\varepsilon}$ on
  both sides of~{\eqref{sec6 temp YTE dual}}, it holds
  \[ u^{\varepsilon}_{\mathd t} + A_{\mathd t} \cdot \nabla u^{\varepsilon}_t
     + (c_{\mathd t} u_t)^{\varepsilon} + R^{\varepsilon} (A_{\mathd t}, u_t)
     = 0 \]
  where we used the definition of $R^{\varepsilon}$; equivalently by
  Remark~\ref{sec6 temp remark defn}, the above expression can be interpreted
  as
  \[ \| u^{\varepsilon}_{s, t} + A_{s, t} \cdot \nabla u^{\varepsilon}_s +
     (c_{s, t} u_s)^{\varepsilon} + R^{\varepsilon} (A_{s, t}, u_s) \|_{C^0}
     \lesssim_{\varepsilon} | t - s |^{\alpha (1 + \beta)} \quad
     \text{uniformly in } (s, t) \in \Delta_2 \]
  Since $v$ is a weak solution to~{\eqref{sec6 temp YCE dual}}, it holds
  
  \begin{align*}
    \langle u^{\varepsilon}_t, v_t \rangle - \langle u^{\varepsilon}_s, v_s
    \rangle & = \, \langle u^{\varepsilon}_{s, t}, v_s \rangle + \langle
    u^{\varepsilon}_t, v_{s, t} \rangle\\
    & \sim_{\varepsilon} \, - \langle A_{s, t} \cdot \nabla u^{\varepsilon}_t
    + (c_{s, t} u_t)^{\varepsilon} + R^{\varepsilon} (A_{s, t}, u_t), v_s
    \rangle + \langle A_{s, t} \cdot \nabla u^{\varepsilon}_t + c_{s, t}
    u^{\varepsilon}_t, v_s \rangle\\
    & \sim \, \langle c_{s, t} u_t^{\varepsilon} - (c_{s, t}
    u_t)^{\varepsilon} - R^{\varepsilon} (A_{s, t}, u_t), v_s \rangle
  \end{align*}
  
  where by $a \sim_{\varepsilon} b$ we mean that $| a - b |
  \lesssim_{\varepsilon} | t - s |^{\alpha (1 + \beta)}$. As a consequence,
  defining $f^{\varepsilon}_t \assign \langle u^{\varepsilon}_t, v_t \rangle$,
  we deduce that $f^{\varepsilon}_t - f^{\varepsilon}_0 = J
  (\Gamma^{\varepsilon}_{s, t})$ for the choice
  \[ \Gamma^{\varepsilon}_{s, t} \assign \langle c_{s, t} u_t^{\varepsilon} -
     (c_{s, t} u_t)^{\varepsilon} - R^{\varepsilon} (A_{s, t}, u_t), v_s
     \rangle . \]
  Our aim is to show that $J (\Gamma^{\varepsilon}_{s, t}) \rightarrow 0$ as
  $\varepsilon \rightarrow 0$; to this end, we start estimating $\|
  \Gamma^{\varepsilon} \|_{\alpha, \alpha (1 + \beta)}$.
  
  It holds
  \begin{eqnarray*}
    \delta \Gamma^{\varepsilon}_{s, r, t} & = & \langle c_{s, r}
    u^{\varepsilon}_{r, t}, v_s \rangle - \langle c_{r, t} u^{\varepsilon}_t,
    v_{s, r} \rangle\\
    &  & + \langle c_{r, t} u_{s, r}, v^{\varepsilon}_t \rangle - \langle
    c_{s, r} u_s, v^{\varepsilon}_{r, t} \rangle\\
    &  & + \langle R^{\varepsilon} (A_{r, t}, u_t), v_{s, r} \rangle -
    \langle R^{\varepsilon} (A_{s, r}, u_{r, t}), v_s \rangle .
  \end{eqnarray*}
  Therefore, up to choosing a suitable compact $K \subset \mathbb{R}^d$, we
  have the estimates
  
  \begin{align*}
    | \Gamma^{\varepsilon}_{s, t} | & \leqslant \, (\| c_{s, t}
    u^{\varepsilon}_t \|_{C^0_K} + \| (c_{s, t} u^{\varepsilon}_t) \|_{C^0_K}
    + \| R^{\varepsilon} (A_{s, t}, u_t) \|_{C^0_K}) \| v_s
    \|_{(C^0_K)^{\ast}}\\
    & \lesssim \, | t - s |^{\alpha} (\| c \|_{\alpha, \beta} + \| A
    \|_{\alpha, 1}) \| u \|_{C^0_t C^0_c} \| v_s \|_{(C^0_K)^{\ast}}
  \end{align*}
  
  as well as
  \begin{eqnarray*}
    | \delta \Gamma^{\varepsilon}_{s, r, t} | & \leqslant & \| c_{s, r}
    u^{\varepsilon}_{r, t} \|_{C^0_K}  \| v_s \|_{(C^0_K)^{\ast}} + \| c_{r,
    t} u^{\varepsilon}_t \|_{C^{\beta}_K}  \| v_{s, r}
    \|_{(C^{\beta}_K)^{\ast}}\\
    &  & + \| c_{r, t} u_{s, r} \|_{C^0_K}  \| v^{\varepsilon}_t
    \|_{(C^0_K)^{\ast}} + \| c_{s, r} u_s \|_{C^{\beta}_K}  \|
    v^{\varepsilon}_{r, t} \|_{(C^{\beta}_K)^{\ast}}\\
    &  & + \| R^{\varepsilon} \|  \| A_{r, t} \|_{1 + \beta}  \| u_t
    \|_{C^1_K} \| v_{s, r} \|_{(C^{\beta}_K)^{\ast}} + \| R^{\varepsilon} \| 
    \| A_{s, r} \|_{1 + \beta}  \| u_{r, t} \|_{C^0_K} \| v_s
    \|_{(C^0_K)^{\ast}}\\
    & \lesssim & | t - s |^{\alpha (1 + \beta)} (\| c \|_{\alpha, \beta} + \|
    R^{\varepsilon} \| \| A \|_{\alpha, 1 + \beta}) \times\\
    &  & \times (\| u \|_{C^0_t C^1_K} \| v \|_{L^{\infty}_t (C^0_K)^{\ast}}
    + \| u \|_{C^{\alpha}_t C^0_K} \| v \|_{C^{\alpha \beta}_t
    (C^{\beta}_K)^{\ast}}) .
  \end{eqnarray*}
  Overall we deduce that $\| \Gamma^{\varepsilon} \|_{\alpha}$ and $\| \delta
  \Gamma^{\varepsilon} \|_{\alpha (1 + \beta)}$ are bounded uniformly in
  $\varepsilon > 0$; moreover by properties of convolutions and
  Lemma~\ref{sec6 temp commutator lemma}, it holds $\Gamma^{\varepsilon}_{s,
  t} \rightarrow 0$ as $\varepsilon \rightarrow 0$ for any $(s, t) \in
  \Delta_2$ fixed. By Lemma~\ref{sec2 sewing lemma} it holds
  \[ | f^{\varepsilon}_{s, t} - \Gamma^{\varepsilon}_{s, t} | \lesssim | t - s
     |^{\alpha (1 + \beta)} \]
  uniformly in $\varepsilon > 0$ and so passing to the limit as $\varepsilon
  \rightarrow 0$ we deduce that
  \[ | \langle u_t, v_t \rangle - \langle u_s, v_s \rangle | \lesssim | t - s
     |^{\alpha (1 + \beta)} \quad \forall \, (s, t) \in \Delta_2 \]
  which implies the conclusion.
\end{proof}

We are now ready to establish uniqueness of solutions to the YTE and YCE under
suitable regularity conditions on $(A, c)$.

\begin{theorem}
  \label{sec6 temp main thm}Let $A \in C^{\alpha}_t C^{1 + \beta}_x$, $c \in
  C^{\alpha}_t C^{1 + \beta}_x$ with $\alpha (1 + \beta) > 1$. Then for any
  $u_0 \in C^1_{\tmop{loc}}$ there exists a unique solution to the
  YTE~{\eqref{sec6 temp YTE dual}} with initial condition $u_0$ in the class
  $C^{\alpha}_t C^0_{\tmop{loc}} \cap C^0_t C^1_{\tmop{loc}}$, which is given
  by formula~{\eqref{sec6 temp candidate solution2}}; similarly, for any
  $\mu_0 \in M_{\tmop{loc}}$ there exists a unique solution to the
  YCE~{\eqref{sec6 temp YCE dual}} with initial condition $\mu_0$ in the class
  $L^{\infty}_t (C^0_c)^{\ast} \cap C^{\alpha \beta}_t (C^{\beta}_c)^{\ast}$,
  which is given by formula~{\eqref{sec6 temp candidate solution}}.
\end{theorem}

\begin{proof}
  Existence follows from Proposition~\ref{sec6 temp prop existence}, so we
  only need to establish uniqueness. By linearity of YTE, it suffices to show
  that the only solution $u$ to~{\eqref{sec6 temp YTE dual}} in the class
  $C^{\alpha}_t C^0_{\tmop{loc}} \cap C^0_t C^1_{\tmop{loc}}$ with $u_0 \equiv
  0$ is given by $u \equiv 0$. Let $u$ be such a solution and fix $\tau \in
  [0, T]$; since $(\tmop{div} A - c) \in C^{\alpha}_t C^{\beta}_x$, by
  Proposition~\ref{sec6 temp prop existence} and Remark~\ref{sec6 temp remark
  backward solution}, for any compactly supported $\mu \in M$ there exists a
  solution $v \in L^{\infty}_t M_K \cap C^{\alpha \beta}_t
  (C^{\beta}_c)^{\ast}$ to~{\eqref{sec6 temp YCE dual}} with terminal
  condition $v_{\tau} = \mu$, up to taking a suitable compact set $K$. By
  Proposition~\ref{sec6 temp prop duality} it follows that
  \[ \langle u_{\tau}, \mu \rangle = \langle u_{\tau}, v_{\tau} \rangle =
     \langle u_0, v_0 \rangle = 0 ; \]
  as the reasoning holds for any compactly supported $\mu \in M$, $u_{\tau}
  \equiv 0$ and thus $u \equiv 0$.
  
  Uniqueness of solutions to YCE~{\eqref{sec6 temp YCE dual}} in the class
  $L^{\infty}_t (C^0_c)^{\ast} \cap C^{\alpha \beta}_t (C^{\beta}_c)^{\ast}$
  follows similatly.
\end{proof}

\section{Parabolic nonlinear Young PDEs}\label{sec7}

We present in this section a generalization to the nonlinear Young setting of
some of the results contained in~{\cite{gubinellilejay}}. Specifically, we are
interested in studying a parabolic nonlinear evolutionary problem of the form
\begin{equation}
  \mathd x_t = - A x_t \mathd t + B (\mathd t, x_t) \label{sec7 preliminary
  parabolic yde}
\end{equation}
where $- A$ is the generator of an analytical semigroup.

\

In order not to create confusion, in this section the nonlinear Young term
will be always denoted by $B$. As we will use a one-parameter family of spaces
$\{ V_{\alpha} \}_{\alpha \in \mathbb{R}}$, the regularity of $B$ will be
denoted by $B \in C^{\gamma}_t C^{\beta}_{W, U}$, with $W$ and $U$ being taken
from that family; whenever it doesn't create confusion, we will still denote
the associated norm by $\| B \|_{\gamma, \beta}$.

\

Let us first recall the functional setting from~{\cite{gubinellilejay}},
Section~2.1. It is based on the theory of analytical semigroups and
infinitesimal generators, see~{\cite{pazy}} for a general reference, but the
reader not acquainted with the topic may consider for simplicity $A = I -
\Delta$, $V = L^2 (\mathbb{R}^d)$ and $V_{\alpha} = H^{2 \alpha}
(\mathbb{R}^d)$ fractional Sobolev spaces.

\

Let $(V, \| \cdot \|_V)$ be a separable Banach space, $(A, \tmop{Dom} (A))$
be an unbounded linear operator on $V$, $\tmop{rg} (A)$ be its range; suppose
its resolvent set is contained in $\Sigma = \{ z \in \mathbb{C}: | \arg (z) |
> \pi / 2 - \delta \} \cup U$ for some $\delta > 0$ and some neighbourhood $U$
of $0$ and that there exist positive constants $C, \eta$ such that its
resolvent $R_{\alpha}$ satisfies
\[ \| R_{\alpha} \|_{\mathcal{L} (V ; V)} \leqslant C (\eta + | \alpha |)^{-
   1} \quad \forall \, \alpha \in \Sigma . \]
Under these assumptions, $- A$ is the infinitesimal generator of an analytical
semigroup $(S (t))_{t \geqslant 0}$ and there exist positive constants $M,
\lambda$ such that
\[ \| S (t) \|_{\mathcal{L} (V ; V)} \leqslant M e^{- \lambda t} \quad \forall
   \, t \geqslant 0. \]
Moreover, $- A$ is one-to-one from $\tmop{Dom} (A)$ to $V$ and the fractional
powers $(A^{\alpha}, \tmop{Dom} (A^{\alpha}))$ of $A$ can be defined for any
$\alpha \in \mathbb{R}$; if $\alpha < 0$, then $\tmop{Dom} (A^{\alpha}) = V$
and $A^{\alpha}$ is a bounded operator, while for $\alpha \geqslant 0$
$(A^{\alpha}, \tmop{Dom} (A^{\alpha}))$ is a closed operator with $\tmop{Dom}
(A^{\alpha}) = \tmop{rg} (A^{- \alpha})$ and $A^{\alpha} = (A^{- \alpha})^{-
1}$.

For $\alpha \geqslant 0$, let $V_{\alpha}$ be the space $\tmop{Dom}
(A^{\alpha})$ with norm $\| x \|_{V_{\alpha}} = \| A^{\alpha} x \|_V$; for
$\alpha = 0$ it holds $A^0 = \tmop{Id}$ and $V_0 = V$. For $\alpha < 0$, let
$V_{\alpha}$ be the completion of $V$ w.r.t. the norm $\| x \|_{V_{\alpha}} =
\| A^{\alpha} x \|_V$, which is thus a bigger space than $V$. The
one-parameter family of spaces $\{ V_{\alpha} \}_{\alpha \in \mathbb{R}}$ is
such that $V_{\delta}$ embeds continuously in $V_{\alpha}$ whenever $\delta
\geqslant \alpha$ and $A^{\alpha} A^{\delta} = A^{\alpha + \delta}$ on the
common domain of definition; moreover $A^{- \delta}$ maps $V_{\alpha}$ onto
$V_{\alpha + \delta}$ for all $\alpha \in \mathbb{R}$ and $\delta \geqslant
0$.

The operator $S (t)$ can be extended to $V_{\alpha}$ for all $\alpha < 0$ and
$t > 0$ and maps $V_{\alpha}$ to $V_{\delta}$ for all $\alpha \in \mathbb{R}$,
$\delta \geqslant 0$, $t > 0$; finally, it satisfies the following properties:
\begin{equation}
  \| A^{\alpha} S (t) \|_{\mathcal{L} (V ; V)} \leqslant M_{\alpha} t^{-
  \alpha} e^{- \lambda t} \text{ for all $\alpha \geqslant 0$, $t > 0$} ;
  \label{sec7 semigroup 1}
\end{equation}
\begin{equation}
  \| S (t) x - x \|_V \leqslant C_{\alpha} t^{\alpha} \| A^{\alpha} x \|_V
  \text{ for all $x \in V_{\alpha}$, $\alpha \in (0, 1]$} . \label{sec7
  semigroup 2}
\end{equation}
\begin{remark}
  \label{sec7 remark initial data}It follows from the statements above and the
  semigroup property of $S (t)$ that for any $\alpha \in \mathbb{R}$, $\delta
  > 0$, $x \in V_{\alpha}$ and any $s \leqslant t$ it holds
  \[ \| S (t) x - S (s) x \|_{V_{\alpha}} = \| S (s) [S (t - s) x - x]
     \|_{V_{\alpha}} \lesssim_{\alpha, \delta} | t - s |^{\delta} \| x
     \|_{V_{\alpha + \delta}} \]
  which implies that $\| S (t) - S (s) \|_{\mathcal{L} (V_{\alpha + \delta} ;
  V_{\alpha})} \lesssim | t - s |^{\delta}$, equivalently $S (\cdot) \in
  C^{\delta}_t \mathcal{L} (V_{\alpha + \delta} ; V_{\alpha})$. It also
  follows that for any given $x_0 \in V_{\alpha + \delta}$, the map $t \mapsto
  S (t) x_0$ belongs to $C^{\delta}_t V_{\alpha}$ with
  \begin{equation}
    \llbracket S (\cdot) x_0 \rrbracket_{\delta, V_{\alpha}} \lesssim_{\alpha,
    \delta} \| x_0 \|_{V_{\alpha + \delta}} .
  \end{equation}
\end{remark}

The following result shows that the mild solution formula for the linear
equation
\[ \mathd x_t = - A x_t \mathd t + \mathd y_t, \]
which is formally given by
\[ x_t = S (t) x_0 + \int_0^t S (t - s) \mathd y_s, \]
can be extended by continuity to suitable non differentiable functions $y \in
C ([0, T] ; V)$.

\begin{theorem}
  \label{sec7 thm linear equation}Let $\alpha \in \mathbb{R}$ and consider the
  map $\Xi$ defined for any $y \in C^1_t V_{- \alpha}$ by
  \[ \Xi (y)_t = \int_0^t S (t - s) \dot{y}_s \mathd s. \]
  Then for any $\gamma > \alpha$, $\Xi$ extends uniquely to a map $\Xi \in
  \mathcal{L} (C^{\gamma}_t V_{- \alpha} ; C^{\kappa}_t V_{\delta})$ for all
  $\delta \in (0, \gamma - \alpha)$ and all $\kappa \in (0, (\gamma - \alpha -
  \delta) \wedge 1)$. Moreover there exists a constant $C = C (\alpha, \kappa,
  \delta, \gamma)$ such that
  \begin{equation}
    \llbracket \Xi (y) \rrbracket_{\kappa, V_{\delta}} \leqslant C \llbracket
    y \rrbracket_{\gamma, V_{- \alpha}}, \quad \sup_{t \in [0, T]} \| \Xi
    (y)_t \|_{V_{\delta}} \leqslant C T^{\gamma - \delta - \alpha} \llbracket
    y \rrbracket_{\gamma, V_{- \alpha}} . \label{sec7 thm linear equation
    estimates}
  \end{equation}
\end{theorem}

We omit the proof, for which we refer to Theorem~1
from~{\cite{gubinellilejay}}. Let us only provide an heuristic derivation of
the relation between the parameters $\alpha, \kappa, \delta, \gamma$ based on
a regularity counting argument. It follows from Remark~\ref{sec7 remark
initial data} that $\| S (t - s) \|_{\mathcal{L} (V_{- \alpha} ; V_{\delta})}
\lesssim | t - s |^{- \delta - \alpha}$; if it's possible to define the map
$\Xi (y)$ taking values in $V_{\delta}$, then we would expect its time
regularity to be analogue to that of
\begin{equation}
  g_t \assign \int_0^t | t - s |^{- \delta - \alpha} \mathd f_s, \label{sec7
  heuristic eq}
\end{equation}
where now $f, g$ are real valued functions, $f \in C^{\gamma}_t$; indeed,
considering a fixed $y_0 \in V_{- \alpha}$, the result should also apply to
$y_t \assign f_t y_0$. The integral in~{\eqref{sec7 heuristic eq}} is a type
of fractional integral of order $1 - \delta - \alpha$ and by hypothesis
$\mathd f \in C^{\gamma - 1}_t$, therefore $g$ should have regularity $\gamma
- \delta - \alpha$, which is exactly the threshold parameter for $\kappa$
(this is because H\"{o}lder spaces do not behave well under fractional
integration and one must always give up an $\varepsilon$ of regularity by
embedding them in nicer spaces).

\begin{definition}
  \label{sec7 defn solution}Given $A$ as above and $B \in C^{\gamma}_t
  C^{\beta}_{V_{\delta}, V_{\rho}}$, $\rho \leqslant \delta$, we say that $x
  \in C^{\kappa}_t V_{\delta}$ is a mild solution to equation~{\eqref{sec7
  preliminary parabolic yde}} with initial data $x_0 \in V_{\delta}$ if
  $\gamma + \beta \kappa > 1$, so that $\int_0^{\cdot} B (\mathd s, x_s)$ is
  well defined as a nonlinear Young integral, and if $x$ satisfies
  \begin{equation}
    x_t = S (t) x_0 + \int_0^t S (t - s) B (\mathd s, x_s) = S (t) x_0 + \Xi
    \left( \int_0^{\cdot} B (\mathd s, x_s) \right)_t \quad \forall \, t \in
    [0, T]
  \end{equation}
  where $\Xi$ is the map defined by Theorem~\ref{sec7 thm linear equation} and
  the equality holds in $V_{\alpha}$ for suitable $\alpha$.
\end{definition}

We are now ready to prove the main result of this section.

\begin{theorem}
  \label{sec7 main thm}Assume $A$ as above, $B \in C^{\gamma}_t C^{1 +
  \beta}_{V_{\delta}, V_{\rho}}$ with $\rho > \delta - 1$ and suppose there
  exists $\kappa \in (0, 1)$ such that
  \begin{equation}
    \left\{\begin{array}{l}
      \gamma + \beta \kappa > 1\\
      \kappa < \gamma + \rho - \delta
    \end{array}\right. . \label{sec7 condition k}
  \end{equation}
  Then for any $x_0 \in V_{\delta + \kappa}$ there exists a unique solution
  with initial data $x_0$ to~{\eqref{sec7 preliminary parabolic yde}}, in the
  sense of Definition~\ref{sec7 defn solution}, in the class $C^{\kappa}_t
  V_{\delta} \cap C^0_t V_{\delta + \kappa}$.
  
  Moreover, the solution depends in a Lipschitz way on $(x_0, B)$, in the
  following sense: for any $R > 0$ exists a constant $C = C (\beta, \gamma,
  \delta, \rho, \kappa, T, R)$ such that for any $(x_0^i, B^i)$, $i = 1, 2$,
  satisfying $\| x_0^i \|_{V_{\delta + \kappa}} \vee \| B^i \|_{\gamma, 1 +
  \beta} \leqslant R$, denoting by $x^i$ the associated solutions, it holds
  \[ \llbracket x^1 - x^2 \rrbracket_{\kappa, V_{\rho}} \leqslant C (\| x_0^1
     - x_0^2 \|_{V_{\delta + \kappa}} + \| B^1 - B^2 \|_{\gamma, 1 + \beta}) .
  \]
\end{theorem}

\begin{remark}
  If $B \in C^{\gamma}_t C^2_{V_{\delta}, V_{\rho}}$, then it is possible to
  find $\kappa$ satisfying~{\eqref{sec7 condition k}} if and only if
  \[ 2 \gamma + \rho - \delta > 1. \]
\end{remark}

\begin{proof}
  The basic idea is to apply a Banach fixed point argument to the map
  \begin{equation}
    x \mapsto \mathcal{I} (x)_t : = S (t) x_0 + \Xi \left( \int_0^{\cdot} B
    (\mathd s, x_s) \right)_t \label{sec7 main thm proof eq1}
  \end{equation}
  defined on a suitable domain.
  
  By Remark~\ref{sec7 remark initial data}, if $x_0 \in V_{\delta + \kappa}$,
  then $S (\cdot) x_0 \in C^{\kappa}_t V_{\delta}$; moreover $B \in
  C^{\gamma}_t C^1_{V_{\delta}, V_{\rho}}$, so under the condition $\gamma +
  \kappa > 1$ the nonlinear Young integral in~{\eqref{sec7 main thm proof
  eq1}} is well defined for $x \in C^{\kappa}_t V_{\delta}$, $y_t = \int_0^t B
  (\mathd s, x_s) \in C^{\gamma}_t V_{\rho}$ and then $\Xi (y) \in
  C^{\kappa}_t V_{\delta}$ under the condition $\kappa < \gamma + \rho -
  \delta$. So under our assumptions $\mathcal{I}$ maps $C^{\kappa}_t
  V_{\delta}$ into itself; our first aim is to find a closed bounded subset
  which is invariant under $I$.
  
  For suitable $\tau, M$ to be fixed later, consider the set
  \[ E : = \{ x \in C^{\kappa} ([0, \tau] ; V_{\delta}) : x (0) = x_0,
     \llbracket x \rrbracket_{\kappa, V_{\delta}} \leqslant M, \sup_{t \in [0,
     \tau]} \| x_t \|_{V_{\delta + \kappa}} \leqslant M \} ; \]
  $E$ is a complete metric space endowed with the distance $d_E (x_1, x_2) =
  \llbracket x_1 - x_2 \rrbracket_{\kappa, V_{\delta}}$. It holds
  
  \begin{align*}
    \llbracket \mathcal{I} (x) \rrbracket_{\kappa, V_{\delta}} & \leqslant
    \llbracket S (\cdot) x_0 \rrbracket_{\kappa, V_{\delta}} + \left\llbracket
    \Xi \left( \int_0^{\cdot} B (\mathd s, x_s) \right)
    \right\rrbracket_{\kappa, V_{\delta}} \lesssim \| x_0 \|_{V_{\delta +
    \rho}} + \left\llbracket \int_0^{\cdot} B (\mathd s, x_s)
    \right\rrbracket_{\gamma, V_{\rho}} ;
  \end{align*}
  
  for the nonlinear Young integral we have the estimate
  
  \begin{align*}
    \left\| \int_s^t B (\mathd r, x_r) \right\|_{V_{\rho}} & \lesssim \| B_{s,
    t} (x_s) \|_{V_{\rho}} + | t - s |^{\gamma + \kappa} \llbracket B
    \rrbracket_{\gamma, 1} \llbracket x \rrbracket_{\kappa, V_{\delta}}\\
    & \lesssim \| B_{s, t} (x_s) - B_{s, t} (x_0) \|_{V_{\rho}} + | t - s
    |^{\gamma} \| B \|_{\gamma, 0} + | t - s |^{\gamma} \tau^{\kappa}
    \llbracket B \rrbracket_{\gamma, 1} \llbracket x \rrbracket_{\kappa}\\
    & \lesssim | t - s |^{\gamma} \| B \|_{\gamma, 1} (1 + \tau^{\kappa}
    \llbracket x \rrbracket_{\kappa, V_{\delta}})
  \end{align*}
  
  and so
  \[ \left\llbracket \int_0^{\cdot} B (\mathd r, x_r)
     \right\rrbracket_{\gamma, V_{\rho}} \lesssim \| B \|_{\gamma, 1} (1 +
     \tau^{\kappa} \llbracket x \rrbracket_{\kappa, V_{\delta}}) . \]
  Overall, we can find a constant $\kappa_1$ such that
  \[ \llbracket \mathcal{I} (x) \rrbracket_{\kappa, V_{\delta}} \leqslant
     \kappa_1 \| x_0 \|_{V_{\delta + \kappa}} + \kappa_1 \| B \|_{\gamma, 1}
     (1 + \tau^{\kappa}  \llbracket x \rrbracket_{\kappa, V_{\delta}}) . \]
  Similar computations, together with estimate~{\eqref{sec7 thm linear
  equation estimates}}, show the existence of $\kappa_2$ such that
  \[ \sup_{t \in [0, \tau]} \| I (x)_t \|_{V_{\delta + \kappa}} \leqslant
     \kappa_2 \| x_0 \|_{V_{\delta + \kappa}} + \kappa_2 \| B \|_{\gamma, 1}
     \tau^{\gamma - \delta + \rho} (1 + \tau^{\kappa} \llbracket x
     \rrbracket_{\kappa, V_{\delta}}) . \]
  Therefore takng $\tau \leqslant 1$, $\kappa_3 = \kappa_1 \vee \kappa_2$, in
  order for $\mathcal{I}$ to map $E$ into itself it suffices
  \[ \kappa_3 \| x_0 \|_{V_{\delta + \kappa}} + \kappa_3 \| B \|_{\gamma, 1}
     (1 + \tau^{\kappa} M) \leqslant M, \]
  which is always possible, for instance by requiring
  \[ 2 \kappa_3 \| B \|_{\gamma, 1} \tau^{\kappa} \leqslant 1, \quad 2
     \kappa_3 \| x_0 \|_{V_{\delta + \kappa}} + 2 \kappa_3 \| B \|_{\gamma, 1}
     \leqslant M. \]
  Observe that $\tau$ can be chosen independently of $\| x_0 \|_{V_{\delta +
  \kappa}}$; moreover for the same choice of $\tau$, analogous computations
  show that any solution $x$ to~{\eqref{sec7 preliminary parabolic yde}}
  defined on $[0, \tilde{\tau}]$ with $\tilde{\tau} \leqslant \tau$ satisfies
  the a priori estimate
  \begin{equation}
    \llbracket x \rrbracket_{\kappa, V_{\delta} ; 0, \tilde{\tau}} + \sup_{t
    \in [0, \tilde{\tau}]} \| x_t \|_{V_{\delta + \kappa}} \leqslant \kappa_4 
    (\| x_0 \|_{V_{\delta + \kappa}} + \| B \|_{\gamma, 1}) \label{sec7 main
    thm proof eq2}
  \end{equation}
  for another constant $\kappa_4$, independent of $x_0$.
  
  We now want to find $\tilde{\tau} \in [0, \tau]$ such that $I$ is a
  contraction on $\tilde{E}$, $\tilde{E}$ being defined as $E$ in terms of
  $\tilde{\tau}, M$. Given $x^1, x^2 \in \tilde{E}$, it holds
  
  \begin{align*}
    d_E (\mathcal{I} (x^1), \mathcal{I} (x^2)) & = \left\llbracket \Xi \left(
    \int_0^{\cdot} B (\mathd s, x^1_s) - \int_0^{\cdot} B (\mathd s, x^2_s)
    \right) \right\rrbracket_{\kappa, V_{\delta}}\\
    & \lesssim \left\llbracket \left( \int_0^{\cdot} B (\mathd s, x^1_s) -
    \int_0^{\cdot} B (\mathd s, x^2_s) \right) \right\rrbracket_{\kappa,
    V_{\rho}}
  \end{align*}
  
  and under the assumptions we can apply Corollary~\ref{sec2 corollary
  frechet}, so we have
  
  \begin{align*}
    \left\| \int_s^t B (\mathd r, x_r^1) - \int_s^t B (\mathd r, x^2_r)
    \right\|_{V_{\rho}} & = \left\| \int_s^t v_{\mathd r} (x^1_r - x^2_r)
    \right\|_{V_{\rho}}\\
    & \lesssim | t - s |^{\gamma} \llbracket v \rrbracket_{\gamma,
    \mathcal{L}} \| x^1_s - x^2_s \|_{V_{\rho}} + | t - s |^{\gamma + \kappa}
    \llbracket v \rrbracket_{\gamma, \mathcal{L}} \llbracket x^1 - x^2
    \rrbracket_{\kappa, V_{\rho}}\\
    & \lesssim | t - s |^{\gamma} \| B \|_{\gamma, 1 + \beta} (1 + M)
    \tilde{\tau}^{\kappa} \llbracket x^1 - x^2 \rrbracket_{\kappa, V_{\rho}} .
  \end{align*}
  
  This implies
  \[ \left\llbracket \int_0^{\cdot} B (\mathd r, x^1_r) - B (\mathd r, x^2_r)
     \right\rrbracket_{\gamma, V_{\rho}} \lesssim \| B \|_{\gamma, 1 + \beta}
     (1 + M) \tilde{\tau}^{\kappa} \llbracket x^1 - x^2 \rrbracket_{\kappa,
     V_{\rho}} \]
  and so overall, for a suitable constant $\kappa_5$,
  \[ d_E (\mathcal{I} (x^1), \mathcal{I} (x^2)) \leqslant \kappa_5 \| B
     \|_{\gamma, 1 + \beta} (1 + M) \tilde{\tau}^{\kappa} d_E (x^1, x^2) . \]
  Choosing $\tilde{\tau}$ small enough such that $\kappa_5 \| B \|_{\gamma, 1
  + \beta} (1 + M) \tilde{\tau}^{\kappa} < 1$, we deduce that there exists a
  unique solution to~{\eqref{sec7 preliminary parabolic yde}} defined on $[0,
  \tilde{\tau}]$. Since we have the uniform estimate~{\eqref{sec7 main thm
  proof eq2}}, we can iterate the contraction argument to construct a unique
  solution on $[0, \tau]$; but since the choice of $\tau$ does not depend on
  $x_0$ and $x_{\tau} \in V_{\delta + \kappa}$, we can iterate further to
  cover the whole interval $[0, T]$ with subintervals of size $\tau$.
  
  To check the Lipschitz dependence on $(x_0, B)$, one can reason using the
  Comparison Principle as usual, but let us give an alternative proof; we only
  check Lipschitz dependence on $B$, as the proof for $x_0$ is similar.
  
  Given $B^i$, $i = 1, 2$ as above, denote by $\mathcal{I}_{B^i}$ the map
  associated to $B^i$ defined as in~{\eqref{sec7 main thm proof eq1}}; we can
  choose $\tilde{\tau}$ and $M$ such that they are both strict contractions of
  constant $\kappa_6 < 1$ on $E$ defined as before. Observe that for any $z
  \in E$ it holds
  
  \begin{align*}
    d_E (\mathcal{I}_{B^1} (z), \mathcal{I}_{B^2} (z)) & = \left\llbracket \Xi
    \left( \int_0^{\cdot} B^1 (\mathd s, z_s) - \int_0^{\cdot} B^2 (\mathd s,
    z_s) \right) \right\rrbracket_{\kappa, V_{\delta}}\\
    & \lesssim \left\llbracket \int_0^{\cdot} B^1 (\mathd s, z_s) -
    \int_0^{\cdot} B^2 (\mathd s, z_s) \right\rrbracket_{\gamma, V_{\rho}}\\
    & \lesssim (1 + M) \| B^1 - B^2 \|_{\gamma, \beta} .
  \end{align*}
  
  Denote by $x^i$ the unique solutions on $E$ associated to $B^i$, then by the
  above computation we get
  
  \begin{align*}
    \llbracket x^1 - x^2 \rrbracket_{\kappa, V_{\delta}} & = d_E
    (\mathcal{I}_{B^1} (x^1), \mathcal{I}_{B^2} (x^2))\\
    & \leqslant d_E (\mathcal{I}_{B^1} (x^1), \mathcal{I}_{B^1} (x^2)) + d_E
    (\mathcal{I}_{B^1} (x^2), \mathcal{I}_{B^2} (x^2))\\
    & \leqslant \kappa_6  \llbracket x^1 - x^2 \rrbracket_{\kappa,
    V_{\delta}} + \kappa_7 (1 + M) \| B^1 - B^2 \|_{\gamma, \beta}
  \end{align*}
  
  which implies that
  \[ \llbracket x^1 - x^2 \rrbracket_{\kappa, V_{\delta}} \leqslant
     \frac{\kappa_7}{1 - \kappa_6} (1 + M) \| B^1 - B^2 \|_{\gamma, \beta} \]
  which shows Lipschitz dependence on $B^i$ on the interval $[0,
  \tilde{\tau}]$. As before, a combination of a priori estimates and iterative
  arguments allows to extend the estimate to a global one.
\end{proof}

By the usual localization and blow-up alternative arguments, we obtain the
following result.

\begin{corollary}
  Assume $A$ as above, $B \in C^{\gamma}_t C^{1 + \beta}_{V_{\delta},
  V_{\rho}, \tmop{loc}}$ with $\rho > \delta - 1$ and suppose there exists
  $\kappa \in (0, 1)$ satisfying~{\eqref{sec7 condition k}}. Then for any $x_0
  \in V_{\delta + \kappa}$ there exists a unique maximal solution $x$ starting
  from $x_0$, defined on an interval $[0, T^{\ast}) \subset [0, T]$, such that
  either $T^{\ast} = T$ or
  \[ \lim_{t \uparrow T^{\ast}} \| x_t \|_{V_{\delta + \kappa}} = + \infty .
  \]
\end{corollary}

\begin{remark}
  For simplicity we have only treated here uniqueness results, but if the
  embedding $V_{\delta} \hookrightarrow V_{\alpha}$ for $\delta > \alpha$ is
  compact, as is often the case, one can use compactness arguments to deduce
  existence of solutions under weaker regularity conditions on $B$, in analogy
  with Theorem~\ref{sec3.1 thm existence}. Once can also consider equations of
  the form
  \[ \mathd x_t = - A x_t \mathd t + F (x_t) \mathd t + B (\mathd t, x_t), \]
  in which case uniqueness can be achieved under the same conditions on $B$ as
  above and a Lipschitz condition on $F$, see also Remark~1
  from~{\cite{gubinellilejay}}.
\end{remark}

\appendix\section{Appendix}
\subsection{Some useful lemmas}\label{appendixA}

We collect in this appendix some basic tools; we start with a Fubini-type
result for the sewing map. In the following, the separable Banach space $V$ is
endowed with its Borel $\sigma$-algebra, the space $C^{\alpha, \beta}_2 V$
with the $\sigma$-algebra induced by the norm $\| \cdot \|_{\alpha, \beta}$;
recall that by the sewing lemma, $\mathcal{J} : C^{\alpha, \beta}_2 V
\rightarrow C^{\alpha}_t V$ is linear and continuous.

\begin{lemma}[Fubini for sewing map]
  \label{appendix lemma fubini}Let $V$ as above, $(S, \mathcal{A}, \mu)$ a
  measure space and consider a measurable map $\Gamma : S \rightarrow
  C^{\alpha, \beta}_2 V$, $\theta \mapsto \Gamma (\theta)$, such that
  \[ \int_S \| \Gamma (\theta) \|_{\alpha, \beta} \mu (\mathd \theta) < \infty
     . \]
  Then the map $\mathcal{J} \circ \Gamma : S \rightarrow C^{\alpha}_t V$ is
  measurable and it holds
  \begin{equation}
    \mathcal{J} \left( \int_S \Gamma (\theta) \mu (\mathd \theta) \right) =
    \int_S \mathcal{J} (\Gamma (\theta)) \mu (\mathd \theta) . \label{appendix
    fubini identity}
  \end{equation}
\end{lemma}

\begin{proof}
  Since $\mathcal{J}$ is continuous, in particular it is measurable, and so is
  $\mathcal{J} \circ \Gamma$ being a composition of measurable functions; it
  also follows that for any fixed $(s, t) \in \Delta_2$, the map $\theta
  \mapsto \mathcal{J} (\Gamma (\theta))_{s, t}$ is measurable from $S$ to $V$.
  We can therefore define both integrals appearing in~{\eqref{appendix fubini
  identity}} as Bochner integrals, by considering them for any fixed pair $(s,
  t) \in \Delta_2$. For instance it holds
  \[ \left\| \int_S \Gamma (\theta)_{s, t} \mu (\mathd \theta) \right\|_V
     \leqslant \int_S \| \Gamma (\theta)_{s, t} \|_V \mu (\mathd \theta)
     \leqslant | t - s |^{\alpha}  \int_S \| \Gamma (\theta) \|_{\alpha,
     \beta} \mu (\mathd \theta) < \infty \]
  which also shows that the map $(s, t) \mapsto \int_S \Gamma (\theta)_{s, t}
  \mu (\mathd \theta)$ belongs to $C^{\alpha, \beta}_2 V$ with
  \[ \left\| \int_S \Gamma (\theta) \mu (\mathd \theta) \right\|_{\alpha,
     \beta} \leqslant \int_S \| \Gamma (\theta) \|_{\alpha, \beta} \mu (\mathd
     \theta) . \]
  In order to show that~{\eqref{appendix fubini identity}} holds, by the
  sewing lemma it suffices to prove that
  \[ \left\| \left( \int_S \Gamma (\theta) \mu (\mathd \theta) \right)_{s, t}
     - \int_S \mathcal{J} (\Gamma (\theta))_{s, t} \mu (\mathd \theta)
     \right\|_V \lesssim | t - s |^{\beta} \quad \forall \, (s, t) \in
     \Delta_2 ; \]
  from the properties of $\mathcal{J} (\Gamma (\theta))$, we have the estimate
  \begin{eqnarray*}
    \left\| \left( \int_S \Gamma (\theta) \mu (\mathd \theta) \right)_{s, t} -
    \int_S \mathcal{J} (\Gamma (\theta))_{s, t} \mu (\mathd \theta) \right\|_V
    & \leqslant & \int_S \| \Gamma (\theta)_{s, t} - \mathcal{J} (\Gamma
    (\theta))_{s, t} \|_V \mu (\mathd \theta)\\
    & \lesssim & | t - s |^{\beta}  \int_S \| \Gamma (\theta) \|_{\alpha,
    \beta} \mu (\mathd \theta)
  \end{eqnarray*}
  and the conclusion follows.
\end{proof}

\begin{lemma}
  \label{appendix lemma interpolation}Let $\{ \Gamma^n \}_n \subset C^{\alpha,
  \beta}_2 V$ be a sequence such that $\sup_n \| \delta \Gamma^n \|_{\beta}
  \leqslant R$ and $\lim_n \| \Gamma^n \|_{\alpha} \rightarrow 0$. Then
  $\mathcal{J} \Gamma^n \rightarrow 0$ in $C^{\alpha}_t V$ and for all $n$ big
  enough it holds
  \begin{equation}
    \llbracket \mathcal{J} \Gamma^n \rrbracket_{\alpha} \lesssim_{T, \beta} (1
    + R)  \| \Gamma^n \|_{\alpha}^{(\beta - 1) / (\beta - \alpha)} .
    \tmcolor{blue}{} \label{appendix interpolation estimate}
  \end{equation}
\end{lemma}

\begin{proof}
  Fix an interval $[s, t] \subset [0, T]$. By hypothesis, it holds
  \[ \| (\mathcal{J} \Gamma^n)_{s, t} \|_V \leqslant \| \Gamma^n \|_{\alpha} 
     | t - s |^{\alpha} + \kappa_{\beta} \| \delta \Gamma^n \|_{\beta}  | t -
     s |^{\beta} ; \]
  splitting the interval in $m$ subintervals of size $| t - s | / m$, applying
  the estimate to each of them and summing over we also have
  \begin{equation}
    \| (\mathcal{J} \Gamma^n)_{s, t} \|_V \leqslant \| \Gamma^n \|_{\alpha}
    m^{1 - \alpha}  | t - s |^{\alpha} + \kappa_{\beta} \| \delta \Gamma^n
    \|_{\beta} m^{1 - \beta}  | t - s |^{\beta} . \label{appendix lemma eq1}
  \end{equation}
  By hypothesis it holds
  \[ \lim_{n \rightarrow \infty} \frac{\| \delta \Gamma^n \|_{\beta}}{\|
     \Gamma^n \|_{\alpha}} = + \infty, \]
  therefore for all $n$ big enough we can choose $m \in \mathbb{N}$ such that
  $m^{1 - \alpha} \sim (\| \delta \Gamma^n \|_{\beta} / \| \Gamma^n
  \|_{\alpha})^{\theta}$ for some $\theta \in (0, 1)$. Then in
  estimate~{\eqref{appendix lemma eq1}}, diving by $| t - s |^{\alpha}$ and
  taking the supremum, we obtain
  
  \begin{align*}
    \llbracket \mathcal{J} \Gamma^n \rrbracket_{\alpha} & \lesssim_{T, \beta}
    \| \Gamma^n \|_{\alpha}^{1 - \theta}  \| \delta \Gamma^n
    \|_{\beta}^{\theta} + \| \Gamma^n \|_{\alpha}^{\theta (\beta - 1) / (1 -
    \alpha)}  \| \delta \Gamma^n \|_{\beta}^{1 - \theta (\beta - 1) / (1 -
    \alpha)}\\
    & \lesssim_{T, \beta} (1 + R)  [\| \Gamma^n \|^{1 - \theta}_{\alpha} + \|
    \Gamma^n \|_{\alpha}^{\theta (\beta - 1) / (1 - \alpha)}] .
  \end{align*}
  
  The conclusion follows choosing $\theta = (1 - \alpha) / (\beta - \alpha)$.
\end{proof}

The following basic result was used in Section~\ref{sec5.2}.

\begin{lemma}
  \label{appendix conditional lemma translations}Let $f \in C^{n + \beta}_V$,
  $z_1, z_2 \in V$. Then for any $\eta \in (0, 1)$ with $\eta < n + \beta$ it
  holds
  \[ \left\| f \left( \, \cdot \, + z_1 \right) - f \left( \, \cdot \, + z_2
     \right) \right\|_{n + \beta - \eta} \lesssim \| z_1 - z_2 \|_V^{\eta} \,
     \| f \|_{n + \beta} . \]
\end{lemma}

\begin{proof}
  It is enough to prove the claim in the cases $n = 0$ and $n = 1$, the others
  being similar.
  
  Assume first $n = 0$, then we have the elementary estimates
  
  \begin{align*}
    \| f (x + z_1) - f (y + z_1) - f (x + z_2) + f (y + z_2) \|_V & \leqslant
    2 \| f \|_{\beta}  \| x - y \|_V^{\beta},\\
    \| f (x + z_1) - f (y + z_1) - f (x + z_2) + f (y + z_2) \|_V & \leqslant
    2 \| f \|_{\beta} \| z_1 - z_2 \|^{\beta}_{C^{\beta}_V}
  \end{align*}
  
  which interpolated together give the conclusion.
  
  Now consider $n = 1$ and $\eta \in (\beta, 1 + \beta)$, then
  
  \begin{align*}
    \|f (x + z_1) - f (y + z_1) & - f (x + z_2) + f (y + z_2) \|_V\\
    & = \left\| \int_0^1 [D f (z_1 + y + \theta (x - y), x - y) - D f (z_2 +
    y + \theta (x - y), x - y)] \mathd \theta \right\|_V\\
    & \lesssim \| x - y \|_V  \| z_1 - z_2 \|^{\beta}  \| f \|_{1 + \beta} ;
  \end{align*}
  
  inverting the roles of $z_1$ and $x$ (respectively $z_2$ and $y$) we also
  obtain
  \[ \| f (x + z_1) - f (y + z_1) - f (x + z_2) + f (y + z_2) \|_V \lesssim \|
     z_1 - z_2 \|_V  \| x - y \|^{\beta}  \| f \|_{1 + \beta} . \]
  Interpolating the two inequalities again yields the conclusion.
\end{proof}

\subsection{Alternative constructions of Young integrals}\label{appendixB}

We collect in this appendix several other constructions of the nonlinear Young
integral, although mostly equivalent to the one from Section~\ref{sec2}.

\

In Section~\ref{sec2} we constructed the nonlinear Young integral following
the modern approach based on an application of the sewing lemma, but this is
not how it was first introduced in~{\cite{catelliergubinelli}}. The approach
therein was instead based on combining property~\tmtextit{4.} of
Theorem~\ref{sec2 thm definition young integral} with estimate~{\eqref{sec2
nonlinear integral estimate1}}; namely, the classical integral $\int_0^{\cdot}
\partial_t A (s, x_s) \mathd s$ can be controlled by $\| A \|_{\alpha, \beta}$
and $\| x \|_{\gamma}$, and thus its definition can be extended by an
approximation procedure, as the following lemma shows.

\begin{lemma}
  \label{sec2 remark approximation A}Any $A \in C^{\alpha}_t C^{\beta}_{V, W}$
  can be approximated in $C^{\alpha -}_t C^{\beta -}_{V, W}$ by a sequence
  $A^n$ such that $\partial_t A^n$ exists and is continuous.
\end{lemma}

\begin{proof}
  Extend $A$ to $t \in (- \infty, \infty)$ by
  \[ A (t, x) = A (0, x) \mathbbm{1}_{t < 0} + A (t, x) \mathbbm{1}_{t \in [0,
     T]} + A (T, x) \mathbbm{1}_{t > T} \]
  and consider $\rho \in C^{\infty}_c (\mathbb{R})$ s.t. $\rho \geqslant 0$,
  $\rho (0) = 1$ and $\int \rho (t) \mathd t = 1$. Setting $\rho^{\varepsilon}
  (t) = \varepsilon^{- 1} \rho (t / \varepsilon)$ and
  \[ A^{\varepsilon} (t, x) = \int_{\mathbb{R}} \rho^{\varepsilon} (t - s) A
     (s, x) \mathd s, \]
  it's immediate to check that $\sup_{(t, x)} \| A - A^{\varepsilon} \|
  \rightarrow 0$ as $\varepsilon \rightarrow 0$ by the uniform continuity of
  $A$ (which is granted from the fact that $A \in C^{\alpha}_t C^{\beta}_{V,
  W}$). We also have the uniform bound $\llbracket A^{\varepsilon}
  \rrbracket_{\alpha, \beta} \leqslant \llbracket A \rrbracket_{\alpha,
  \beta}$, since
  
  \begin{align*}
    \| A^{\varepsilon}_{s, t} (x) - A^{\varepsilon}_{s, t} (y) \|_W & =
    \left\| \int_{\mathbb{R}} \rho^{\varepsilon} (u) [A (t - u, x) - A (s - u,
    x) - A (t - u, y) + A (s - u, y)] \mathd u \right\|_W\\
    & \leqslant \llbracket A \rrbracket_{\alpha, \beta}  | t - s |^{\alpha}
    \| x - y \|^{\beta}_V,
  \end{align*}
  
  as well as similar uniform bounds for $\| A_{s, t} \|_{\beta}$, etc.
  Interpolating these estimates together, convergence of $A^{\varepsilon}$ to
  $A$ in $C^{\alpha - \delta}_t C^{\beta - \delta}_{V, W}$ as $\varepsilon
  \rightarrow 0$, for any $\delta > 0$, immediately follows.
\end{proof}

Observe that in the above giving up a $\delta$ of regularity is not an issue
in terms of defining $\int_0^{\cdot} A (\mathd s, x_s)$, since we can always
find $\delta > 0$ small enough such that it still holds $\alpha - \delta +
(\beta - \delta) \gamma > 1$.

\

Another more functional way to define nonlinear Young integrals is the
following one: for any $\beta > 0$, consider the map $J : V \rightarrow
\mathcal{L} (C^{\beta}_{V, W} ; W)$ given by $x \mapsto \delta_x$; such a map
is trivially $\beta$-H\"{o}lder regular, since
\[ \| J x - J y \|_{\mathcal{L} (C^{\beta}_{V, W} ; W)} = \sup_{g \in
   C^{\beta}_{V, W}} \frac{\| \langle J x - J y, g \rangle \|_W}{\| g
   \|_{\beta}} = \sup_{g \in C^{\beta}_{V, W}} \frac{\| g (x) - g (y) \|_W}{\|
   g \|_{\beta}} \leqslant \| x - y \|_V^{\beta} . \]
where we denoted by $\langle \cdot, \cdot \rangle$ the pairing between
$\mathcal{L} (C^{\beta}_{V, W} ; W)$ and $C^{\beta}_{V, W}$. Therefore for any
$x \in C^{\gamma}_t V$, the map $t \mapsto J x_t = \delta_{x_t}$ is now an
element of $C^{\gamma \beta}_t \mathcal{L} (C^{\beta}_{V, W} ; W)$. If on the
other hand $A \in C^{\alpha}_t C^{\beta}_{V, W}$ and $\alpha + \gamma \beta >
1$, then we can define the (linear) Young integral
\[ \int_0^t \langle \delta_{x_s}, A_{\mathd s} \rangle = \lim_{| \Pi |
   \rightarrow 0} \sum_i \langle \delta_{x_{t_i}}, A_{t_i, t_{i + 1}} \rangle
   = \lim_{| \Pi | \rightarrow 0} \sum_i A_{t_i, t_{i + 1}} (x_{t_i}) \]
which immediately shows that it coincides with the definition from
Section~\ref{sec2}.

While this construction might seem unnecessarily abstract, it shows that
nonlinear Young integrals can be regarded as linear ones, after the nonlinear
transformation $x \mapsto \delta_x$ has been applied. It also allows to give
intuitive derivations of several integral relations: for instance by Young
product rule it must hold
\[ \langle \delta_{x_t}, A_t \rangle - \langle \delta_{x_0}, A_0 \rangle =
   \int_0^t \langle \delta_{x_s}, A_{\mathd s} \rangle + \int_0^t \langle
   \mathd \delta_{x_s}, A_s \rangle \]
which is the abstract analogue of the It\^{o}-like formula from
Proposition~\ref{sec2 prop ito formula}.

\

We finally mention a third construction of nonlinear Young integrals, given
in~{\cite{hu2}} by means of fractional calculus, in the spirit of the results
by Z\"{a}hle~{\cite{zahle}} for the classical Young integral. Fractional
calculus is a powerful tool in the study of detailed properties of solutions
to classical YDEs, see~{\cite{hunualart1}},~{\cite{hunualart2}} and the
references therein.

The statement therein is restricted to the case $V =\mathbb{R}^d$, although we
believe the same proof extends to more general Banach spaces.

\begin{theorem}
  Let $A \in C^{\alpha}_t C^{\beta}_{\tmop{loc}}$, $x \in C^{\gamma}_t$ with
  $\alpha + \beta \gamma > 1$ and $\delta \in (1 - \alpha, \beta \gamma)$.
  Then the following identity holds:
  
  \begin{align*}
    \int_0^T A (\mathd s, x_s) & = - \frac{1}{\Gamma (\delta) \Gamma (1 -
    \delta)} \left\{ \int_0^T \right. \frac{A_{T -} (t, x_t)}{(T - t)^{1 -
    \delta}} \mathd t\\
    & + \delta \int_0^T \int_0^t \frac{A_{T -} (t, x_t) - A_{T -} (t,
    x_r)}{(T - t)^{1 - \delta} (t - r)^{1 + \delta}} \mathd r \mathd t\\
    & + (1 - \delta) \int_0^T \int_s^T \frac{A (t, x_t) - A (s, x_t)}{(t -
    s)^{2 - \delta} s^{\delta}} \mathd t \mathd s\\
    & - \delta (1 - \delta) \int_0^T \int_0^s \int_s^T \frac{A_{s, t} (x_t) -
    A_{s, t} (x_r)}{(t - s)^{2 - \delta} (t - r)^{1 + \delta}} \mathd t \mathd
    r \mathd s
  \end{align*}
  
  where $A_{T -} (t, z) : = A (t, z) - A (T, z)$.
\end{theorem}

See Theorem~1 from~{\cite{hu2}} for a proof.

\subsection{The set of solutions to nonlinear YDEs}\label{appendixC}

We restrict here to the case $V =\mathbb{R}^d$. Inspired by a series of
results by Stampacchia, Vidossich, Browder, Gupta and others
(see~{\cite{vidossich}} and the references therein), we want to study the
topological structure of the set
\[ C (x_0, A) = \left\{ x \in C^{\alpha}_t \text{ such that } \, x_t = x_0 +
   \int_0^t A (\mathd s, x_s)  \text{ for all } \: t \in [0, T] \right\} \]
where $A \in C^{\alpha}_t C^{\beta, \lambda}_x$ with $\alpha (1 + \beta) > 1$
and $\beta + \lambda \leqslant 1$; namely, $C (x_0, A)$ is the set of
solutions to the Cauchy problem associated to $(x_0, A)$. Recall that by
Corollary~\ref{sec3.1 cor local existence} and Proposition~\ref{sec3.3
proposition bounds growth condition}, existence of global solutions is
granted, but uniqueness is not unless $A \in C^{\alpha}_t C^{1 +
\beta}_{\tmop{loc}}$; therefore $C (x_0, A)$ may not consist of a singleton.
The following result is an extension of Proposition~43
from~{\cite{galeatiharang}}, where the structure of the set $C (x_0 ; A)$ was
already addressed.

\begin{theorem}
  \label{appendixA3 main thm}Let $A \in C^{\alpha}_t C^{\beta, \lambda}_x$
  with $\alpha, \beta, \lambda$ as above, $x_0 \in \mathbb{R}^n$; then the set
  $C (x_0, A)$ is nonempty, compact and simply connected. Moreover, for any
  fixed $y \in \mathbb{R}^d$, the map
  \[ \mathbb{R}^d \times C^{\alpha}_t C^{\beta, \lambda}_x \ni (x_0, A)
     \mapsto d (y, C (x_0, A)) \in \mathbb{R} \]
  is lower semincontinuous.
\end{theorem}

Here we recall that for $y \in C^{\alpha}_t$, $K \subset C^{\alpha}_t$, the
distance of an element from a set is defined by
\[ d (y, K) = \inf_{z \in K} \| y - z \|_{\alpha} . \]
A main tool in the proof of Theorem~\ref{appendixA3 main thm} is the use of
the Browder--Gupta theorem from~{\cite{browder}}; we recite here a slight
modification due to Gorniewicz.

\begin{theorem}[Theorem~69.1, Chapter~VI from {\cite{gorniewicz}}]
  \label{appendixA3 browder gupta thm}Let $X$ be a metric space, $(E, \| \cdot
  \|)$ a Banach space and $f : X \rightarrow E$ a proper map, i.e. f is
  continuous and for every compact $K \subset E$ the set $f^{- 1} (K)$ is
  compact. Assume further that for each $\varepsilon > 0$ a proper map
  $f_{\varepsilon} : X \rightarrow E$ is given and the following two
  conditions are satisfied:
  \begin{enumerateroman}
    \item $\| f_{\varepsilon} (x) - f (x) \| \leqslant \varepsilon$ for all $x
    \in X$;
    
    \item for any $\varepsilon > 0$ and $u \in E$ such that $\| u \| \leqslant
    \varepsilon$, the equation $f_{\varepsilon} (x) = u$ has exactly one
    solution.
  \end{enumerateroman}
  Then the set $S = f^{- 1} (0)$ is $R^{\delta}$ in the sense of Aronszajn.
\end{theorem}

Recall that an $R^{\delta}$ set is the intersection of a decreasing sequence
of compact absolute retracts, thus always simply connected.

In order to prove Theorem~\ref{appendixA3 main thm} we need the a preliminary
lemma.

\begin{lemma}
  \label{appendixA3 preliminary lemma}For $A$ as above and for any $y \in
  C^{\alpha}_t$, there exists at least one solution $x \in C^{\alpha}_t$ to
  \begin{equation}
    x_t = y_t + \int_0^t A (\mathd s, x_s) \qquad \forall \: t \in [0, T] ;
    \label{appendixA3 YDE}
  \end{equation}
  moreover, there exists $C = C (\alpha, \beta, T)$ such that any solution
  satisfies the a priori estimate
  \begin{equation}
    \| x \|_{\alpha} \leqslant C \exp (C \| A \|_{\alpha, \beta, \lambda}^2 +
    \| y \|_{\alpha}^2) (1 + | y_0 |) . \label{appendixA3 a priori}
  \end{equation}
  If in addition $A \in C^{\alpha}_t C^{1 + \beta}_{\tmop{loc}}$, then the
  solution is unique.
\end{lemma}

\begin{proof}
  Set $\tilde{A} (t, x) = A (t, x) + y_t$, then $x$ is a solution
  to~{\eqref{appendixA3 YDE}} if and only if it solves
  \[ x_t = y_0 + \int_0^t \tilde{A} (\mathd s, x_s) \]
  where $\tilde{A} \in C^{\alpha}_t C^{\beta, \lambda}_x$ with $\| \tilde{A}
  \|_{\alpha, \beta, \lambda} \leqslant \| A \|_{\alpha, \beta, \lambda} + \|
  y \|_{\alpha}$. Existence and the estimate~{\eqref{appendixA3 a priori}}
  then follow from Corollary~\ref{sec3.1 cor local existence} and
  Proposition~\ref{sec3.3 proposition bounds growth condition}; $A \in
  C^{\alpha}_t C^{1 + \beta}_{\tmop{loc}}$ implies $\tilde{A} \in C^{\alpha}_t
  C^{1 + \beta}_{\tmop{loc}}$ and so uniqueness follows from
  Corollary~\ref{sec3.2 cor local uniqueness}.
\end{proof}

\begin{proof}[of Theorem~\ref{appendixA3 main thm}]
  We divide the proof in several steps.
  
  \tmtextit{Step 1: $C (x_0, A)$ nonempty, compact.} Nonemptiness follows
  immediately from Lemma~\ref{appendixA3 preliminary lemma} applied to $y
  \equiv x_0$; let $x^n$ be a sequence of elements of $C (x_0, A)$, then
  by~{\eqref{appendixA3 a priori}} they are uniformly bounded in
  $C^{\alpha}_t$ and so by Ascoli--Arzel{\`a} we can extract a (not
  relabelled) subsequence $x^n \rightarrow x$ in $C^{\alpha - \varepsilon}_t$
  for all $\varepsilon > 0$, for some $x \in C^{\alpha}_t$. Choosing
  $\varepsilon > 0$ sufficiently small such that $\alpha + \beta (\alpha -
  \varepsilon) > 1$, by Theorem~\ref{sec2 thm definition young integral} the
  map $z_{\cdot} \mapsto \int_0^{\cdot} A (\mathd s, z_s)$ is continuous from
  $C^{\alpha - \varepsilon}_t$ to $C^{\alpha}_t$, therefore
  \[ x_{\cdot}^n = x_0 + \int_0^{\cdot} A (\mathd s, x^n_s) \rightarrow x_0 +
     \int_0^{\cdot} A (\mathd s, x_s) = x_{\cdot}  \text{ in } C^{\alpha}_t,
  \]
  which shows compactness.
  
  \tmtextit{Step 2: $C (x_0, A)$ connected.} Given $A \in C^{\alpha}_t
  C^{\beta, \lambda}_x$, consider a sequence $A^{\varepsilon} \in C^{\alpha}_t
  C^{1 + \beta, \lambda}_x$ such that
  \[ \| A^{\varepsilon} \|_{\alpha, \beta, \lambda} \leqslant 2 \| A
     \|_{\alpha, \beta, \lambda}, \quad A^{\varepsilon} \rightarrow A \text{
     in } C^{\alpha}_t C^{\beta}_{\tmop{loc}} \quad \text{as } \varepsilon
     \rightarrow 0 ; \]
  this is always possible, for instance by taking $A^{\varepsilon} =
  \rho^{\varepsilon} \ast A$, $\{ \rho^{\varepsilon} \}_{\varepsilon > 0}$
  being a family of standard spatial mollifiers. For $x_0 \in \mathbb{R}^d$
  fixed, take $R > 0$ big enough such that
  \[ C \exp (C \| A^{\varepsilon} \|_{\alpha, \beta, \lambda}^2 + \| x_0 + y
     \|_{\alpha}^2) (1 + | y_0 + x_0 |) \leqslant R \quad \forall \,
     \varepsilon \in (0, 1), \, y \in C^{\alpha}_t \, \text{ s.t. } \| y
     \|_{\alpha} \leqslant 1, \]
  where $C$ is the constant appearing in~{\eqref{appendixA3 a priori}}; this
  is always possible due to the uniform bound on $\| A^{\varepsilon}
  \|_{\alpha, \beta, \lambda}$. Define the metric space $E$ to be
  \[ E = \left\{ z \in C^{\alpha}_t : \, \| z \|_{\alpha} \leqslant R
     \right\}, \quad d_E (z^1, z^2) = \| z^1 - z^2 \|_{\alpha} ; \]
  and define maps $f, f_{\varepsilon} : E \rightarrow C^{\alpha}_t$ by
  \[ f (x) = x_{\cdot} - x_0 - \int_0^{\cdot} A (\mathd s, x_s), \quad
     f_{\varepsilon} (x) = x_{\cdot} - x_0 - \int_0^{\cdot} A^{\varepsilon}
     (\mathd s, x_s) . \]
  By Theorem~\ref{sec2 thm definition young integral}, they are continuous
  maps from $E$ to $C^{\alpha}_t$; by reasoning exactly as in Step~1 it is
  easy to check thar they are proper. Observe that an element $x \in E$
  satisfies $f (x) = y$ if and only if it satisfies
  \[ x \in C^{\alpha}_t, \quad x_t = x_0 + y_t + \int_0^t A (\mathd s, x_s)
     \quad \forall \, t \in [0, T], \quad \| x \|_{\alpha} \leqslant R, \]
  similarly for $f_{\varepsilon}$; moreover the bound $\| x \|_{\alpha}
  \leqslant R$ is trivially satisfied for all $y$ such that $\| y \|_{\alpha}
  \leqslant 1$, by our choice of $R$ and Lemma~\ref{appendixA3 preliminary
  lemma}. It follows that, for any such $y$, $f_{\varepsilon} (x) = y$ has
  exactly one solution $x \in E$. In order to apply Theorem~\ref{appendixA3
  browder gupta thm} and get the conclusion, it remains to show that
  $f_{\varepsilon} \rightarrow f$ uniformly in $E$; but by Theorem~\ref{sec2
  thm definition young integral} it holds
  
  \begin{align*}
    \| f (z) - f_{\varepsilon} (z) \|_{\alpha} & = \, \left\| \int_0^{\cdot} A
    (\mathd s, x_s) - \int_0^{\cdot} A^{\varepsilon} (\mathd s, x_s)
    \right\|_{\alpha}\\
    & \lesssim \, \| A - A^{\varepsilon} \|_{\alpha, \beta, R}  (1 + \| z
    \|_{\alpha})\\
    & \lesssim \, \| A - A^{\varepsilon} \|_{\alpha, \beta, R} (1 + R)
    \rightarrow 0 \text{ as } \varepsilon \rightarrow 0
  \end{align*}
  
  and the can conclude that $f^{- 1} (0) = C (x_0, A)$ is simply connected in
  $E$, thus also in $C^{\alpha}_t$.
  
  \tmtextit{Step 3: lower semicontinuity.} Consider now a sequence $(x_0^n,
  A^n) \rightarrow (x_0, A)$ in $\mathbb{R}^d \times C^{\alpha}_t C^{\beta,
  \lambda}_x$, we need to show that for any fixed $y \in C^{\alpha}_t$ it
  holds
  \[ d (y, C (x_0, A)) \leqslant \liminf_{n \rightarrow \infty} d (y, C
     (x_0^n, A^n)) . \]
  Since by Step~1 the set $C (x^n_0, A^n)$ is compact, it is always possible
  to find $x^n \in C (x_0^n, A^n)$ such that
  \[ \| y - x_0^n \| = (y, C (x_0^n, A^n)) ; \]
  we can assume wlog that $\lim d (y, C (x_0^n, A^n))$ exists, since otherwise
  we can extract a subsequence realizing the liminf. Since $(x_0^n, A^n)$ is
  convergent, it is also bounded in $\mathbb{R}^d \times C^{\alpha}_t
  C^{\beta, \lambda}_x$, which implies by estimate~{\eqref{appendixA3 a
  priori}} that the sequence $\{ x^n \}_n$ is bounded in $C^{\alpha}_t$. It is
  not difficult to see, invoking Ascoli--Arzel{\`a} and going through the same
  reasoning as in Step~1, that we can extract a (not relabelled) subsequence
  such that $x^n \rightarrow x$ in $C^{\alpha}_t$ where $x \in C (x_0, A)$. As
  a consequence
  \[ d (y, C (x_0, A)) \leqslant \| y - x \|_{\alpha} = \lim_{n \rightarrow
     \infty} \| y - x^n \|_{\alpha} = \liminf_{n \rightarrow \infty} d (y, C
     (x^n_0, A^n)) \]
  which gives the conclusion.
\end{proof}

Theorem~\ref{appendixA3 main thm} has relevant consequence when considering $C
(x_0, A)$ as a multivalued map; we refer the reader to~{\cite{castaing}} for
an overview on the topic.

Recall that, given a complete metric space $(E, d)$, the space
\[ K (E) = \left\{ K \subset E \, : \, K \text{ is compact} \right\} \]
is itself a complete metric space with the Hausdorff metric
\[ d_H (K_1, K_2) = \max \{ \sup_{a \in K_1} d (a, K_2), \sup_{b \in K_2} d
   (b, K_1) \} \]
and that moreover
\[ d_H (K_1, K_2) = \sup_{a \in E} | d (a, K_1) - d (a, K_2) | = \max_{a \in
   K_1 \cup K_2} | d (a, K_1) - d (a, K_2) | . \]
If we endow the space $(K (E), d_H)$ with its Borel $\sigma$-algebra, then
it's possible to show that a map $F : (\Omega, A) \rightarrow (K (E), d_H)$ is
measurable if and only if, for all $a \in E$, the map
\[ \Omega \ni \omega \mapsto d (a, F (\omega)) \in \mathbb{R} \]
is measurable.

\begin{corollary}
  The map from $\mathbb{R}^d \times C^{\alpha}_t C^{\beta, \lambda}_x$ to $K
  (C^{\alpha}_t)$ given by $(x_0, A) \mapsto C (x_0, A)$ is a measurable
  multifunction.
\end{corollary}

\begin{proof}
  It follows immediately from Theorem~\ref{appendixA3 main thm} and the fact
  that lower semicontinuous maps are measurable.
\end{proof}

\begin{remark}
  For simplicity we have only treated the case $V =\mathbb{R}^d$, but it's
  clear that Theorem~\ref{appendixA3 main thm} admits several extensions; for
  instance it can be readapted to the case of equations of the
  form~{\eqref{sec3.5 mixed YDE}} with $A \in C^{\alpha}_t C^{\beta,
  \lambda}_x$ and $F$ continuous of linear growth. In alternative, one can
  consider a general Banach space $V$ and $A \in C^{\alpha}_t C^{\beta,
  \lambda}_{V, W}$ with $W$ compactly embedded in $V$; this is enough to grant
  global existence by Corollary~\ref{sec3.1 cor local existence} and the usual
  a priori estimates.
\end{remark}

\bibliography{biblio}{}
\bibliographystyle{plain}

\end{document}